\documentclass{amsart}
\usepackage{graphicx}
\usepackage{amssymb}
\usepackage{amsfonts}
\usepackage{hyperref}
\usepackage{mathrsfs}
\usepackage[margin=3cm]{geometry}

\swapnumbers
\newtheorem{theorem}{Theorem}[section]
\newtheorem{lemma}[theorem]{Lemma}
\newtheorem{corollary}[theorem]{Corollary}

\newtheorem{proposition}[theorem]{Proposition}
\theoremstyle{definition}

\newtheorem{assumption}[theorem]{Assumption}
\newtheorem{remark}[theorem]{Remark}

\numberwithin{equation}{section}
\theoremstyle{plain}

\numberwithin{equation}{section} 
\numberwithin{figure}{section} 
\theoremstyle{plain}
\theoremstyle{plain}
\theoremstyle{remark}
\newtheorem*{acknowledgement*}{Acknowledgement}
\theoremstyle{example}

\newcommand{\cA}{{\mathcal A}}
\newcommand{\cB}{{\mathcal B}}
\newcommand{\cC}{{\mathcal C}}
\newcommand{\cD}{{\mathcal D}}
\newcommand{\cE}{{\mathcal E}}
\newcommand{\cF}{{\mathcal F}}
\newcommand{\cG}{{\mathcal G}}
\newcommand{\cH}{{\mathcal H}}

\newcommand{\cL}{{\mathcal L}}
\newcommand{\cM}{{\mathcal M}}

\newcommand{\cP}{{\mathcal P}}

\newcommand{\cS}{{\mathcal S}}

\newcommand{\cX}{{\mathcal X}}

\newcommand{\te}{{\theta}}

\newcommand{\Om}{{\Omega}}
\newcommand{\om}{{\omega}}
\newcommand{\ve}{{\varepsilon}}
\newcommand{\del}{{\delta}}
\newcommand{\Del}{{\Delta}}
\newcommand{\gam}{{\gamma}}
\newcommand{\Gam}{{\Gamma}}

\newcommand{\sig}{{\sigma}}
\newcommand{\al}{{\alpha}}
\newcommand{\be}{{\beta}}
\newcommand{\ka}{{\kappa}}
\newcommand{\la}{{\lambda}}

\newcommand{\bbC}{{\mathbb C}}
\newcommand{\bbE}{{\mathbb E}}

\newcommand{\bbN}{{\mathbb N}}
\newcommand{\bbP}{{\mathbb P}}
\newcommand{\bbR}{{\mathbb R}}

\newcommand{\bbZ}{{\mathbb Z}}
\newcommand{\bbI}{{\mathbb I}}

\begin{document}
\title[]{A local limit theorem for number of multiple recurrences generated by some mixing processes with applications to Young towers}
 \vskip 0.1cm
 \author{Yeor Hafouta \\
\vskip 0.1cm
Department of Mathematics\\
The Ohio State University\\
}%
\email{yeor.hafouta@mail.huji.ac.il, hafuta.1@osu.edu}%

\thanks{ }
\dedicatory{  }
 \date{\today}

\maketitle
\markboth{Y. Hafouta}{LLT for nonconventional sums}
\renewcommand{\theequation}{\arabic{section}.\arabic{equation}}
\pagenumbering{arabic}

\begin{abstract}\noindent
We prove a local central limit theorem for ``nonconventional" sums 
generated by some classes of sufficiently fast mixing sequences.
\end{abstract}

\section{Introduction}\label{sec1}
Since the ergodic theory proof of Szemer\' edi's theorem on arithmetic progressions due to Furstenberg \cite{Fur}, limits of expressions having the form 
$S_N/N=1/N\sum_{n=1}^NT^{q_1(n)}f_1\cdots T^{q_\ell(n)}f_\ell$ have been extensively studied in literature,
where $T$ is a measure preserving transformation, $\ell$ is a positive integer, $f_i$'s are bounded measurable
functions and $q_i$'s are linear or polynomial functions taking  integer
values on the set of integers. 
For the proof of  Szemer\' edi's theorem, we only need to consider the case when 
all the $f_i$'s are the indicator of the same measurable set $A$. In this case $S_N|A$ counts the number of multiple recurrences to the set $A$.
Most of the results in this direction are $L^2$-convergences of $S_N/N$ (see, for instance \cite{Berg}), expect for the results in \cite{Bourg} in which an almost sure convergence was established in the case when $\ell=2$ and $q_1$ and $q_2$ are linear. Almost sure convergence was obtained when $\ell>2$ only in particular cases, see for instance \cite{Ki1.1}, \cite{HSY} and references therein.

From the probabilistic point of view  the orbits of the underlying dynamical system $T$ are viewed as random variables $U_n=T^nU_0$, where $U_0$ is distributed according to the invariant measure $\mu$. Thus,
ergodic theorems can be viewed as laws of large numbers and once they are derived it is natural to inquire about other classical limit theorems of probability.
Partially motivated by that, central limit theorems and  large deviations principles for ``nonconventional sums" (the term comes from \cite{Fur}) of the  form 
\begin{equation}\label{Noncon}
S_N^{\{q_j\}}G=\sum_{n=1}^NG(X_{q_1(n)},X_{q_2(n)},...,X_{q_\ell(n)})
\end{equation}
were obtained by Kifer \cite{Ki2} and Kifer and Varadhan \cite{KV1,KV2}. 
 Here $G$ is a real-valued function satisfying some regularity and growth conditions,
 and $\{X_n\}$ is a sequence of random vectors satisfying some mixing, stationarity and moment conditions, which are satisfied for wide classes of  Markov chains and when $X_n$
 has the form $X_n=f(T^nU_0)=f(U_n)$, where $f$ is a sufficiently regular vector-valued function and $T$ is a sufficiently chaotic dynamical system such as a topologically mixing subshift of finite type or an Anosov map and $U_0$ is distributed according to an equilibrium state (i.e. a Gibbs measure), as well as when $T$ is a Young tower with sufficiently fast decaying tails.  Since then a variety  of nonconventional limit theorems  were obtained:  a moderate deviations principle and exponential concentration inequalities were derived in   \cite{HafMD}, stable laws were proven in \cite{KV3} and Berry-Esseen type estimates and other results were derived  in \cite{HK2,book} (see also references therein). 
 
The local central limit theorem (LCLT) concerns the asymptotic behavior of expectations of the form 
 $\bbE[g(S_N)]$, where $g$ is an indicator of a bounded interval or a continuous function with compact support, and it has origins in the classical De Moivre-Laplace theorem.
In the ``conventional" case when $\ell=1$, $q_1(n)=n$ and $S_N=\sum_{n=1}^N G(X_n)$
 the LCLT for Markov chains $\{X_n\}$   was obtained by Nagaev \cite{Neg2} (in the countable state case),  for expanding interval maps $T$ by J.~Rousseau-Egele \cite{RE83} (where $X_n=T^nX_0$), for subshifts of finite type and Anosov maps  it was derived by Guivarach and Hardy \cite{GH} and for Gibbs-Markov maps by Aaronson and Denker \cite{Jon}. These papers  used what these days is commonly referred to as the ``Nagaev-Guivarch method" (spectral gap). In fact, all of these LCLT's follow from a general theory of quasi-compact Markov operators,
and we refer to \cite{HH} for an abstract description of this method.

In \cite{HK1} we proved an LCLT for nonconventional sums of the form
\begin{equation}\label{Noncon0}
\sum_{n=1}^NG(\xi_{n},\xi_{2n},...,\xi_{\ell n})
\end{equation}
for some classes of stationary $\psi$-mixing Markov chains $\{\xi_n\}$ satisfying a two sided version of the Doeblin condition, whose state space is a compact metric space, and measurable functions $G$ satisfying some moment conditions.
 This LCLT was extended in \cite[Ch.2]{book} to Markov chains whose transition operator is the dual of the Koopman operator (with respect to a Gibbs measure) corresponding to certain types of distance expanding maps $T$.
In this setup the function $G$ was assumed to be bounded and H\"older continuous.
The latter yields the LCLT for 
sums of the form 
\begin{equation}\label{Noncon10.1}
\sum_{n=1}^NG(f^{-n}X_0,f^{-2n}X_0,...,f^{-\ell n}X_0)
\end{equation}
where $f$ is   an Anosov map, and $X_0$ is distributed according to an underlying Gibbs-measure.
In the ``conventional" case the sums $\sum_{n=0}^{N-1}G(f^{n}X_0)$ and $\sum_{n=0}^{N-1}G(f^{-n}X_0)$ have the same distribution, but this is no longer true for nonconventional sums.
 Still, since $f^{-1}$ is also Anosov,  by replacing $f$ with $f^{-1}$ in \eqref{Noncon10.1} the LCLT for 
 \begin{equation}\label{Noncon10.11}
\sum_{n=1}^NG(f^{n}X_0,f^{2n}X_0,...,f^{\ell n}X_0)
\end{equation}
follows.

The goal of the current paper is twofold. First, we extend the LCLT for the sums \eqref{Noncon0} for Markov chains $\{\xi_n\}$  whose transition operator is  the dual of the Koopman operator corresponding to a Young tower with exponential tails, where here $G$ is a bounded H\"older continuous function. This has applications to the LCLT for sums of the form \eqref{Noncon10.1}, where now $f$ is 
a partially hyperbolic diffeomorphism which can be modeled by a Young tower \cite{Y1}. The second goal is to obtain the LCLT for sums of the form
\begin{equation}\label{Noncon1}
S_N^{\{q_j\}}G=\sum_{n=1}^NG(X_{q_1(n)},X_{q_2(n)},...,X_{q_\ell(n)})
\end{equation}
for indexes $q_1(n),...,q_\ell(n)$ exhibiting some nonlinear growth. More precisely, we will assume that  $q_1,...,q_k$ are linear for some $k<\ell$, and that $q_j$ grows faster than linearly and faster than $q_{j-1}$ for $k<j\leq\ell$, in a certain quantitative  way (the case $k=0$ corresponds to having no linear functions). 
 For instance, the case when all $q_i$'s are polynomials so that  $q_1,...,q_k$ are linear for some $k<\ell$, $\deg q_{k+1}>1$ and $\deg q_{i}<\deg q_{i+1},i>k$ will be a particular case of our assumptions. 
It turns out that for such indexes the LCLT holds true for a wide class of sufficiently fast (approximately) mixing sequences $\{X_n\}$ taking values in some metric space, which are not necessarily generated by a Markov operator or a chaotic dynamical system,  and bounded H\"older continuous functions $G$. We note that in this setup the H\"older continuity is only needed when  $\{X_n\}$ is not strongly mixing in the probabilistic sense, and it can only be approximated sufficiently well by strongly  mixing sequences. For instance, when $X_n=(\xi_n,\xi_{n+1},...,\xi_{n+m})$ for some $m$ and a geometrically ergodic Markov chain $\{\xi_j\}$ then our results hold true for bounded functions $G$ which are not necessarily continuous. When $X_n=(\xi_k)_{k\geq n}$ then we need $G$ to be H\"older continuous in order to approximate $G(X_{q_1(n)},X_{q_2(n)},...,X_{q_\ell(n)})$ by expressions which depend only on $\big(\xi_{q_1(n)+s},\xi_{q_2(n)+s},...,\xi_{q_\ell(n)+s}\big)_{s=0}^m$, with an error term depending on $m$. 
To the best of our knowledge, even the case when $\ell=1$ and $q_1(n)$ grows faster than linearly  was not considered in literature (the LCLT when $\ell=1$ and $q_1(n)=n$ requires more than some mixing conditions).


In the ``conventional setup" (when $\ell=1$ and $q_1(n)=n$), for certain classes of Markov chains $\{\xi_n\}$ and chaotic dynamical systems $T$, 
 the LCLT  established in \cite{Neg2, RE83, GH, Jon} relies on the following idea. First, there are operators $\cL_{it}$ so that
\begin{equation}\label{Govern}
\bbE(e^{it S_N G})=\mu(\cL_{it}^N \textbf{1})
\end{equation} 
where $\mu$ is the underlying  stationary distribution, $\textbf{1}$ is the function taking the constant value $1$ and $S_N G=\sum_{n=0}^{n-1}G\circ T^n$ or $S_NG=\sum_{n=1}^N G(\xi_n)$ are the usual Birkhoff sums. 
 The operators $\cL_{it}$ have the form $\cL_{it}(g)=\cL_0(ge^{it G})$, where $\cL_0$ is either the Markov operator defining the Markov chain, or the dual of the Koopman (transfer) operator corresponding to the map $T$ with respect to the stationary distribution. The operators $\cL_{it}$  are  smooth in $t$.
 Moreover, they are quasi-compact when their spectral radius is $1$. In the transfer operator case the right hand side of \eqref{Govern} is the characteristic function of the partial sums $\sum_{n=1}^{N}G(\tilde\xi_n)$ of the Markov chain $\{\tilde\xi_n\}$ whose transition operator is $\cL_0$ and its stationary distribution is $\mu$. Thus, the LCLT in both setups above can be obtained by 
using the spectral theory of quasi-compact Markov operators (and we refer again to \cite{HH}).

Quasi-compactness properties of transfer operators also yield the LCLT for 
uniform Young towers (in the terminology of \cite{ChazGO}) with exponential tails, since then the appropriate perturbations of the dual of the Koopman operator are quasi-compact, though this was not explicitly formulated in literature, probably because in the ``conventional" case the LCLT holds true for non-uniform towers with sub-exponential tails. This was proven by  S. Gou\"ezel \cite{GO} using  operator renewal theory (i.e. by inducing), which replaces the quasi-compactness, but it is less relevant to 
our paper since when $q_i(n)=in$ our methods rely on quasi-compactness (and exponential tails), as explained in the following paragraphs.


In the arithmetic progression case $q_i(n)=in$, for sufficiently fast mixing Markov chains $\{\xi_n\}$ the main obstacle in the proof of the LCLT in the nonconventional setup is that a family of operators which ``govern" the characteristic functions in the sense of \eqref{Govern} does not seem to exist, in view of the non-stationarity and long range dependence of the summands $G(\xi_n,\xi_{2n},...,\xi_{\ell n})$.
Still, the idea behind the proofs of the LCLT's from \cite{HK1} and \cite[Ch.2]{book} is based on
``spectral" properties, but now we have to consider random operators instead of a single one.
 Roughly speaking, we showed that when $X_n=\xi_n$ is one of the chains considered in \cite{HK1} and \cite[Ch.2]{book}, then there exist a mixing probability preserving system $(\Om,\cF,P,\te)$ and a family of random  operators $\cL_{it}^{\om},\,\om\in\Om$ so that, on compact sets of $t$'s we have
\begin{equation}\label{Red}
|\bbE(e^{it S_N^{\{q_j\}}G})|\leq \int|\mu(\cL_{it}^{\te^{a_\ell N}\om}\circ \dots\circ\cL_{it}^{\te\om}\circ\cL_{it}^{\om}\textbf{1})|dP(\om)+o(N^{-1/2})
\end{equation}
where $a_\ell$ is some constant and $\mu(g)=\int gd\mu$ for any integrable function $g$. This is achieved by a conditioning argument, which in this paper is referred to as ``\textit{the conditioning step}". 
In general, given any sequence of random variables $\{Z_N\}$ so that $N^{-1/2}(Z_N-mN)$, $m\in\bbR$ satisfies  the CLT, the LCLT follows from certain types of decay rates of the characteristic functions of $Z_N$ (see Theorem \ref{General LLT }). 
Using \eqref{Red}, the type of control  over the characteristic functions needed  to obtain the LCLT for $Z_N=S_N^{\{q_j\}}G$ was achieved by studying the products of random operators appearing in \eqref{Red}, which in some sense reduces the main problem to the random dynamics setup.
For one sided  topologically mixing subshifts of finite type and other expanding maps, these random operators were studied using a complex version of the Hilbert projective metric due to H.H. Rugh \cite{Rug}, corresponding to the canonical complexification of the classical cones of logarithmically H\"older continuous functions.

For Markov chains generated by uniform Young towers (as described above), the  strategy of the proof of the LCLT for $S_N^{\{q_j\}}G$ when $q_i(n)=in$ is as follows. First, using the semi-conjugacy with a Bernoulli shift  established in \cite{Kor1},  we will also have a certain conditioning step, which yields upper bounds of the form \eqref{Red} with $\cL^{\om}_{it}$ now being perturbations of $P^\ell=\cL_0^{\om}$, where $P$ is the dual of the Koopman operator $g\to g\circ F$ corresponding to the tower map $F$.
Already in this conditioning step we need the tower to have exponential tails, since the semi-conjugacy was only proven in this case.
For $t$'s which are not close to $0$, we will study the asymptotic behavior of the latter product by using quasi-compactness of certain related deterministic transfer operators determined by a periodic orbit of the tower. 

For $t$'s close to $0$ we will study the asymptotic behavior of the product of the random operators from \eqref{Red} by showing that these operators contract certain type of complex cones. In contrast with the expanding case, these cones do not consist of (complexficiations of) logarithmically H\"older continuous functions, and instead we will obtain the desired estimates on the norm of the product of the random operators using the canonical complexification of the cones introduced in \cite{Viv2}. This also requires exponential tails, but the results are established for non-uniform towers, so the need in uniform ones only arises when dealing with $t$'s far away from $0$.
In terms of general techniques, showing that these cones satisfy the conditions needed for Rugh's theory \cite{Rug} (see also \cite{Dub1,Dub2}) to be effective, and that the  random complex operators $\cL_{it}^{\om}$ contract these cones is the main novelty of this manuscript. In order not to overload the paper we present the results concerning complex cones in a separate section (Section \ref{sec tower}). 

There are three reasons we need here the tower to have exponential tails. First, it is needed to establish 
 \eqref{Red}. Second, the exponential tails are needed to obtain the desired projective contraction of the random complex operators $\cL_{it}^{\om}$ described above. 
   For deterministic operators one can use the operator renewal theory from \cite{GO} instead of complex cones, but it is still not clear how to adapt this theory  to study the $\om$-wise asymptotic behavior of the products of the random operators on the right hand side of \eqref{Red}.  
   The third reason is that the estimates we obtain on 
 the right hand side of \eqref{Red} for $t$'s bounded away from $0$ rely on quasi-compactness of certain associated deterministic system. Also in this case it is less clear how to adapt the operator renewal theory, or to make a reduction to a deterministic system for such $t$'s without some kind of quasi-compactness assumption.


When some of the $q_j$'s grow faster than linearly, we will also have a conditioning step, but of a different form. This step does not require the underlying sequence $\{X_n\}$ to be a Markov chain, and instead we only need it to satisfy certain mixing and approximation conditions. Because of the nonlinear growth of $q_\ell$, this conditioning argument  yields  a different upper bound of the form 
\begin{equation}\label{Re1}
|\bbE(e^{it S_N^{\{q_j\}}G})|\leq \bbE\left[\prod_{n=[aN]+1}^{N}\zeta(Y_n,t)\right]+o(N^{-1/2})
\end{equation}
on compact sets of $t$'s, where $a\in(0,1)$ is some constant, $\{Y_n\}$
 is some sufficiently fast mixing  process satisfying some stationarity conditions, and $\zeta(y,t)$ are certain functions taking values in $[0,1]$.
 In this case the general estimates needed for the LCLT (see again Theorem \ref{General LLT }) do not require to study compositions of random operators, which is the reason that general mixing conditions are sufficient for the LCLT. Roughly speaking, we will use the mixing properties of $\{Y_n\}$ to replace $\bbE\left[\prod_{n=[aN]+1}^{N}\zeta(Y_n,t)\right]$ with $\prod_{n=[aN]+1}^{N}\bbE\left[\zeta(Y_n,t)\right]=(\zeta(t))^{N-[aN]}$, $\zeta(t)=\bbE[\zeta(Y_n,t)]$. Thus, the function $\zeta(t)$ controls the rate of decay of the characteristic functions.


\section{A nonconventional LLT with nonlinear indexes and some classes of mixing processes}\label{sec2}
Let $(\Om,\cF,\bbP)$ be a probability space and let $\cF_{n,m}\subset\cF$, $n,m\in\bbZ$ be a family of $\sig$-algebras so that $\cF_{n,m}\subset\cF_{n_1,m_1}$ if $n_1\leq n\leq m\leq m_1$. We will measure the dependence between these $\sig$-algebras by the classical $\phi$-mixing coefficients $\phi(n)$ given by 
\begin{equation}\label{phi n}
\phi(n)=\sup\left\{|\bbP(B|A)-\bbP(B)|:\,k\in\bbZ,\,A\in\cF_{-\infty,k},\,B\in\cF_{k+n,\infty},\, \bbP(A)>0\right\}
\end{equation}
where $\cF_{l,\infty}$ is the union  of $\cF_{l,n},\,n\geq l$ and $\cF_{-\infty,k}$ is the union of $\cF_{j,k},j\leq k$.
Let $(\cX,d)$ be a  metric space and let $X_n:\Om\to\cX$ be a sequence of measureable functions (i.e. random variables).
We do not require $\{X_n\}$ to be strongly stationary, and instead we only assume that for each $n\leq m$ the distribution of the pair $(X_n,X_m)$ depends only on $m-n$. In particular, all the $X_n$'s are identically distributed.

We also do not assume that $X_n$ is measurable with respect to $\cF_{n,n}$, and instead we will impose restrictions on the approximation rate.
The $r$-th approximation rate of order $p\geq1$ is given by 
\begin{equation}\label{beta}
\beta_p(r)=\sup_{n}\inf_{Z_{n,r}}\|d(X_n,Z_{n,r})\|_{L^p}
\end{equation}
where the infimum is taken over all the $\cX$-valued  and $\cF_{n-r,n+r}$-measureable random variables $Z_{n,r}$. 
The results described in this section are obtained under assumptions of the following form.
\begin{assumption}\label{AssMix}
There are constants $c>0$  and $\te_1,\te_2>0$ so that for all $n\in\bbN$,
\begin{equation}\label{MixA}
\phi(n)\leq cn^{-\te_1}\,\,\text{ and }\,\,\beta_{2}(n)\leq cn^{-\te_2}.
\end{equation}
\end{assumption}
We refer the readers to \cite[Section 8]{Haf-Arrays} for some examples of processes satisfying Assumption \ref{AssMix}. These examples include the case  when $X_n=T^nX_0$,
where $T$ is a two (or one) sided topologically mixing subshift of finite type \cite{Bow}, $X_0$ is distributed
according to a Gibbs measure and $\cF_{n,m}$ is the $\sig$-algebra generated by the cylinders corresponding to the coordinates at places $n,...,m$. In this case we have $\max(\phi(n),\beta_\infty(n))\leq C\del^n$ for some $C>0$ and $\del\in(0,1)$.
Another example are functionals  $X_n=\sum_{j}a_jf(\xi_{n+j})$ of geometrically ergodic Markov chains $\{\xi_n\}$, where $f$ is a measurable function taking values in some Banach space $(B,|\cdot|)$ so that $\|f\|_2:=\sup_n\||f(\xi_n)|\|_{L_2}<\infty$ and $\sum_{j}a_j$ is a converging series.
If we take $\cF_{n,m}=\sig\{\xi_n,...,\xi_m\}$ then $\phi(n)$ converges exponentially fast to $0$ and  
$\beta_2(r)\leq\|f\|_2\sum_{|j|>r}|a_j|$. Thus, Assumption \ref{MixA} will hold if $\sum_{|j|>r}|a_j|=O(r^{-\te_2})$.
In fact, we can take any  stationary sequence $\{\xi_n\}$ so that with  $\cF_{n,m}=\sig\{\xi_n,...,\xi_m\}$ we have $\phi(n)=O(n^{-\te_1})$, and define $X_n$ similarly.
We refer to Section \ref{BS} for an application to functions of Bernoulli shifts (which has applications to Young towers with exponential tails \cite[Section 5.3]{Kor1}).

Next, let $\ell$ be a positive integer and let $q_1(n),q_2(n),...,q_\ell(n)$ be integer-valued non-negative sequences. 
We assume here that there is an integer $0\leq k<\ell$ so that for all $j>k$  the function $q_j$ grows faster than linearly in the sense that 
there exists $\al\in(0,1)$ so that
for all $n\in\bbN$ large enough 
\begin{equation}\label{NonLinGr} 
q_{j}(n+1)-q_{j}(n)\geq n^{\al},\,\, j=k+1,k+2,...,\ell.
\end{equation}
Furthermore, $q_{i+1}$ grows faster than $q_i$ for $i>k$ in the sense that 
\begin{equation}\label{NonLinGr0}
\forall\ve>0\,\text{ we have }\,\,\,\lim_{n\to\infty}\left(q_{i+1}(\ve n)-q_{i}(n)\right)=\infty,\,\,\,i=k+1,...,\ell-1.
\end{equation}
When $k=0$ then the above conditions hold for functions with nonlinear growth.
 However, we can also consider the case when some of the functions are linear polynomials. This corresponds to $k>0$ and in this case, for the sake of simplicity we assume that 
\begin{equation}\label{kLin}
q_j(n)=jn,\,\,\forall n\in\bbN,\,\,j=1,2,...,k.
\end{equation}

Set $\cX^\ell=\cX\times\cX\times\dots\times\cX$ ($\ell$ times) and let $\ka\in(0,1]$.
 Let $G:\cX^\ell\to\bbR$ be a bounded function so that with some $K>0$ for all $(x_1,...,x_\ell), (y_1,...,y_\ell)\in\cX^{\ell}$ we have
\begin{equation}\label{G Hold}
 |G(x_1,...,x_\ell)-G(y_1,...,y_\ell)|\leq K\sum_{j=1}^{\ell}\left(d(x_i,y_i)\right)^\ka.
\end{equation}
For each $N$ set 
\[
S_N^{\{q_j\}}G=\sum_{n=1}^N G(X_{q_1(n)},X_{q_2(n)},...,X_{q_\ell(n)}).
\]
Set also
\begin{equation}\label{bar G def1}
\bar G=\int G(x_1,x_2,...,x_\ell)d\mu(x_1)d\mu(x_2)\cdots d\mu(x_\ell)
\end{equation}
where $\mu$ is the common distribution of the $X_n$'s. 
The main result in this section is a local central limit theorem (LCLT) for the sequence of random variables $Z_N=S_N^{\{q_i\}}G$.

\subsubsection{The CLT}
Before proving the LCLT we need to discuss the central limit theorem (CLT). As mentioned in Section \ref{sec1}, the CLT for $N^{-1/2}\big(S_N^{\{q_j\}}G-N\bar G\big)$ does not follow from existing results since  $X_n$ do not take values at $\bbR^s$  for some $s$. 

\begin{theorem}[CLT]\label{CLT}
Suppose that Assumption \ref{AssMix} holds true with $\te_1>4$ and $\te_2>\frac{2}{\ka}$, where $\ka$ is the exponent from the right hand side of \eqref{G Hold}.
  Then  the limit 
\[
D^2=\lim_{N\to\infty}\frac 1N\bbE\left[\big(S_N^{\{q_j\}}G-\bar G N\big)^2\right]
\]
exists.
Moreover, the sequence $N^{-\frac12}\big(S_N^{\{q_j\}}G-\bar G N\big)$ converges in distribution as $N\to\infty$ towards a centered normal random variable with variance $D^2$. Furthermore, if $X_n$ is $\cF_{n-r,n+r}$-measurable for some $r\in\bbN$ and all $n\in\bbN$, then the above holds true for any bounded function $G$ (i.e. without \eqref{G Hold}).
\end{theorem}
While the proof of the existence of $D^2$ proceeds exactly as in \cite{KV1}, it is less clear to us how to adapt the martingale approximation techniques from \cite{KV1} to the situation when $X_n$ are not vector-valued. The point is that the main estimates needed for the martingale approximation to work depend on the dimension (see \cite[Theorem 3.4]{KV1}). Therefore, similarly to \cite[Ch.1]{book}, the proof of the CLT is based on Stein's method and (strong) dependency graph.

Next, when $D^2=0$ the CLT is degenerate, and it is interesting to have a characterization for the  positivity of $D^2$. 
In order to present such characterizations,
we first need the following notations.
Let $\mu$ be the distribution of $X_n$, and for every $k<j<\ell$ let us set 
\[
G_j(x_1,...,x_j)=\int G(x_1,...,x_j,z)d\mu^{\ell-j}(z)-\int G(x_1,...,x_{j-1},z)d\mu^{\ell-j+1}(z)
\]
while for $j=\ell$,
\begin{equation}\label{G ell}
G_\ell(x_1,...,x_\ell)=G(x_1,...,x_\ell)-\int G(x_1,...,x_{\ell-1},z)d\mu(z).
\end{equation}
When $k>0$  let  us also consider the function $G_k$ given by
\begin{equation}\label{G k}
G_k(x_1,...,x_k)=\int G(x_1,...,x_k,z)d\mu^{\ell-k}(z)-\bar G.
\end{equation}
Here $\mu^s=\mu\times\mu\times\cdots\times\mu\,$ ($s$-times) for any $s$.
\begin{theorem}[Positivity of the asymptotic variance]\label{Dthm}
Under the conditions of Theorem \ref{CLT} we have the following.
When all the functions $q_j$ grow faster than linearly (i.e. $k=0$) then $D^2=0$ if and only if the function $G$ is constant $\mu^\ell$-a.s. 
When some of the functions are linear (i.e. $k>0$) then  $D^2=0$ if and only if $G_j$ vanishes for every $j>k$ ($\mu^\ell$-a.s.), and $G_k$ is an $L^2$-coboundary with respect to the map $F\times F^2\times\cdots\times F^k$. 

If $X_n$ is $\cF_{n-r,n+r}$-measurable for some $r\in\bbN$ and all $n\in\bbN$, then the above holds true for any bounded function $G$.
\end{theorem}
For vector-valued $X_n$'s, such a characterization was obtained in \cite{HK2} and \cite{PolyPaper}. For bounded functions $G$ the proof when $X_n$ take values in some metric space is essentially the same.

\subsection{The LCLT}
 Usually, the local central limit theorem  concerns two cases, ``non-arithmetic" and ``lattice". 
We call the case non-arithmetic  if there exists no $t\not=0$ so that for some function $\be:\Lambda^{\ell-1}\to[0,2\pi)$ we have
\begin{equation}\label{NonLat}
e^{it G(x_1,...,x_\ell)}=e^{i\be(x_1,...,x_{\ell-1})},\,\mu^\ell\text{-a.s.}
\end{equation}
In particular, $G(x_1,x_2,...,x_\ell)$ is a not a function of the variables $x_1,...,x_{\ell-1}$ ($\mu^\ell$-almost surely), namely the function $G_\ell(x_1,x_2,....,x_\ell)$ is not identically $0$, $\mu^\ell$-almost surely.

\begin{theorem}[LCLT in the non-arithmetic case]\label{LCLT1}
Suppose that Assumption \ref{AssMix} holds true with some $\te_1>\max(4,\frac{3}{2\al})$ and $\te_2>\max(\frac{3}{2\ka\al},\frac{2}{\ka})$ 
 where $\al$ comes from \eqref{NonLinGr} and $\ka$ from \eqref{G Hold}. When some of the functions $q_j$ are linear, we also assume that $\te_2>3$ and $\te_1\geq \te_2\ka$.
 In addition we assume
 that $D^2>0$. Then in the above non-arithmetic case for any continuous function $g:\bbR\to\bbR$ with compact support (or an indicator of a bounded closed interval) we have
\[
\lim_{N\to\infty}\sup_{u\in\bbR}\left|\sqrt{2\pi N}D\bbE[g(S_N^{\{q_j\}}G-u)]-e^{-\frac{(u-\bar G N)^2}{2ND^2}}\int g(x)dx\right|=0.
\]

If $X_n$ is $\cF_{n-r,n+r}$-measurable for some $r\in\bbN$ and all $n\in\bbN$, then the above LCLT holds true for any bounded function $G$.
\end{theorem}

Next, we call the case a lattice one if $G$ is integer-valued and for all $t\in[-\pi,\pi]\setminus\{0\}$ there exists no function  $\beta:\Lambda^{\ell-1}\to[0,2\pi)$ satisfying (\ref{NonLat}). 
As in the non-arithmetic case, also in the lattice case  $G_\ell$ does not vanish $\mu^\ell$-almost surely.
More general ``lattice cases" can be considered, but we prefer to focus on integer valued-functions.
\begin{theorem}[LCLT in the lattice case]\label{LCLT2}
Under the assumptions of Theorem \ref{LCLT1}, in the above lattice case, for any continuous function $g:\bbR\to\bbR$ with compact support (or an indicator of a bounded closed interval) we have
\[
\lim_{N\to\infty}\sup_{u\in\bbZ}\left|\sqrt{2\pi N}D\bbE[g(S_N^{\{q_j\}}G-u)]-e^{-\frac{(u-\bar G N)^2}{2ND^2}}\sum_{k\in\bbZ}g(k)\right|=0.
\]
If $X_n$ is $\cF_{n-r,n+r}$-measurable for some $r\in\bbN$ and all $n\in\bbN$, then the above LCLT holds true for any bounded function $G$.
\end{theorem}

\begin{remark}
The lattice condition specified above includes the case when $G$ has the form
$G(x_1,...,x_\ell)=\prod_{j=1}^\ell\bbI_{\al_i}$ for some sets $\al_i$ with positive measure. Indeed, suppose that for some nonzero $t$ there exists a function $\be(x_1,...,x_{\ell-1})$ so that
\[
e^{it\prod_{j=1}^\ell \bbI_{\al_j}(x_j)}=e^{i\beta(x_1,...,x_{\ell-1})},\,\,\mu^\ell\text{-a.s.}
\]
When $y:=(x_1,...,x_{\ell-1})\not\in \al_1\times\alpha_2\times\cdots\times\alpha_{\ell-1}$ we get $1=e^{i\be(y)}$ and hence $\be(y)=0$. When $y\in\al_1\times\alpha_2\times\cdots\times\alpha_{\ell-1}$ but $x_\ell\not\in\al_\ell$ we still get that $e^{i\be(y)}=1$ and therefore $\be(y)=0$ for almost all $y$. Taking now $x_\ell\in \al_\ell$ we conclude that $e^{it}=1$
and hence $t=2\pi k\not\in[-\pi,\pi]\setminus\{0\}$.
\end{remark}

\begin{remark}
Suppose that $k>0$ and that  $q_1,...,q_k$ are linear.
When the function $G(x_1,x_2,...,x_\ell)$ does not depend on the variable $x_\ell$, but it is also not a function of  $x_1,...,x_k$ then we can write $G(x_1,x_2,...,x_\ell)=G(x_1,...,x_s)$ for a minimal $k<s<\ell$.
In this case, we can take $\be(x_1,...,x_{\ell-1})=itG(x_1,...,x_s)$ in \eqref{NonLat}, and so the conditions of Theorems \ref{LCLT1} and \ref{LCLT2} are never met. However, now we can replace $\ell$ with $s$ in \eqref{NonLat} (and in Theorems \ref{LCLT1} and \ref{LCLT2}). It will be clear from the proofs of Theorems \ref{LCLT1} and \ref{LCLT2} that we can also replace $\ell$ with $k<s<\ell$ when  $G(x_1,x_2,...,x_\ell)=G(x_1,x_2,...,x_s)$ only $\mu^\ell$-almost surely.
 Thus we obtain the LCLT in the non-arithmetic and lattice cases, formulated with $s$ instead of $\ell$. When $G$ depends only on  $x_1,x_2,...,x_k$ then we are in the arithmetic progression case $q_i(n)=in$ considered in the next section. 
\end{remark}

\begin{remark}
Condition \eqref{NonLinGr0} is only needed for the CLT to hold true, but the proofs of Theorems \ref{CLT} and \ref{Dthm} proceed similarly when all the $q_j$'s are polynomials so that  $\deg q_j\leq \deg q_{j+1}$ which take positive integer values on the set of positive integers.
Assuming that the CLT holds true with $D^2>0$ the proofs of Theorems \ref{LCLT1} and \ref{LCLT2} proceed similarly when, instead of \eqref{NonLinGr0}, we assume that
\begin{equation}\label{NonLinGr01}
\exists\ve_0\in(0,1)\,\text{ such that}\,\,\,\lim_{n\to\infty}\left(q_{i+1}(\ve_0 n)-q_{i}(n)\right)=\infty,\,\,\,i=k+1,...,\ell-1.
\end{equation}
Thus, after replacing \eqref{NonLinGr0} with \eqref{NonLinGr01}, we get that if $N^{-1/2}\big(S_N^{\{a_j\}}-\bar G N\big)$ obeys the CLT then the LCLT holds true in both lattice and non-arithmetic cases. We conclude that Theorems \ref{LCLT1} and \ref{LCLT2} hold true  when $q_j$ are polynomials so that $\deg q_j\leq \deg q_{j+1}$, and if the degrees are equal then the leading coefficient of $q_{j+1}$ is larger than the leading coefficient of $q_j$.
\end{remark}


\section{A nonconventional local CLT for Markov chains on Young towers}\label{LLT}
In this section we  describe our results for indexes having the form $q_j(n)=jn$ for $j=1,2,...,\ell$. As explained in Section \ref{sec1}, the LCLT will be obtained for Markov chains whose transition operator is the dual of the Koopman operator of a  Young tower. In Section \ref{HYPER} we will discuss certain applications to the LCLT for partially hyperbolic maps. For readers' convenience, in the following section we recall the definition of a Young tower first introduced in  \cite{Y1,Y2}.

\subsection{Young towers}\label{YTsec}
Let $(\Del_0,\cF_0,\nu_0)$ be a probability space, $\{\Del_0^j:\,j\geq1\}$ be a partition of $\Del_0$ (\text{mod} $\nu_0$), and $R:\Del_0\to\bbN$ be a (return time) function which is constant on each one of the $\Del_0^j$'s. We  identify each element $x$ in $\Del_0$ with the pair $(x,0)$, and
for each nonnegative integer $k$ let the $k$-th floor of the tower be defined by
\[
\Del_k=\{(x,k)\in\Del_0\times\{k\}:\,\,R(x)>k\}.
\]
For each $j$ so that $R|\Del_0^j>k$ set 
\[
\Del_k^j=\{(x,k)\in\Del_k:\,x\in\Del_0^j\}\subset\Del_k.
\]
The whole tower is defined by 
\[
\Del=\{(x,k):\,k\geq0,\,\,(x,k)\in\Del_k\}=\bigcup_{k\geq0}\Del_k.
\]
Let $f_0:\Del_0\to\Del_0$ be so that for each $j$ the map $f_0|\Del_0^j:\Del_0^j\to\Del_0$ is bijective (\text{mod} ${\nu_0}$). The dynamics on the tower is given by the map $F:\Del\to\Del$ defined by
\[
F(x,k)=\begin{cases}(x,k+1) & \text{ if }R(x)>k+1 \\(f_0(x),0)& \text{ if }R(x)=k+1\end{cases}.
\]
We think of $(f_0(x),0)$ as the return (to the base $\Del_0$) function corresponding to $F$, and when $R(x)=k+1$ we will also write $F^R(x,0):=F(x,k)=(f_0(x),0)$. It will also be convenient to set 
$F^R(x,k)=F^R(x,0)$ for any $k\geq1$ and $(x,k)\in\Del_k$.
We note that in applications usually $\Del_0$ is a subset of a larger set, and $f_0=f^R$ is the return time function (to $\Del_0$) of a different function $f$ (so that the tower is constructed in order to study statistical properties of $f$). We assume here that the partition $\cC=\{\Del_k^j\}$ is generating in the sense that 
\[
\bigvee_{i=0}^\infty F^{-i}\cC
\]
is a partition into points. For each $k\geq1$ and $x\in\Del$, we will denote the element of the partition 
\[
\cC_k=\bigvee_{i=0}^{k-1} F^{-i}\cC
\]
containing $x$ by $\cC_k(x)$ (so that $\{x\}=\cap_{k\geq0}\cC_k(x)$). The partition elements of $\cC_k$ are called cylinders of length $k$.

Next, we lift the $\sig$-algebra $\cF_0$ to $\Del$ by identifying $\Del_k^j$ with $\Del_0^j$ and lift the  probability measure $\nu_0$ to a measure on $\Del$, by assigning the mass $\nu_0(\Gam)$ to each subset $\Gamma$ of each $\Del_k^j$, for any $k$ and $j$ so that $R|\Del_0^j>k$. Let us denote the above $\sig$-algebra and measure on $\Del$  by $\cF$ and $m$, respectively.
Then the dual of the Koopman operator $g\to g\circ F$ with respect to the measure $m$ is given by 
\begin{equation}\label{P def}
Pg(x)=\sum_{y\in F^{-1}\{x\}}\frac{g(y)}{J F(y)}
\end{equation}
where $JF=\frac{dF_*m}{dm}$.
We will always assume that $\int Rd\nu_0<\infty$ which means that $m(\Del)<\infty$. Henceforth we will assume that $\nu_0$ has been normalized so that $m(\Del)=1$.

\subsubsection{Uniform towers}
The uniform distance $d_U$ on the space $\Del$ is defined as follows: for 
every $x$ and $y$ in $\Del$, we denote by $s_U(x,y)$ the greatest positive integer $n$ so that $F^px$ and $F^p y$ belong the same element of the partition $\{\Del_k^j\}$, for all $p<n$ (namely, they  belong to the same partition element in $\cC_{n}$  but not to the same element of $\cC_{n+1}$).  When $x=y$ we set $\beta(x,y)=\infty$.
Let $\be\in(0,1)$ and define a metric $d_U(\cdot,\cdot)$ on $\Del\times\Del$ by $d_U(x,y)=\beta^{s_U(x,y)}$ (where $\beta^\infty:=0$). The tower is called uniform if for every $k,j$,
\[
F^R:\Del_k^j\to\Del_0
\]
and its inverse are both non-singular with respect to $m$, and   the (inverse) Jacobian $JF^R$ is logarithmically locally Lipschitz continuous in the sense that for all $k$ and $x,y\in\Del_k^j$,
\begin{equation}\label{Jack Reg1}
\left|\frac{JF^R(x)}{JF^R(y)}-1\right|\leq Cd_U(F^Rx,F^Ry)
\end{equation}
for some constant $C$ which does not depend on $k,j,x$ and $y$.
We remark that uniform towers arise as extensions for certain classes of partially hyperbolic diffeomorphisms after collapsing along stable manifolds, see \cite{Y1}. 

In the next section we will obtain nonconventional limit theorems for uniform Young towers, but for the sake of clarity let us describe the setup of non-uniform towers.

\subsubsection{Non-uniform towers}
The non-uniform (separation) distance on the space $\Del$ is defined as follows: for 
any $x=(x^0,0)$ and $y=(y^0,0)$ in $\Del_0$, denote by $s_{NU}(x,y)$ the greatest positive integer $n$ so that $(F^R)^p(x)=f_0^p(x^0)$ and $(F^R)^p(y)=f_0^p(y^0)$ belong to the same element of the partition $\{\Del_0^j\}$ of $\Del_0$, for all $p<n$. If $x=(x^0,k)$ and $y=(y^0,k)$ belong to the same floor $\Del_k$ for some $k\geq1$ we  set $s_{NU}(x,y)=s_{NU}(x^0,y^0)$. When $x$ and $y$ are not in the same floor we set $s_{NU}(x,y)=0$. Let $\be\in(0,1)$ and define a metric $d_{NU}(\cdot,\cdot)$ on $\Del\times\Del$ by $d_{NU}(x,y)=\beta^{s_{NU}(x,y)}$. The tower is called non-uniform if for every $j$,
\[
F^R:\Del_0^j\to\Del_0
\]
and its inverse are both non-singular with respect to $m$ (or $\nu_0$), and for all $x,y\in\Del_0^j$,
\begin{equation}\label{Jack Reg}
\left|\frac{JF^R(x)}{JF^R(y)}-1\right|\leq Cd_{NU}(F^R x,F^R y)
\end{equation}
for some constant $C$ which does not depend on $j$. 
\vskip0.2cm
It is evident that $d_U\leq d_{NU}$. It is also clear that the topologies induced by $d_U$ and $d_{NU}$ coincide.

We have the following result.
\begin{theorem}[\cite{Y2}]
Let $(\Del,F)$ be a non-uniform Young tower so that
$\int R d\nu_0<\infty$. Then there exists a strictly positive Lipschitz continuous function $h:\Del\to\bbR$ (w.r.t. to $d_{NU}$) which is bounded and bounded away from $0$ and the measure $\mu=hd m$ is $F$-invariant. The measure $\mu$ is the unique absolutely continuous $F$-invariant measure and $h$ satisfies $Ph=h$. 
\end{theorem}
We remark that this theorem was also  obtained in \cite{Y1} for uniform Young towers, and that in this case the function $h$ is Lipschitz continuous w.r.t. $d_U$.

As discussed in Section \ref{sec1}, when $q_i(n)=in$ we consider uniform Young towers
with exponential tails, namely we assume that there exist constants $p>0$ and $q>0$ so that for all $n\geq1$,
\begin{equation}\label{ExTails}
m\{x:\,R(x)>n\}\leq qe^{-pn}.
\end{equation}
Of course, this is equivalent to having exponential tails $\nu_0\{x:\,R(x)>n\}\leq q_1e^{-p_1n}$ for $\nu_0$ with some $q_1,p_1>0$.
Finally, we will assume in this paper that the tower is aperiodic in the sense that $\gcd\{R_j\}=1$. This is equivalent to $F$ being mixing with respect to $\mu$.

\subsection{Markov chains on towers}\label{MC}
Let $(\Del,m,F)$ be an aperiodic uniform Young tower satisfying \eqref{ExTails}, and let $\mu=hdm$ be its  unique absolutely continuous invariant measure. Let $P$ be the transfer operator defined by \eqref{P def}, and consider the operator $\cA$ given by $\cA g=P(gh)/h$.
 Let $\{\xi_n:\,\,n\geq0\}$ be the stationary Markov chain on $\Del$ whose initial distribution is $\mu$, having $\cA$ as its transition operator. Namely, for every $n\in\bbN$ and a Borel measurable set $\Gamma\subset\Del$ we have
 $$
\bbP(\xi_n\in \Gamma|\xi_{n-1},...,\xi_1)=\bbP(\xi_n\in \Gamma|\xi_{n-1})
 =h(\xi_{n-1})^{-1}\sum_{y\in\Gamma: Fy=\xi_{n-1}}\frac{h(y)}{J F(y)}.
 $$
 Let $X_0$ be a $\Del$-valued random variable whose distribution is $\mu$. Then, it follows by induction  that for every $n\in\bbN$ we have
\begin{equation}\label{Inversion}
(F^{n-1}X_0,...,FX_0, X_0)\overset{d}{=}(\xi_0,\xi_1,...,\xi_{n-1})
\end{equation}
where $\overset{d}{=}$ stands for equality in distribution. 
Let $\ell\in\bbN$ and set $\Del^\ell=\Del\times\Del\times\dots\times\Del$ ($\ell$ times).
 Let $G:\Del^\ell\to\bbR$ be a bounded function so that with some $K>0$ for all $(x_1,...,x_\ell), (y_1,...,y_\ell)\in\Del^{\ell}$ we have
\begin{equation}\label{G Hold1}
 |G(x_1,...,x_\ell)-G(y_1,...,y_\ell)|\leq K\sum_{j=1}^{\ell}d_{U}(x_i,y_i).
\end{equation}
For each $N$ set
\[
S_N^{\{q_j\}}G=\sum_{n=1}^N G(\xi_{n},\xi_{2n},...,\xi_{\ell n})
\]
where $q_j(n)=jn$. We also set 
\begin{equation}\label{bar G def2}
\bar G=\int G(x_1,x_2,...,x_\ell)d\mu(x_1)d\mu(x_2)\cdots d\mu(x_\ell).
\end{equation}
The main result in this section is an LCLT for the sequence of random variables $S_N^{\{q_i\}}G$.

Before discussing the LCLT, let us present our results concerning the CLT.
\begin{theorem}[CLT and asymptotic variance]\label{CLT1}
Suppose that \eqref{Jack Reg} and \eqref{ExTails} hold true and that $\gcd\{R_j\}=1$. Moreover, 
assume that  $G$ is a bounded function satisfying \eqref{G Hold1}. Then:

(i) the limit
\[
D^2=\lim_{N\to\infty}\frac 1N\bbE\left[\big(S_N^{\{q_j\}}G-\bar G N\big)^2\right]
\]
exists.
Moreover, the sequence $N^{-\frac12}\big(S_N^{\{q_j\}}G-\bar G N\big)$ converges in distribution as $N\to\infty$ towards a centered normal random variable with variance $D^2$.

(ii)  $D^2=0$ if and only if $G-\bar G$ is an $L^2$-coboundary with respect to the map $F\times F^2\times\cdots\times F^\ell$.
\end{theorem}
This theorem is proved similarly to Theorems \ref{CLT} and \ref{Dthm} taking into account Korepanov's semi-conjugacy \cite{Kor1} and Section \ref{BS}.

In what follows
we will  assume that the function $G_\ell$ defined by \eqref{G ell}
is not a coboundary with respect to $F_\ell=F\times F^2\times\cdots\times F^\ell$.
This means that
\[
D_\ell^2:=\lim_{N\to\infty}\frac 1N\bbE\left[(S_{N}^{\{q_j\}}G_\ell)^2\right]>0.
\]
 Note that $D^2>0$ if $D_\ell^2>0$ since $G_\ell$ admits a coboundary representation with respect to $F_\ell$ if $G$ admits such a representation.

We will obtain the LCLT under the following.
\begin{assumption}[Regularity around a periodic orbit]\label{Second continuity of F}
(i) The map $F$ has a periodic point $x_0$.  
\vskip0.1cm
(ii) Let $n_0$ be the period of $x_0$. Then
the map $y\to G_y:=G(y,\cdot)$, $y\in \Del^{\ell-1}=\Del\times\Del\times\dots\times\Del$
is continuous at the points 
$y=(F^k x_0,F^{2k}x_0,...,F^{(\ell-1)k}x_0),\,k=0,1,...,n_0-1$ when considered as a map from 
$\Del^{\ell-1}$ to the space of bounded  Lipschitz continuous functions  equipped with the norm $\|g\|_{\text{Lip},U}=\sup|g|+\text{Lip}_{U}(g)$, where $\text{Lip}_{U}(g)$ is the smallest Lipschitz constant of $g$ with respect to the metric $d_{U}$.
\end{assumption}

As in the case when $q_\ell$ grows faster than linearly, in order to present the LCLT we will need to distinguish between non-arithmetic and lattice cases. In contrast with that case, this will be  done in a close manner to the ``conventional case" $\ell=1$ and $q_1(n)=n$, as described in the following.
Under Assumption \ref{Second continuity of F} we define a function $G_{x_0,n_0}:\Del\to\bbR$ by
\begin{equation}\label{F x0 m}
 G_{x_0,n_0}(x)=
\sum_{k=0}^{n_0-1}G(F^kx_0,F^{2k}x_0,...,F^{(\ell-1)k}x_0,F^{\ell k}x)=\sum_{k=0}^{n_0-1}G\circ F_\ell^k( x_0,x_0,...,x_0,x)
\end{equation} 
where $F_\ell=F\times F^2\times\cdots\times F^\ell$. We call the case ``non-arithmetic" if there exists no real nonzero $t$ so that with some $\la\in\mathbb S^1$ and a Lipschitz continuous non-vanishing function $g$ we have
\[
e^{it  G_{x_0,n_0}}=\la g/(g\circ F^{n_0\ell}),\,\mu\text{-a.s.}
\]
In other words, $G_{x_0,n_0}$ is non-arithmetic  with respect to $(F^{n_0\ell},\mu)$ in the classical sense (see \cite{GH,HH,GO}).

\begin{theorem}[LCLT in the non-arithmetic case]\label{Non lattice LLT3}
Assume that \eqref{Jack Reg1} and \eqref{ExTails} hold true, and that $\gcd\{R_j\}=1$. 
Suppose also that $D^2_\ell>0$ (so $D^2>0$) and that $G$ is a bounded function satisfying \eqref{G Hold1}  so that Assumption \ref{Second continuity of F} holds true. Then, in the non-arithmetic case for any continuous function $g:\bbR\to\bbR$ with compact support (or an indicator of a bounded closed interval) we have
\[
\lim_{N\to\infty}\sup_{u\in\bbR}\left|\sqrt{2\pi N}D\bbE[g(S_{N}^{\{q_j\}}G-u)]-e^{-\frac{(u-\bar G N)^2}{2ND^2}}\int g(x)dx\right|=0.
\]
\end{theorem}

Next, we call the case a lattice one if $G$ is integer-valued and the function
$G_{x_0,n_0}$ cannot be written in the form
\begin{equation}\label{LLT NonCon lattice equality}
G_{x_0,n_0}=a+\be-\be\circ F^{\ell n_0}+q_0\textbf{k},\,\,\mu\text{-a.s.}
\end{equation}
for some  $q_0>1$, $a\in\bbR$, $\be:\Del\to\bbR$ and an integer valued function $\textbf{k}:\Del\to\bbZ$. This means that $e^{it G_{x_0,n_0}}$ is not cohomologous to a constant when $0<|t|<2\pi$.

\begin{theorem}[LCLT in the lattice case]\label{Lattice LLT3}
Assume that \eqref{Jack Reg1} and \eqref{ExTails} hold true, and that $\gcd\{R_j\}=1$.
Suppose also that $D^2_\ell>0$ (so $D^2>0$). Then, in the lattice case for any continuous function $g:\bbR\to\bbR$ with compact support (or an indicator of a bounded closed interval) we have
\[
\lim_{N\to\infty}\sup_{u\in\bbZ}\left|\sqrt{2\pi N}D\bbE[g(S_{N}^{\{q_j\}}G-u)]-e^{-\frac{(u-\bar G N)^2}{2ND^2}}\sum_{k\in\bbZ}g(k)\right|=0.
\]
\end{theorem}

\section{Nonconventional CLT limit theorems with nonlinear indexes}\label{Sec4}
Let $(\Om,\cF,\bbP)$ be the probability space from Section \ref{sec2} on which $\{X_n\}$ is defined.
First, by replacing $G$ with $G-\bar G$, where $\bar G$ is defined in \eqref{bar G def1}, we can assume without loss of generality that $\bar G=0$.
For every $n$ and $r$ in $\bbN$ let us take an $\cX$-valued and $\cF_{n-r,n+r}$ measureable random variable $X_{n,r} $ so that 
\begin{equation}\label{X n r }
\|d(X_n,X_{n,r})\|_{L^2}\leq 2\beta_2(r)=O(r^{-\te_2}).
\end{equation}
The difference in Theorem \ref{CLT} in comparison with \cite[Ch.1]{book} is that in the present paper the sequence $\{X_n\}$ is not vector-valued.
 However, this was only needed in \cite[Ch.1]{book} since the functions considered there were not bounded, and instead they had polynomial growth. 
 Using $X_{n,r}$ instead of $\bbE[X_n|\cF_{n-r,n+r}]$, 
the proof of the CLT  proceeds almost exactly as in \cite[Ch.1]{book}.
 For readers' convenience we will provide most of the details.

\subsubsection{Expectation estimates and the asymptotic variance}
Recall that the $\phi$-mixing (dependence) coefficient between two sub-$\sig$-algerbras $\cG,\cH\subset\cF$ is defined by
\begin{equation}\label{GenPhi}
\phi(\cG,\cH)=\sup\{|\bbP(B|A)-\bbP(B)|:\,A\in\cG, B\in\cH, \bbP(A)>0\}.
\end{equation}
By \cite[Ch. 4]{Br}, it can also be written in the form
\begin{equation}\label{PhiId}
\phi(\cG,\cH)=\frac12\sup\{\|\bbE[h|\cG]-\bbE[h]\|_{L^\infty}:\,h\in L^\infty(\Om,\cH), \|h\|_{L^\infty}\leq 1\}.
\end{equation}
Using \eqref{PhiId}, we have the following.
\begin{lemma}\label{thm3.11-StPaper}
Let $\cG_1,\cG_2\subset\cF$ be two sub-$\sigma$-algebras of $\cF$ and for $i=1,2$ let
$V_i$ be an $\cX^{d_i}$-valued random variable which is $\cG_i$-measurable. 	Let us denote by $\mu_i$ the distribution of $V_i$ and  set
$d=d_1+d_2$ and $\mu=\mu_1\times\mu_2$. Denote by $\zeta$
the distribution of  $(V_1,V_2)$
and consider the measure $\nu=\frac12(\zeta+\mu)$. Let $\cB$  be the Borel $\sigma$-algebra on $\cX^d$ and 
$H\in L^\infty(\cX^d,\cB,\nu)$. Then $\bbE[H(V_1,V_2)|\cG_1]$ and $\bbE[H(v,V_2)]$ exist for $\mu_1$-almost
every $v\in\cX^{d_1}$ and
\begin{equation}\label{Meas-StPaper}
|\bbE[H(V_1,V_2)|\cG_1]-h(V_1)|
\leq2\|H\|_{L^\infty(\cX^d,\cB,\nu)}\phi(\cG_1,\cG_2),\,\,\bbP-a.s.
\end{equation}
 where $h(v)=\bbE[H(v,V_2)]$ and a.s. stands for almost surely.
\end{lemma}

The proof proceeds exactly as the proof of \cite[Lemma 1.3.11]{book}, and it is given here for readers' convenience.
\begin{proof}
Clearly $H$ is bounded $\mu$ and $\zeta$ almost surely. Thus,  
$\bbE[H(V_1,V_2)|\cG_1]$ exists  and existence of $\bbE[H(v,V_2)]$ ($\mu_1$-a.s.)
 follows from the Fubini theorem.
 Relying on (\ref{PhiId}), inequality (\ref{Meas-StPaper}) 
follows easily for functions 
of the form $G(v_1,v_2)=\sum_i\bbI_{\{v_1\in A_i\}}g_i(v_2)$, where $\{A_i\}$
is a measurable partition of the support of $\mu_1$ and $\bbI_{\{v_1\in A_i\}}=1$
when $v_1\in A_i$ and equals $0$ otherwise.
 Any uniformly continuous function
$H$ is a uniform limit of functions of the above form, which implies that (\ref{Meas-StPaper}) 
holds true for uniformly continuous functions. 
Finally, by Lusin's theorem, 
any  function $H\in L^\infty(\cX^d,\cB,\nu)$ is an $L^1$ (and a.s.) limit of a sequence 
$\{H_n\}$ of continuous functions with compact support satisfying  
$\|H_n\|_{L^\infty(\cX^d,\cB,\nu)}\leq\|H\|_{L^\infty(\cX^d,\cB,\nu)}$
 and (\ref{Meas-StPaper}) follows for all  $H\in L^\infty(\cX^d,\cB,\nu)$.
\end{proof}

Next, let $k_0\in\bbN$ and $U_i,\,i=1,2,...,k_0$ be  random variables so that $U_i$ is $\cX^{d_i}$-valued for some $d_i\in\bbN$, which are defined on the
probability space $(\Om,\cF,\bbP)$, and
$\{\cC_j: 1\leq j\leq s\}$ be a partition of $\{1,2,...,k_0\}$. 
Consider the random variables $U(\cC_j)=\{U_i: i\in\cC_j\}$, $j=1,...,s$, 
and let 
\[
U^{(j)}(\cC_i)=\{U_i^{(j)}: i\in\cC_j\},\,\, j=1,...,s
\]
be independent copies of the $U(\cC_j)$'s.
For each  $1\leq i\leq k_0$ let $a_i\in\{1,...,s\}$ be the unique index
such that $i\in\cC_{a_i}$, and 
for every  bounded Borel function $H:\bbR^{d_1+d_2+...+d_{k_0}}\to\bbR$ set
\begin{equation}\label{ExpDiffDef}
\cD(H)=\big|\bbE[H(U_1,U_2,...,U_{k_0})]-\bbE[H(U_1^{(a_1)},U_2^{(a_2)},...,U_{k_0}^{(a_{k_0})})]\big|.
\end{equation}
Relying on Lemma \ref{Meas-StPaper}, the following result is proved exactly as \cite[Corollary 1.3.11]{book}.
\begin{corollary}\label{lem3.1-StPaper}
Suppose that each $U_i$ is $\cF_{m_i,n_i}$-measurable, where $n_{i-1}<m_i\leq n_i<m_{i+1}$, 
$i=1,...,k_0$, $n_0=-\infty$ and $m_{k_0+1}=\infty$.  
Then, for every bounded
Borel function $H:\cX^{d_1+d_2+...+d_{k_0}}\to\bbR$,
\begin{equation*}
\cD(H)\leq 4\sup|H|\sum_{i=2}^{k_0}\phi(m_i-n_{i-1})
\end{equation*}
where $\sup|H|$ is the supremum of $|H|$. 
\end{corollary}

\begin{proof}[Existence of asymptotic variance and proof of Theorem \ref{Dthm}]
Using Corollary \ref{lem3.1-StPaper} instead of \cite[Lemma 4.3]{KV1}, the proof of the existence of the limit $D^2$ proceeds similarly to \cite{KV1}, as well as the proofs of the characterizations of positivity of $D^2$ from \cite{HK2} and \cite{PolyPaper} since for bounded functions the arguments in these  proofs only require that
$$
\sum_{n=0}^{\infty}(n+1)\left(\phi(n)+(\beta_2(n))^\ka\right)<\infty.
$$
\end{proof}

\subsection{Proof of the CLT}
First, it is enough to prove the CLT when $D^2>0$, for otherwise we get convergence in $L^2$ to $0$.
We will brake the proof of the CLT into two parts.
\subsubsection{Step 1: approximation.}\label{S1}
Recall first that $\te_1>4$ and $\ka\te_2>2$. Let us fix some $0<\zeta<\frac14$ so that $\zeta\te_1>1$ and $\ka\te_2\zeta>1/2$.  Set $r_N=[N^\zeta]$ and 
$$
S_{N,r_N}=S_{N,r_N}^{\{q_j\}}G=\sum_{n=1}^{N}G(X_{q_1(n),r_N},X_{q_2(n),r_N},...,X_{q_\ell(n),r_N}).
$$
\begin{lemma}\label{ApLema}
We have
\begin{equation}\label{L2 approx}
\lim_{n\to\infty}\left\|N^{-1/2}S_N^{\{q_j\}}G-N^{-1/2}\left(S_{N,r_N}-\bbE[S_{N,r_N}]\right)\right\|_{L^2}=0.
\end{equation}
Therefore, in order to prove the CLT for $N^{-1/2}S_N^{\{q_j\}}G$ it is enough to prove that
\begin{equation}\label{W N}
W_N:=\frac{S_{N,r_N}-\bbE[S_{N,r_N}]}{\sqrt N D}
\end{equation}
converges in distribution to the standard normal law.
\end{lemma}
\begin{proof}
By \eqref{X n r } and the H\"older inequality we have
\begin{equation}\label{kappa}
\bbE[d(X_n,X_{n,r})^\ka]=\|d(X_n,X_{n,r})^\ka\|_{L^1}\leq \|d(X_n,X_{n,r})^\ka\|_{L^{1/\ka}}
\end{equation}
$$
=\left(\bbE[d(X_n,X_{n,r})]\right)^{\ka}\leq \left(2\be_2(r)\right)^\ka.
$$
Thus, by \eqref{G Hold} and \eqref{X n r },
\begin{equation}\label{Ap}
\|S_N^{\{q_j\}} G-S_{N,r_N}\|_{L^2}=O(N(\beta_2(r_N))^\ka)=O(N^{1-\ka\te_2\zeta}).
\end{equation}
Since  $\ka\te_2\zeta>1/2$, we have
$$
\lim_{n\to\infty}\|N^{-1/2}S_N^{\{q_j\}}G-N^{-1/2}S_{N,r_N}\|_{L^2}=0
$$
and 
$$
D^2=\lim_{N\to\infty}N^{-1}\bbE[(S_{N,r_N})^2].
$$
In order to complete the proof of the lemma, it is enough to show that
\begin{equation}\label{Cent}
\lim_{N\to\infty}N^{-1/2}\bbE[S_{N,r_N}]=0.
\end{equation}
For $j=1,2,...,\ell$, let $\{(X_{n}^{(j)}, X_{n,r}^{(j)})\}$ be independent copies of $\{(X_n,X_{n,r})\}$.
Using  \eqref{X n r }, Corollary \ref{lem3.1-StPaper} with $U_j=X_{q_j(n),r_N}$, \eqref{kappa} with the independent copies  and that $\bar G=0$ we have 
$$
|\bbE[G(X_{q_1(n),r_N}, X_{q_2(n),r_N},...,X_{q_\ell(n),r_N})]|\leq 
|\bbE[G(X_{q_1(n),r_N}^{(1)}, X_{q_2(n),r_N}^{(2)},...,X_{q_\ell(n),r_N}^{(\ell)})]|
$$
$$+
\ell\phi(([d_n/3]))\leq |\bbE[G(X_{q_1(n)}^{(1)}, X_{q_2(n)}^{(2)},...,X_{q_\ell(n)}^{(\ell)})]|+K\ell(2\be_1(r_N))^\ka+\ell\phi(([d_n/3]))$$
$$= |\bar G|+K\ell(\be_2(r_N))^\ka+\ell\phi([d_n/3])=O(N^{-\ka\zeta\te_2})+\ell\phi([d_n/3])
$$
where $d_n=\min\{|q_{i+1}(n)-q_i(n)|:\,1\leq i<\ell\}$. Since $d_n\geq c_0n$ for some $c_0>0$ and all $n$ large enough, by \eqref{MixA}  we have
$$
|\bbE[S_{N,r_N}]|=O\left(\sum_{n}n^{-\te_1}\right)+O(N^{1-\ka\zeta\te_2})=o(N^{1/2})
$$
where we have used that $\te_1>4$ and $\ka\zeta\te_2>1/2$. Now
 \eqref{Cent} follows, and the proof of the lemma is complete.
\end{proof}

\subsubsection{Step 2: dependency graphs-the CLT for $W_N$}\label{Step2}
For every $n,m\in\bbN$ set 
\begin{equation}\label{rhoDef}
\rho(n,m)=\min\{|q_i(n)-q_j(m)|:\,\,1\leq i,j\leq\ell\}.
\end{equation}
Consider the graph $\cG_N=(V_N,E_N)$, where $V_N=\{1,...,N\}$ and $(n,m)\in E_N$ if $\rho(n,m)\leq r_N$. For every $n\in V_N$ set 
$$Q_{n,N}=G(X_{q_1(n),r_N},X_{q_2(n),r_N},...,X_{q_\ell(n),r_N})$$
and 
$$
Z_n=Z_{n,N}=\frac{Q_{n,N}-\bbE[Q_{n,N}]}{\sqrt N D}.
$$
Then  $W_N$ from \eqref{W N} satisfies
\begin{equation}\label{WN}
W_N=\sum_{n\in V_N}Z_n\,\,\text{ and }\,\,\lim_{N\to\infty}\bbE[W_N^2]=1.
\end{equation}
We have the following.
\begin{lemma}\label{D lemma}
There exists $A_0>0$ which depends only on $q_1,q_2,...,q_\ell$ so that
the size of a ball of radius $1$ in the graph $G_N$ does not exceed  $A_0r_N$. 
Thus,  a ball of radius $3$ is at most of size $A_1r_N^3$, where $A_1>0$ is another constant.
\end{lemma}

\begin{proof}
Let us fix some $n\in V_N$ and let $m\in V_N$ be so that $d_{\cG_N}(n,m)=1$ (i.e $(n,m)$ is an edge). Then there are indexes $1\leq i,j\leq \ell$  so that 
$$
|q_i(n)-q_j(m)|\leq r_N.
$$
Therefore $q_j(m)\in[q_i(n)-r_N,q_i(n)+r_N]:=I_{n,i,r_N}$. If  $q_j$ is linear then $q_j(m)=jm$ and so there are at most $(2r_N+1)/j$ positive integers satisfying $jm\in I_{n,i,r_N}$. If $q_j$ grows faster than linearly then there are at most $c_0r_N$ positive integers satisfying $jm\in I_{n,i,r_N}$, for some $c_0>0$ which does not depend on $n$.
\end{proof}

\begin{proof}[Proof of the CLT for $W_N$]
Fix some $N$ and  let $\cG_N=(V_N,E_N)=(V,E)$ be the above graph. Set $Z_{n,N}=Z_n$ and $\gamma_4=\gamma_4(N)=\max_{n\in V}\|Z_{n}\|_{L^4}=O(N^{-1/2})$. Then 
$$
r_N^2(\gamma_4^3|V_N|+\gamma_4^2|V_N|^{1/2})=O(N^{2\zeta-1/2})=o(1)
$$
where $|V_N|=|V|=N$ is the cardinality of $|V|$. For every $v\in V$ set
\begin{equation}
N_v=B(1,v)=\{v\}\cup\{u\in V: (u,v)\in E\},\,\,N_v^c=V\setminus N_v.
\end{equation} 
Thus, applying \cite[Theorem 1.2.1]{book} with  $W=W_N$, $\rho=3$ and the graph $\cG_N=(V_N,E_N)$, the CLT for $W_N$ will follow if the following sequences converge to $0$ as $N\to\infty$:
$$
\del_1(N)=\sup\left\{\left|\sum_{n\in V}\bbE\left[Z_{n}g\left(\sum_{u\in N_v^c}Z_{u}\right)\right]\right|: \sup|g|\leq1\right\},$$
$$
\del_2(N)=\left|\sum_{v\in V}\sum_{u\in N_v^c}\bbE[Z_{n}Z_{u}]\right|,$$
$$
\del_3(N)=\left|\sum_{(v_1,u_1,v_2,u_2)\in\Gam}\text{Cov}(Z_{v_1}Z_{u_1},Z_{v_2}Z_{u_2})\right|^{1/2}
$$
where $\Gam=\Gamma_N=\{(v_1,u_1,v_2,u_2): d(v_1,v_2)>3
\text{ and }u_i\in N_{v_i}, i=1,2\}\subset V^4=V\times V\times V\times V$ (and $d(\cdot,\cdot)$ is the distance function in the graph $\cG$).

Proceeding as in  \cite[Corollary 1.3.16]{book} with $b=\infty$, we obtain from Corollary \ref{lem3.1-StPaper} and Lemma \ref{D lemma} that
$$\del_1=O(N^{1/2}\phi(r_N))=O(N^{1/2-\zeta\te_1}),\, \del_2=O(N\phi(r_N))=O(N^{1-\zeta\te_1})$$ and $\del_3^2=O(r_N^2\phi(r_N))=O(N^{2\zeta-\te_1\zeta})$. Using that $\zeta\te_1>1$ and $\zeta<1/4$ we see that $\del_i(N)=o(1)$ for $i=1,2,3$. 
\end{proof}

\section{A Nonconventional LCLT with nonlinear indexes}

\subsection{General conditions for the local CLT}
In this section we will state a general result showing that the LCLT follows from the CLT plus certain decay rates of the characteristic functions on appropriate domains.

Let $Z_1,Z_2,Z_3,...$ be a sequence of real-valued random variables and let
 $\varphi_N:\bbR\to\bbC$ be the
characteristic function of $Z_N$ given by $\varphi_N(t)=\bbE[e^{itZ_N}]$.  

We will consider the following
growth properties of the characteristic functions $\varphi_N$. 
\begin{assumption}\label{DecRate1}
There exist $\del_0\in(0,1)$, positive constants $c_0$ and $d_0$  
and a sequence $(b_N)_{N=1}^\infty$ of real numbers such that
$\lim_{N\to\infty}N^{\frac12}b_N=0$ and   
\[
|\varphi_N(t)|\leq c_0e^{-d_0Nt^2}+b_N
\]
for all $N\in\bbN$ and $t\in[-\del_0,\del_0]$.
\end{assumption}
\begin{assumption}\label{DecRate2}
For every $\del>0$ we have
\[
\lim_{N\to\infty}N^{\frac12}\sup_{t\in J_\del}|\varphi_N(t)|=0
\]
where
$J_\del=[-\del^{-1},-\del]\cup[\del,\del^{-1}]$.
\end{assumption}

Assumption \ref{DecRate2} cannot hold when $Z_N\in h_0\bbZ=\{kh_0:\,k\in\bbZ\}$ for some $h_0>0$ since then $\varphi_N$ is periodic. In this lattice  case we consider the following assumption.

\begin{assumption}\label{DecRate3}
There exists $h_0>0$ so that for all $N\in\bbN$ we have $\bbP(Z_N\in h_0\bbZ)=1$. Moreover,
for every $\del>0$,
\[
\lim_{N\to\infty}N^{\frac12}\sup_{t\in J_{\del,h_0}}|\varphi_N(t)|=0
\]
where $J_{\del,h_0}=[-\frac{\pi}{h_0},\frac{\pi}{h_0}]\setminus(-\del,\del)$.
\end{assumption}

\begin{theorem}[Theorem 2.2.3, \cite{book}]\label{General LLT }
Suppose that  for some $m\in\bbR$ the sequence $N^{-\frac12}(Z_N-mN)$ converges in distribution as $N\to\infty$
to a centered normal random variable with variance $\sig^2>0$.
Then the following types of local central limit theorem hold true: 
\vskip0.2cm
(i) Under  Assumptions \ref{DecRate1} and \ref{DecRate2},
 for any real continuous function $g$ on $\bbR$ with compact support (or an indicator of a bounded interval),
\begin{equation}\label{LLT-General}
\lim_{N\to\infty}\sup_{u\in\bbR}
\left|\sig\sqrt{2\pi N}\bbE[g(Z_N-u)]-e^{-\frac{(u-mN)^2}{2N\sig^2}}
\int g(x)dx\right|=0.
\end{equation} 
\vskip0.2cm
(ii)
Under Assumptions  \ref{DecRate1} and \ref{DecRate3},
 for any real continuous function $g$ on $\bbR$ with compact support (or an indicator of a bounded interval),
\begin{equation}\label{LLT-Genera2}
\lim_{N\to\infty}\sup_{u\in\bbZ}
\left|\sig\sqrt{2\pi N}\bbE[g(Z_N-uh_0)]-e^{-\frac{(u-mN)^2}{2N\sig^2}}
\sum_{k\in\bbZ}g(h_0k)\right|=0.
\end{equation} 
\end{theorem}

\subsection{Proof of Theorems \ref{LCLT1} and \ref{LCLT2} without linear indexes}\label{Sec5.2}
To increase readability  we will treat the case when all $q_j$ grow faster than linearly separately (namely the case $k=0$). This will also motivate the proof in the case when some of the functions are linear, which is more complicated.
 In the course of the proof in both cases we will need the following simple observation, which for the sake of convenience is formulated as a lemma.
\begin{lemma}\label{Qlem}
There exist $\del_0,a\in(0,1)$ so that for every $N\in\bbN$ large enough,  we have
\begin{equation}\label{A}
q_i(N)+\del_0N\leq q_{i+1}([aN])-\del_0 N,\,i=1,2,...,\ell-1.
\end{equation}
\end{lemma}
This lemma follows directly from \eqref{NonLinGr0} and \eqref{kLin} (the latter is used when some of the functions are linear). When  the functions $q_2,...,q_\ell$ grow faster than linearly (i.e. $k=1$) we can take any $a$ and $\del_0$, while when $q_1,...,q_k$ are the linear ones, $1<k<\ell$, we can take $a=a_k=1-\frac1{4k}$ and $\del_0=\frac{1}{6}$ (which insures \eqref{A} for $i=1,2,...,k-1$). Note that by \eqref{NonLinGr}, for every $N$ large enough
and for all $n\geq [aN]$,
\begin{equation}\label{B}
q_{j}(n+1)-q_j(n)\geq ([aN])^\al,\,j=k+1,...,\ell.
\end{equation} 


Next, for each $r,N\in\bbN$ set
\begin{equation}\label{SNr}
S_{N,r}^{\{q_j\}}G=\sum_{n=1}^N G(X_{q_1(n),r},...,X_{q_\ell(n),r}).
\end{equation}
Using \eqref{MixA} and  \eqref{X n r } 
and the H\"older inequality we have 
$$
\|d(X_n,X_{n,r})^\ka\|_{L^2}\leq\|d(X_n,X_{n,r})^\ka\|_{L^{2/\ka}}=\|d(X_n,X_{n,r})\|_{L^2}^\ka\leq 2(\beta_2(r))^\ka.
$$
Thus, by \eqref{G Hold} with some constants $C_1,C_2>0$, for all $N$ and $r$ we have
\begin{equation}\label{SNr approx}
\|S_{N}^{\{q_j\}}G-S_{N,r}^{\{q_j\}}G\|_{L^1}\leq\|S_{N}^{\{q_j\}}G-S_{N,r}^{\{q_j\}}G\|_{L^2}\leq C_1N(\be_2(r))^{\ka}\leq C_2Nr^{-\te_2}.
\end{equation}

Now, since $S_N^{\{q_j\}}G$ obeys the CLT, our goal here is to verify Assumption \ref{DecRate1} and either Assumption \ref{DecRate2} (in the non-arithmetic case) or Assumption \ref{DecRate3} with $h_0=1$ (in the lattice case) with $Z_N=S_N^{\{q_j\}}G$. 
The main ingredient in the proof when $q_1,q_2,...,q_\ell$  grow faster than linearly is the following lemma.

\begin{lemma}[Conditioning step without linear indexes]\label{ZeLem1}
Let $\zeta:\bbR\to\bbR$ be given by 
\begin{equation}\label{zeta t def}
\zeta(t)=\int\left|\int e^{itG(x_1,...,x_{\ell-1},x_\ell)}d\mu(x_\ell)\right|d\mu(x_1)d\mu(x_2)...d\mu(x_{\ell-1}).
\end{equation}
Then for every compact set $J\subset\bbR$
there exist a sequence $(b_N)$ so that $b_N=o(N^{-1/2})$ and a constant $c_0>0$ such that for all $N\in\bbN$  and $t\in J$ we have
\begin{equation}\label{PP}
\left|\bbE\big[\exp\big(itS_{N}^{\{q_j\}}G\big)\big]\right|\leq b_N+\left(\zeta(t)\right)^{c_0 N}.
\end{equation}
\end{lemma}

\begin{proof}
First, let us take some compact set $J\subset\bbR$ and fix some $t\in J$ and $N\in\bbN$.
Set $r=r_N=[aN^\al]/4$ and $M_N=[aN]$, where $\al$ comes from \eqref{NonLinGr} and $a$ from Lemma \ref{Qlem}. By \eqref{SNr approx} and the mean value theorem we get that
\begin{equation}\label{Onee}
\left|\bbE\big[\exp\big(itS_{N}^{\{q_j\}}G\big)\big]-\bbE\big[\exp\big(itS_{N,r_N}^{\{q_j\}}G\big)\big]\right|
\end{equation}
$$
\leq |t|\|S_N^{\{q_j\}}G-S_{N,r}^{\{q_j\}}G\|_{L^1}\leq C'|t| N^{1-\ka\te_2\al}=o(N^{-1/2})
$$
where $C'>0$ is some constant, and in the last inequality we have used our assumption that $\ka\te_2\al>3/2$ (and that $|t|\leq C(J)$ for some $C(J)$ which depends only on $J$).
Consider now the random variables $V_1,V_2$ and $V_3$ given by
$$
V_1=V_{1,N}=\{X_{q_i(n),r_N}:\,1\leq i\leq \ell,\,1\leq n\leq M_N\},
$$
$$
V_2=V_{2,N}=\{X_{q_i(n),r_N}:\,1\leq i<\ell,\,M_N<n\leq N\}
$$
and 
$$
V_3=V_{3,N}=\{X_{q_\ell(n),r_N}:\, M_N<n\leq N\}.
$$
Let the functions $H_1=H_{1,t,N}$ and $H_2=H_{2,t,N}$ be given by 
\begin{equation}\label{H1}
H_1(\{x_{q_j(n)}\})=\exp\left(it\sum_{n=1}^{M_N}G(x_{q_1(n)},x_{q_2(n)},...,x_{q_\ell(n)})\right)
\end{equation}
and 
\begin{equation}\label{H2}
H_2(\{x_{q_j(n)}\},\{z_{q_\ell(n)}\})=\exp\left(it\sum_{n=M_N+1}^{N}G(x_{q_1(n)},...,x_{q_{\ell-1}(n)},z_{q_\ell(n)})\right).
\end{equation}
Then $|H_1|=|H_2|=1$,
 $$
\exp\left(itS_{M_N,r_N}^{\{q_j\}}G\right)=H_1(V_1)
 $$
and
 $$
 \exp\left(it\big(S_{N,r_N}^{\{q_j\}}G-S_{M_N,r_N}^{\{q_j\}}G\big)\right)=H_2(V_2,V_3).
 $$

 Using these notations we have
 \begin{equation}\label{WEHAVE0}
\left|\bbE\big[\exp\big(itS_{N,r_N}^{\{q_j\}}G\big)\big]\right|=|\bbE[H_1(V_1)H_2(V_2,V_3)]|=\left|\bbE\left[H_1(V_1)\bbE\big[H_2(V_2,V_3)|V_1,V_2\big]\right]\right|
\end{equation}
$$
\leq \bbE\left[\left|\bbE\big[H_2(V_2,V_3)|V_1,V_2\big]\right|\right].
$$
 
Next, when $N$ large enough then by \eqref{A}, \eqref{B} and the definition of $r_N$, the random variable $(V_1,V_2)$ is $\cF_{-\infty,q_{\ell}(M_N)-2r_N}$ measurable, while $V_3$ is $\cF_{q_{\ell}(M_N)-r_N,\infty}$ measurable. Therefore, if we set $h_2(v)=\bbE[H_2(v,V_3)]$ then by
Lemma \ref{thm3.11-StPaper} we have
\begin{equation}\label{H rel 2}
\|\bbE[H_2(V_2,V_3)|V_1,V_2]-h_2(V_2)\|_{L^\infty}\leq \|\bbE[H_2(V_2,V_3)|\cF_{-\infty,q_{\ell}(M_N)-2r_N}]-h_2(V_2)\|_{L^\infty}
\end{equation}
$$
\leq 2\phi(r_N)=O(N^{-\al\te_1})=o(N^{-1/2})
$$
where we have used our assumption that $\al\te_1>3/2$, and the above estimate holds uniformly in $t\in J$.
Hence, by \eqref{WEHAVE0} we have
\begin{equation}\label{WEHAVE}
\left|\bbE\big[\exp\big(itS_{N,r_N}^{\{q_j\}}G\big)\big]\right|\leq \bbE[|h_2(V_2)|]+o(N^{-1/2}).
\end{equation}
Next, set 
$$
U_{n,r_N}(y;t)=\bbE\left[\exp\left(itG(y,X_{q_\ell(n),r_N})\right)\right].
$$
Set $\bar{X}_{n,r}=(X_{q_1(n),r},...,X_{q_{\ell-1}(n),r})$. 
Then, using definition of the function $H_2$, applying Corollary \ref{lem3.1-StPaper} with the collection of random variables $U_n=X_{q_\ell(n),r_N}$ for $M_N=[aN]<n\leq N$ and $i=1,2,...,\ell-1$ and the function $H=H_2(v,\cdot)$ (for any fixed $v$), taking into account \eqref{B}, uniformly in $t\in J$ we have
\begin{equation}
\sup_{v}\left|\bbE[H_2(v,V_3)]-\prod_{n=[aN]+1}^{N}U_{n,r_N}(\bar v_n;t)\right|=
O(N\phi(r_N))=o(N^{-1/2})
\end{equation}
where $v=\{v_{q_j(n)}:\,1\leq j<\ell,\,M_N<n\leq N\}$ and $\bar v_n=(v_{q_1(n)},v_{q_2(n)},...,v_{q_{\ell-1}(n)})$, and 
 we have used that $N\phi(r_N)=O(N^{1-\al\te_1})=o(N^{-1/2})$.
Plugging in $v=V_2$ we conclude that uniformly in $t\in J$ we have
\begin{equation}\label{CON0}
\left\|h_2(V_2)-\prod_{n=[aN]+1}^{N}U_{n,r_N}(\bar X_{n,r_{N}};t)\right\|_{L^\infty}=o(N^{-1/2}).
\end{equation}
Next, since $\mu$ is the distribution of $X_{q_\ell(n)}$, by \eqref{X n r }, \eqref{G Hold}, the mean value theorem and \eqref{kappa} we have
\begin{equation}\label{Con00}
\left|U_{n,r_N}(\bar X_{n,r_{N}};t)- \int\exp\left(it G(\bar X_{n,r_{N}},x)\right)d\mu(x)\right|=
\end{equation}
$$
\left|\bbE\left[\exp\left(itG(y,X_{q_\ell(n),r_N})\right)\right]-\bbE\left[\exp\left(itG(y,X_{q_\ell(n)})\right)\right]\right|_{y=\bar X_{n,r_N}}
$$
$$
\leq K|t|\|d(X_{q_\ell(n)},X_{q_\ell(n),r_N})^\ka\|_{L^1}\leq
|t|O(\be_2(r_N))^{\ka}=O(N^{-\al\te_2\ka})=o(N^{-3/2})
$$
where we have used our assumption that $\te_2>\frac{3}{2\al\ka}$.
We conclude from \eqref{CON0} and \eqref{Con00} that 
\begin{equation}\label{CON}
\left\|h_2(V_2)-\prod_{n=[aN]+1}^{N}\int \exp\left(it G(\bar X_{n,r_{N}},x)\right)d\mu(x)\right\|_{L^\infty}=o(N^{-1/2}).
\end{equation}

Finally, let us take $\ell-1$ independent copies of the collection of variables $X_{n}$ and $X_{n,r}$, where $n,r\in\bbN$ and denote them by $X_n^{(i)}$ and $X_{n,r}^{(i)}$, $i=1,2,...,\ell-1$. Set 
$$
\tilde X_{n,r_{N}}=\left(X_{q_1(n),r_N}^{(1)},X_{q_2(n),r_N}^{(2)},...,X_{q_{\ell-1}(n),r_N}^{(\ell-1)}\right).
$$
By Corollary \ref{lem3.1-StPaper} applied with $U_i=X_{q_i(n),r_N}$, taking into account  \eqref{A}, we see that when $N$ is large enough then, uniformly in $t\in J$ and $M_N<n\leq N$
 we have
\begin{equation}\label{W1}
\bbE\left[\left|\int \exp\left(it G(\bar X_{n,r_{N}},x)\right)d\mu(x)\right|\right]=\zeta_{n,r_N}(t)+O(\phi([\del_0 N/2]))=\zeta_{n,r_N}(t)+o(N^{-3/2})
\end{equation}
where
$$
\zeta_{n,r_N}(t)=\bbE\left[\left|\int \exp\left(it G(\tilde X_{n,r_{N}},x)\right)d\mu(x)\right|\right]
$$
and  we have used that $\phi(u)=O(u^{-\te_1})$ and $\te_1>4$. Set 
$$
\tilde X_{n}=\left(X_{q_1(n)}^{(1)},X_{q_2(n)}^{(2)},...,X_{q_{\ell-1}(n)}^{(\ell-1)}\right).
$$
Then, $\tilde X_{n}$ is distributed according to $\mu^{\ell-1}$ and so the function $\zeta(t)$ defined in \eqref{zeta t def} can also be written as
$$
\zeta(t)=\bbE\left[\left|\int \exp\left(it G(\tilde X_{n},x)\right)d\mu(x)\right|\right].
$$
Thus by \eqref{X n r }, which also holds true for the independent copies, \eqref{G Hold} and \eqref{kappa} we have 
\begin{equation}\label{W2}
|\zeta_{n,r_N}(t)-\zeta(t)|\leq\bbE\left[\int\left|\exp\left(it G(\tilde X_{n,r_{N}},x)\right)-\exp\left(it G(\tilde X_{n},x)\right)\right|d\mu(x)\right]
$$
$$
\leq |t|O(\be_2(r_N))^{\ka}=|t|O(N^{-\al\te_2\ka})=o(N^{-3/2})
\end{equation}
where we have used again that $|t|\leq C(J)$.
Hence, by  \eqref{WEHAVE}, \eqref{CON}, \eqref{W1} and \eqref{W2}  we get that
$$
\left|\bbE\big[\exp\big(itS_{N,r_N}^{\{q_j\}}G\big)\big]\right|\leq\bbE[|h_2(V_2)|]+o(N^{-1/2})$$
$$=\prod_{n=[aN]+1}^{N}\left(\zeta(t)+o(N^{-3/2})\right)
+o(N^{-1/2})=\left(\zeta(t)\right)^{N-[aN]}+o(N^{-1/2})
$$
and the proof of the lemma is complete, taking into account \eqref{Onee}.
\end{proof}

\begin{proof}[Proof of Theorem \ref{LCLT1} without linear indexes ($k=0$)]

Observe  that
\[
\zeta(t)=1-\frac12t^2\int G_\ell^2(x_1,x_2,...,x_\ell)d\mu(x_1)d\mu(x_2)...d\mu(x_\ell)+O(|t|^3)
\]
where $G_\ell$ is defined in \eqref{G ell}. 
 Therefore, since the function $G_\ell$ is not $\mu^\ell$-almost surely constant, there exist constants $c,c',\del_1>0$ so that for every $t\in[-\del_1,\del_1]$  we have $\zeta(t)\leq 1-c't^2\leq e^{-ct^2}$. Hence, for every $N\in\bbN$ and $t\in[-\del_1,\del_1]$  we have 
\[
(\zeta(t))^{c_0 N}\leq e^{-c_0cNt^2}. 
\]
This together with (\ref{One}) and Lemma \ref{ZeLem1} yields the validity of Assumption \ref{DecRate1} with $Z_N=S_N^{\{q_j\}}G$. In order to verify Assumption \ref{DecRate2} (in the non-arithmetic case) or Assumption \ref{DecRate3} (in the lattice case)
with $Z_N=S_N^{\{q_j\}}G$,  we first observe that the function $\zeta(t)$ is continuous. In  what we have designated as  the non-arithmetic case, we have $\zeta(t)=|\zeta(t)|<1$ for every nonzero $t$, for otherwise (\ref{NonLat}) would have hold true with some $\be(\cdot)$. Therefore for every compact set $J\subset\bbR\setminus\{0\}$ we have $\sup_{t\in J}\zeta(t)<1$,
which together with \eqref{Conc0} yields that 
$$
\sup_{t\in J}\left|\bbE[\exp(it S_N^{\{q_j\}}G)]\right|=o(N^{-1/2})
$$
and thus Assumption \ref{DecRate2} holds true. 
In the lattice case, for all $t\in[-\pi,\pi]\setminus\{0\}$ we have $\zeta(t)<1$ and so for  every compact set $J\subset[-\pi,\pi]\setminus\{0\}$ we have $\sup_{t\in J}\zeta(t)<1$, which yields Assumption \ref{DecRate3}.
\end{proof}

\subsection{Linear and nonlinear indexes: proof of Theorems \ref{LCLT1} and \ref{LCLT2} when $0<k<\ell$}\label{sec k}
We assume here that  the polynomials $q_k,...,q_\ell$ satisfy \eqref{NonLinGr} and \eqref{NonLinGr0} for some $0<k<\ell$ and that $q_j(n)=jn$ for all $1\leq j\leq k$ and $n\in\bbN$.
Since the CLT holds, as in the previous section, our goal here is to verify  Assumptions \ref{DecRate1} and \ref{DecRate2} (in the non-arithmetic case) or \ref{DecRate3} with $h_0=1$ (in the lattice case) appearing in Theorem \ref{General LLT } with $Z_N=S_N^{\{q_j\}}G$.
The proof shares some similarities with the proof in the case when all the functions grow faster than linearly, but there are additional complications because of the linear indexes. The main difference is that we cannot pass to independent copies of $X_{q_i(n)}$ for $M_N=[aN]<n\leq N$ and $i=1,2,...,k$ because of the linearity of $q_i$. 	Instead, after a conditioning argument, using \eqref{A} we will show that we can pass to independent copies of $\{X_{q_i(n)}:\,M_N<n\leq N\},\,i=1,2,...,k$. The sequence $\{Y_n\}$ from \eqref{Re1} will be defined using these copies.

\subsubsection{Some preparations}
Let $X_{r,n}$ be as specified before \eqref{X n r }.
 By considering the space $\Om^k=\Om\times\Om\times\dots\times\Om$, we get  copies $\{X_{n}^{(i)}\}$ and $\{X_{n,r}^{(i)}\}$ of $\{X_{n}\}$ and $\{X_{q_i(n),r}\}$, respectively, which are independent of each other (when $i=1,2,...,k$), so that for all $n,r\in\bbN$ and $1\leq i\leq k$ we have 
\begin{equation}\label{AppR}
\left\|d\left(X_{n}^{(i)},X_{n,r}^{(i)}\right)\right\|_{L^2}\leq 2\be_2(r).
\end{equation}
Next,  set
\begin{equation}\label{Y n}
Y_n=(X^{(1)}_{n},X^{(2)}_{2n},...,X^{(k)}_{kn}).
\end{equation}
Then the joint distribution of $Y_n$ and $Y_m$ depends only on $m-n$, and  all $Y_n$'s are distributed according to $\mu\time\mu\times\dots\times\mu=\mu^{\ell-1}$. 
We also set 
\begin{equation}\label{Y nr}
Y_{n,r}=(X^{(1)}_{n,r},X^{(2)}_{2n,r},...,X^{(k)}_{kn,r}).
\end{equation}
The mixing properties of $\{Y_{n,r}\}$ needed for the verification of Assumptions \ref{DecRate1}, \ref{DecRate2} and \ref{DecRate3} are specified in the following.
\begin{lemma}\label{Mix Y}
Let $n_1<n_2<...<n_m$ be positive integers and let $r\in\bbN$ be so that $n_{s+1}-n_s>2r$ for all $s=1,2,...,m-1$. Let $f_1,....,f_m:\cX^{k}\to[-1,1]$ be measurable functions. Then 
$$
\left|\bbE\left[\prod_{s=1}^m f_s(Y_{n_s,r})\right]-\prod_{s=1}^m\bbE[f_s(Y_{n_s,r})]\right|\leq 4\sum_{j=1}^{k}\sum_{s=1}^{m-1}\phi(j(n_{s+1}-n_s-2r)).
$$
\end{lemma}
\begin{proof}
We will prove the lemma  by induction on $k$. For $k=1$ the lemma follows from Corollary \ref{lem3.1-StPaper} applied with $U_s=X_{n_s,r}$ and  the partition $\cC=\{s\}$. 

Let us assume that the lemma is true for some $k$ and all possible choices of $n_1,...,n_m$, $r$ and $f_1,...,f_m$ as specified in the lemma. Let $n_1<n_2<...<n_m$ and let $r\in\bbN$ be so that $n_{s+1}-n_s> 2r$. Let $f_1,....,f_m:\cX^{k+1}\to[-1,1]$ be measurable functions. 

Let $\{X_{n,r}^{(j)}:\, n\in\bbN\}$, $j=1,2,...,k,k+1$ be $k+1$ independent copies of $\{X_{n,r}:\,n\in\bbN\}$.
 Set 
$$Y_{k,n_s,r}=(X_{n_s,r}^{(1)},X_{2n_s,r}^{(2)},...,X_{k n_s,r}^{(k)}).$$
Then 
$$
Y_{n_s,r}=Y_{k+1,n_s,r}=(Y_{k,n_s,r},X_{(k+1)n_s,r}^{(k+1)}).
$$
By conditioning on  $\{Y_{k,n_s,r}:\,1\leq s\leq m\}$ we see that 
$$
\bbE\left[\prod_{s=1}^m f_s(Y_{n_s,r})\right]=\bbE\left[\tilde f\left(\{Y_{k,n_s,r}\}\right)\right]
$$
where 
$$
\tilde f(\{y_{k,n_s,r}\})=\bbE\left[\prod_{s=1}^m f_s(y_{k,n_s,r},X_{(k+1)n_s,r}^{(k+1)})\right]=
\bbE\left[\prod_{s=1}^m f_s(y_{k,n_s,r},X_{(k+1)n_s,r})\right].
$$
Now, by Lemma \ref{lem3.1-StPaper} applied with $U_s=X_{(k+1)n_s,r}$ and  the partition $\cC=\{s\}$, for a fixed realization $\{y_{k,n_s,r}\}$ of $\{Y_{k,n_s,r}\}$ we have
$$
\left|\tilde f(\{y_{k,n_s,r}\})-\prod_{s=1}^m\bbE[f_s(y_{k,n_s,r},X_{(k+1)n_s,r})]\right|\leq 4\sum_{s=1}^{m-1}\phi\big((k+1)(n_{s+1}-n_s-2r)\big).
$$
Thus, if we define $\tilde f_s:\cX^{k}\to[-1,1]$ by
$$
\tilde f_s(x_1,...,x_k)=\bbE[f_s(x_1,...,x_k,X_{(k+1)n_s,r})]=\bbE[f_s(x_1,...,x_k,X_{(k+1)n_s,r}^{(k+1)})]
$$
then
$$
\left|\bbE\left[\prod_{s=1}^m f_s(Y_{n_s,r})\right]-\bbE\left[\prod_{s=1}^m \tilde f_s(Y_{k,n_s,r})\right]\right|\leq 4\sum_{s=1}^{m-1}\phi\big((k+1)(n_{s+1}-n_s-2r)\big).
$$
To complete the induction we use now the induction hypothesis with the functions $\tilde f_s$.
\end{proof}

Next, consider the family of functions $\zeta(y,\cdot):\bbR\to[0,1]$, $y=(y_1,...,y_k)\in\cX^k$ given by 
\begin{equation}\label{zeta y def}
\zeta(y,t)=\int\left|\int \exp(it G(y,x_{k+1},...,x_\ell))d\mu(x_\ell)\right|d\mu(x_{k+1})\cdots d\mu(x_{\ell-1}).
\end{equation}
The following result follows from the definitions of $\zeta(y_1,...,y_k,t)$ and $Y_n$ and  the H\"older continuity of $G$ (i.e. \eqref{G Hold}).
\begin{lemma}
(1) For every $n\in\bbN$ and $t\in\bbR$ we have
 \begin{equation}\label{Expec}
\bbE[\zeta(Y_n,t)]=\zeta(t)
\end{equation}
where $\zeta(t)$ was defined in (\ref{zeta t def}).

(2)  For all $y=(y_i)$ and $y'=(y'_i)$ in $\cX^{k}$ we have
\begin{equation}\label{Hold}
|\zeta(y,t)-\zeta(y',t)|\leq C|t|\sum_{j=1}^{k}\left(d(y_j,y'_j)\right)^{\ka}
\end{equation}
where $C>0$ is some constant.
\end{lemma}

The following lemma is a consequence of \eqref{Hold}, \eqref{AppR} and \eqref{kappa}.

\begin{lemma}
There exists a constant $A>0$ so that for every $n,r\in\bbN$, $1\leq i\leq k$ and $t\in\bbR$ we have
\begin{equation}\label{L1}
\|\zeta(Y_{n},t)-\zeta(Y_{n,r},t)\|_{L^1}\leq A|t|\left(\beta_2([r/2))\right)^\ka.
\end{equation}
\end{lemma}

\begin{corollary}
For every positive integers $q_1<q_2$, $r\in\bbN$ and $t\in\bbR$ we have 
\begin{equation}\label{L2}
\left|\bbE\left[\prod_{n=q_1+1}^{q_2}\zeta(Y_n,t)\right]-\bbE\left[\prod_{n=q_1+1}^{q_2}\zeta(Y_{n,r},t)\right]\right|\leq A|t|(q_2-q_1)\left(\beta_1([r/2))\right)^\ka.
\end{equation}
\end{corollary}
\begin{proof}
For any complex numbers $\al_j,\beta_j,\,q_1<j\leq q_2$ so that $|\alpha_j|,|\beta_j|\leq1$ we have 
$$
\left|\prod_{j=q_1+1}^{q_2}\al_j-\prod_{j=q_1+1}^{q_2}\be_j\right|\leq\sum_{j=q_1+1}^{q_2}|\al_j-\be_j|.
$$
Applying this with $\al_j=\zeta(Y_n,t)$ and $\be_j=\zeta(Y_{n,r},t)$ and then using \eqref{L1} we obtain \eqref{L2}.
\end{proof}

\subsubsection{The conditioning step}
The first step of the proof of Theorems \ref{LCLT1} and \ref{LCLT2} when $k$ from \eqref{NonLinGr} and \eqref{NonLinGr0} is positive  and $q_1,...,q_k$ satisfy \eqref{kLin} is the following lemma.

\begin{lemma}[Conditioning step with linear and nonlinear indexes]\label{ZeLem2}
Let $a\in(0,1)$ be the number from Lemma \ref{Qlem} and set $M_N=[aN]$. 
Then for every compact set $J\subset\bbR$ there exists a sequence $(b_N)$ so that $b_N=o(N^{-1/2})$ and for all $N\in\bbN$ and $t\in J$ we have
\begin{equation}\label{Conc0}
\left|\bbE\big[\exp\big(itS_N^{\{q_j\}}G\big)\big]\right|\leq b_N+\bbE\left[\prod_{M_N+1}^{N}\zeta(Y_n,t)\right].
\end{equation}
\end{lemma}

\begin{proof}
Let us fix some compact set $J\subset\bbR$ and $N\in\bbN$.
Let $\del_0,a\in(0,1)$ satisfy \eqref{A} and \eqref{B}. As in the proof of Lemma \ref{ZeLem1}, let us also set  $r=r_N=[aN^\al]/4$.
 Then, by \eqref{SNr approx}, 
\begin{equation}\label{One}
\left|\bbE\big[\exp\big(itS_N^{\{q_j\}}G\big)\big]-\bbE\big[\exp\big(itS_{N,r_N}^{\{q_j\}}G\big)\big]\right|
\end{equation}
$$
\leq |t|\|S_N^{\{q_j\}}G-S_{N,r_N}^{\{q_j\}}G\|_{L^1}\leq C'|t| N^{1-\al\te_2\ka}=o(N^{-1/2})
$$
where we have used our assumption that $\al\te_2\ka>3/2$.
Similarly to the proof of Lemma \ref{ZeLem1}, 
let us consider the random variables $V_1,V_2$ and $V_3$ given by
$$
V_1=V_{1,N}=\{X_{q_i(n),r_N}:\,\,1\leq i\leq \ell,\,1\leq n\leq M_N\},
$$
$$
V_2=V_{2,N}=\{X_{q_i(n),r_N}:\,\,1\leq i<\ell, M_N<n\leq N\}
$$
and 
$$
V_3=V_{3,N}=\{X_{q_\ell(n),r_N}:\,\, M_N<n\leq N\}.
$$
Let $H_1=H_{1,t,N}$ and $H_2=H_{2,t,N}$ be the functions given by \eqref{H1} and \eqref{H2}. Then
 $|H_1|=|H_2|=1$,
 $$
\exp\left(itS_{M_N,r_N}^{\{q_j\}}G\right)=H_1(V_1)
 $$
and
 $$
 \exp\left(it\big(S_{N,r_N}^{\{q_j\}}G-S_{M_N,r_N}^{\{q_j\}}G\big)\right)=H_2(V_2,V_3).
 $$
Proceeding similarly to the proof of Lemma \ref{ZeLem1}, relying on \eqref{A} and \eqref{B} we have
\begin{equation}\label{1.1}
\left|\bbE\big[\exp\big(itS_{N,r_N}^{\{q_j\}}G\big)\big]\right|=|\bbE[H_1(V_1)H_2(V_2,V_3)]|=\left|\bbE\left[H_1(V_1)\bbE[H_2(V_2,V_3)|V_1,V_2]\right]\right|
\end{equation}
$$=|\bbE\left[H_1(V_1)h_2(V_2)\right]|+o(N^{-1/2})\leq \bbE[|h_2(V_2)|]+o(N^{-1/2})$$
where $h_2(v_2)=h_{2,t,N}(v_2)=\bbE[H_2(v_2,V_3)].$
Moreover, uniformly in $t\in J$ we have
\begin{equation}\label{h2Def}
\left\|h_2(V_2)-\prod_{n=M_N+1}^{N}\int\exp\left(itG(\bar X_{n,r_N},x)\right)d\mu(x)\right\|_{L^\infty}=o(N^{-1/2})
\end{equation}
where $\bar X_{n,r_N}=(X_{q_1(n),r_N},X_{q_2(n),r_N},...,X_{q_{\ell-1}(n),r_N})$.

Now, as explained at the beginning of Section \ref{sec k}, the difference in comparison to Lemma \ref{ZeLem1} is that the random variables in the definition of $V_2$ are not close to being independent (because of the linear indexes). Let us write 
$$\bar X_{n,r_N}=(\check X_{n,r_N},\hat X_{n,r_N})$$
where
 $$\check X_{n,r_N}=(X_{n,r_N},X_{2n,r_N},...,X_{kn,r_N})$$ and $\hat X_{n,r_N}=(X_{q_{k+1}(n),r_N},...,X_{q_{\ell-1}(n),r_N})$.
By applying Corollary \ref{lem3.1-StPaper} with the collection of random variables $\{U_l\}$ whose members are $\{\check X_{n,r_N}: M_N<n\leq N\}$ and $\{X_{q_j(n),r_N}\}$, $M_N<n\leq N$, $j=k+1,...,\ell$  together the trivial partition $\cC_l=\{l\}$, and taking into account \eqref{A} and \eqref{B}
we get that
\begin{equation}\label{U0}
\bbE\left[\prod_{n=M_N+1}^{N}\left|\int\exp\left(itG(\check X_{n,r_N},\hat X_{n,r_N}, x)\right)d\mu(x)\right|\right]
\end{equation}
$$=\bbE\left[\prod_{n=M_N+1}^{N}\zeta(\check X_{n,r_N},t)\right]+O(N\phi(r_N)).
$$
Since $\phi(r_N)=O(N^{-\te_1\al})$ and $\al\te_1>3/2$, by \eqref{MixA} we have $O(N\phi(r_N))=o(N^{-1/2})$, and so we can  disregard this term.
Similarly,  applying Corollary \ref{lem3.1-StPaper} with $U_i=\{X_{in,r_N}:M_N<n\leq N\},\,i=1,2,...,k$ and with the partition $\cC_{i}=\{i\}$, and taking into account \eqref{A}, we
have
\begin{equation}\label{U1}
\bbE\left[\prod_{n=M_N+1}^{N}\zeta(\check X_{n,r_N},t)\right]=
\bbE\left[\prod_{n=M_N+1}^{N}\zeta(Y_{n,r_N},t)\right]+O(\phi([\del_0 N])).
\end{equation}
Here $Y_{n,r}$ is given by \eqref{Y nr}.
Finally, by \eqref{L2} we have
\begin{equation}\label{1.3}
\bbE\left[\prod_{n=M_N+1}^{N}\zeta(Y_{n,r_N},t)\right]=\bbE\left[\prod_{n=M_N+1}^{N}\zeta(Y_{n},t)\right]+|t|O(N^{1-\te_2\ka})=o(N^{-1/2})
\end{equation}
where we have used that $\te_2\ka\geq\te_2\ka\al>3/2$.
Since $\phi([\del_0 N])=o(N^{-1/2})$
 the proof of the lemma is completed by \eqref{1.1}, \eqref{h2Def} and \eqref{1.3}, \eqref{U0} and \eqref{U1}.
\end{proof}

\subsubsection{Verification of Assumption \ref{DecRate1}}
After establishing Lemma \ref{ZeLem2}, 
the second ingredient needed to verify  Assumption \ref{DecRate1} for $Z_N=S_N^{\{q_j\}}G$ is the following.
\begin{lemma}\label{A set Lemma 1}
There exist constants $c_3,N_3,\del_3>0$ and measurable sets $B_N\subset(\cX^{k})^{N-M_N}$ so that 
$$
\lim_{N\to\infty}\sqrt{N}\bbP((Y_{M_N+1},...,Y_N)\in B_N)=0
$$
\footnote{Henceforth we will use $\bbP$ as a generic notation for the probability of sets of the form $\{Z\in A\}$ (writing $\bbP(Z\in A)$) where $Z$ is some random variable and $A$ is some measurable set.}and for all $N\geq N_3$ and $t\in[-\del_3,\del_3]$, 
when $(Y_{M_N+1},...,Y_N)\notin B_N$ we have
$$
\prod_{n=M_N+1}^{N}\zeta(Y_n,t)\leq e^{-c_3Nt^2}.
$$
Therefore, for all $N\geq N_3$ and $t\in[-\del_3,\del_3]$ we have
\begin{equation}\label{Conc}
\bbE\left[\prod_{n=M_N+1}^{N}\zeta(Y_n,t)\right]\leq 
b_N+e^{-c_3Nt^2}
\end{equation}
where $b_N=o(N^{-1/2})$.
\end{lemma} 

\begin{corollary}
The sequence $Z_N=S_N^{\{q_j\}}G$ satisfies the conditions of
Assumption \ref{DecRate1}.
\end{corollary}
\begin{proof}
This is a direct consequence of \eqref{Conc0} and \eqref{Conc}.
\end{proof}

\begin{proof}[Proof of Lemma \ref{A set Lemma 1}]
Since $G$ is a bounded function, 
uniformly in $y=(y_1,...,y_k)\in\cX^k$ and $\bar x=(x_k,...,x_{\ell-1})\in\cX^{\ell-k-1}$ we have
$$
\left|\int\exp\left(itG(y,\bar x,x_\ell)\right)d\mu(x_\ell)\right|=1-\frac{t^2}2\int G_\ell^2(y,\bar x,x_\ell) d\mu(x_\ell)+O(|t|^3)
$$
where $G_\ell(y,\bar x,x_\ell)= G(y,\bar x,x_\ell)-\int G(y,\bar x,z)d\mu(z)$ (which was also defined in \eqref{G ell}).
Therefore, uniformly in $y$ we have
\begin{equation}\label{Al}
\zeta(y,t)=1-\frac12t^2\int G_\ell^2(y,x_{k+1},...,x_{\ell})d\mu(x_{k+1})\dots d\mu(x_\ell)+O(|t|^3).
\end{equation} 
Consider the sequence of random variables $\{H_n\}$ given by
\[
H_n=\int G_\ell^2(Y_n,x_{k+1},...,x_{\ell})d\mu(x_{k+1})\dots d\mu(x_\ell):=H(Y_n).
\]
Then $H_1,H_2,...$ are equally distributed. Moreover, since $G$ is not $\mu^\ell$-almost surely a function of the first $\ell-1$ variables $x_1,...,x_{\ell-1}$, the function $G_\ell$ does not vanish $\mu^\ell$-almost surely (where $\mu^\ell=\mu\times\mu\times\dots\times\mu$), and therefore
\[
\bbE[H_n]=\int G_\ell^2(x_1,x_2,...,x_\ell)d\mu(x_1)d\mu(x_2)\dots d\mu(x_\ell):=v>0.
\]
The idea behind the proof of the lemma is to show that with sufficiently high probability we can replace $H_n$ with its expectation, that is we will prove a certain type of concentration inequality involving the variables $H_n$. Using it, and taking into account \eqref{Al}, with high probability we can replace $\zeta(Y_n,t)$ with $1-t^2v'+O(|t|^3)$ for some $0<v'<v$. Then when $|t|$ is small enough we will multiply the variables $\zeta(Y_n,t)$ for $n=M_N+1,...,N$.

In order to formalize the above idea, we first set  $\ve=P(H_{n}\geq v/2)/2>0$
 and for every $N\in\bbN$ set 
\[
A_N=\left\{\sum_{n=M_N+1}^N\bbI(H_n\geq v/4)\leq \ve(N-M_N)/2\right\}
\]
and
\[
B_N=\left\{(\bar y_{M_N+1},...,\bar y_N)\in(\cX^{k})^{N-M_N}:\,\sum_{n=M_N+1}^N\bbI(H(\bar y_n)\geq v/4)\leq \ve(N-M_N)/2\right\}.
\] 
Then $A_N=\{(Y_{M_N+1},...,Y_N)\in B_N\}$.
Now, by \eqref{Al}, taking into account that $0\leq\zeta(y,t)\leq 1$ and  that $M_N=[aN]$, $a\in(0,1)$ we obtain that there are positive constants $c_3,\del_3$ and $N_3$ so that for all $N\geq N_3$ and $t\in[-\del_3,\del_3]$,
on the complement of $A_{N}$ (i.e. when $(Y_{M_N+1},...,Y_N)\not\in B_N$) we have 
\begin{equation}\label{On A comp}
\prod_{n=M_N+1}^{N}\zeta(Y_n,t)\leq \left(1-t^2v/8+O(|t|^3)\right)^{\ve(N-M_N)/2}
\end{equation}
$$
=\left(1-t^2v/8+O(|t|^3)\right)^{\ve(N-M_N)/2}\leq \left(1- t^2v/9\right)^{(N-M_N)\ve/2}\leq e^{-c_3 Nt^2}.
$$
Hence, 
$$
\bbE\left[\prod_{n=M_N+1}^{N}\zeta(Y_n,t)\right]\leq \bbP(A_{N})+e^{-c_3 Nt^2}.
$$
We conclude that the lemma will follow if we show that 
\begin{equation}\label{Accom}
\bbP(A_N)=\bbP((Y_{M_N+1},...,Y_N)\in B_N)=o(N^{-1/2}).
\end{equation}

To prove \eqref{Accom},  applying the Markov inequality and then using \eqref{AppR} for every $\del>0$ we have
\begin{equation}\label{Y approx}
\bbP(d(Y_n,Y_{n,r})\geq\del)\leq\frac{\bbE[d(Y_n,Y_{n,r})]}{\del}\leq\frac{2k\be_2(r)}{\del}\leq\frac{2kcr^{-\te_2}}{\del}
\end{equation}
where $Y_{n,r}$ was defined in \eqref{Y nr}, and the distance on $\cX^{k}=\cX\times\cX\times\cdots\times\cX$ is given by $d(x,y)=\sum_{i=1}^k d(x_i,y_i)$, where $x=(x_i)$ and $y=(y_i)$.
The idea in the proof of \eqref{Accom} is to replace $Y_n$ with $Y_{n,r}$, for some $r=r_N$ relying on \eqref{Y approx}, and then to use the mixing properties of $\{Y_{n,r}:\,n\geq1\}$ from Lemma \ref{Mix Y}.
For this purpose, we first set
\[
H_{n,r}=\int G_\ell^2(Y_{n,r},x_{k+1},...,x_{\ell})d\mu(x_{k+1})\dots d\mu(x_\ell)=H(Y_{n,r}).
\]
Since $H$ is a bounded H\"older continuous function, $v=\bbE[H_n]$ and $\bbP(H_{n}\geq v/2)=2\ve$, using \eqref{Y approx} we see that 
\begin{equation}\label{liminf}
\liminf_{r\to\infty}\inf_{n}\bbP(H_{n,r}\geq v/3)>\ve.
\end{equation}
Set
\[
A_{N,r}=\left\{\sum_{n=M_N+1}^N\bbI(H_{n,r}\geq v/3)\leq \ve(N-M_N)/2\right\}.
\]
Let $\del$ be small enough so that $|H(y)-H(y')|<v/3-v/4=v/12$ if $d(y,y')<\del$.
Then by (\ref{Y approx}), taking into account that $H$ 
is H\"older continuous, if $\del$ is small enough then for all $r$ and $N$ we have
\begin{equation}\label{A N r}
\bbP(A_N)\leq \bbP(A_{N,r})+\sum_{n=M_N+1}^{N}\bbP(d(Y_n,Y_{n,r})\geq\del)\leq \bbP(A_{N,r})+O(Nr^{-\te_2}).
\end{equation}
Next, set $W_{n,r}=\bbI(H_{n,r}\geq v/3)$. Let us assume that $r$ is large enough so that
$\bbE[W_{n,r}]=\bbP(H_{n,r}\geq v/3)\geq \ve$ for all $n$ (recall \eqref{liminf}). Then
\begin{equation}\label{PRO.0}
\bbP(A_{N,r})\leq \bbP\left\{\sum_{n=M_N+1}^N(W_{n,r}-\bbE[W_{n,r}])\geq(N-M_N)\ve/2 \right\}.
\end{equation}
We claim that 
\begin{equation}\label{PRO}
\bbP\left\{\sum_{n=M_N+1}^N(W_{n,r}-\bbE[W_{n,r}])\geq (N-M_N)\ve/2 \right\}\leq \frac{4C_1r}{\ve^2(N-M_N)}
\end{equation}
where $C_1>0$ is some constant which does not depend on $r$ and $N$. Let us complete the proof of the lemma relying on \eqref{PRO}.
Since $\te_2>3$, there exists $q\in(0,\frac12)$ so that $q\te_2>\frac32$. Let us take $r=r_N=[N^q]$. Then the second term $O(Nr^{-\te_2})$ on the right hand side of \eqref{A N r}  is $o(N^{-1/2})$ and the right hand side of \eqref{PRO} is $o(N^{-1/2})$. The proof of Lemma \ref{A set Lemma 1} is completed now by \eqref{A N r} and \eqref{PRO.0}  applied with $r=r_N$.

Finally, let us prove \eqref{PRO}. By an application of the Markov inequality, it is enough to show that
\begin{equation}\label{Vst}
\text{Var}\left(\sum_{n=M_N+1}^NW_{n,r}\right)\leq C_1rN
\end{equation}
for some constant $C_1$ which does not depend on $N$ and $r$. The above inequality  holds true since by  Lemma \ref{Mix Y}  for all $s\geq 1$ and $r\in\bbN$ we have
$$
\text{Cov}(W_{n,r},W_{n+2r+s,r})\leq C\phi(s)\leq C's^{-\te_1}
$$
and $\sum_{s}s^{-\te_1}<\infty$ (since $\te_1>4$), where $C,C'>0$ are some constants.
\end{proof}

\subsubsection{Verification of Assumptions \ref{DecRate2} or Assumption \ref{DecRate3}}
The following result is the additional ingredient needed to verify either  Assumption \ref{DecRate2} or  Assumption \ref{DecRate3}   with $Z_N=S_N^{\{q_j\}}G$.
\begin{lemma}\label{LastLemma}
In the lattice case, let $J\subset[-\pi,\pi]\setminus\{0\}$ be a compact set, while in the  non-arithmetic case let $J\subset\bbR\setminus\{0\}$ be a compact set. Then, in both cases
\[
\lim_{N\to\infty}\sqrt{N}\sup_{t\in J}\bbE\left[\prod_{n=M_N+1}^{N}\zeta(Y_n,t)\right]=0.
\]
\end{lemma}

\begin{corollary}
In the non-arithmetic case, the sequence $Z_N=S_N^{\{q_j\}}G$ verifies the conditions of
Assumption \ref{DecRate2}, while in the lattice case it verifies the conditions of Assumption \ref{DecRate3} with $h_0=1$.
\end{corollary}
\begin{proof}
This is a direct consequence of Lemmas \ref{ZeLem2} and \ref{LastLemma}.
\end{proof}

\begin{proof}[Proof of Lemma \ref{LastLemma}]
In the non-arithmetic case, let $J\subset\bbR\setminus\{0\}$ be a compact set, while in the lattice case let $J	\subset[-\pi,\pi]\setminus\{0\}$ be a compact set. Let us fix some $t\in J$.
We first note that the Assumptions in Theorems \ref{LCLT1} and \ref{LCLT2} imply that $\te_2>\frac{3}{2\te\al}\geq\frac{3}{2\ka}$.  
Let us take $\frac{3}{2\te_2\ka}<b<1$. Since $\te_1\geq \te_2\ka$ by taking $b$ close enough to $\frac{3}{2\te_2\ka}$ we can also insure that $\te_1b>3/2$. Let us also set
$r_N=[N^{b}]$. Then $Nr_N^{-\te_2\ka}=o(N^{-1/2})$ and so by \eqref{L2} we have
\begin{equation}\label{L2AP}
\left|\bbE\left[\prod_{n=M_N+1}^{N}\zeta(Y_{n,r_N},t)\right]-\bbE\left[\prod_{n=M_N+1}^{N}\zeta(Y_{n},t)\right]\right|\leq|t|b_N\leq C(J)b_N
\end{equation}
where $b_N=O(Nr_N^{-\te_2\ka})=o(N^{-1/2})$ and $C(J)=\max\{|t|: t\in J\}<\infty$.
 Next, since $\zeta(y,t)\in[0,1]$ we have 
\begin{equation}\label{L3AP}
\bbE\left[\prod_{n=M_N+1}^{N}\zeta(Y_{n,r_N},t)\right]\leq 
\bbE\left[\prod_{j=1}^{[(N-M_N)/3r_N]}\zeta(Y_{M_N+3jr_N,r_N},t)\right].
\end{equation}
Now, by Lemma \ref{Mix Y} we have
\begin{equation}\label{L4AP} 
\bbE\left[\prod_{j=1}^{[(N-M_N)/3r_N]}\zeta(Y_{M_N+3jr_N,r_N},t)\right]=\prod_{j=1}^{[(N-M_N)/3r_N]}\bbE\left[\zeta(Y_{M_N+3jr_N,r_N},t)\right]+O(N\phi(r_N))
\end{equation}
$$
=\prod_{j=1}^{[(N-M_N)/3r_N]}\bbE\left[\zeta(Y_{M_N+3jr_N,r_N},t)\right]+o(N^{-1/2})
$$
where we have used that $N\phi(r_N)=O(N^{1-\te_1 b})=o(N^{-1/2})$. 

Finally, by  \eqref{L1} we have 
\begin{equation}\label{L5AP}
\prod_{j=1}^{[(N-M_N)/3r_N]}\bbE\left[\zeta(Y_{M_N+3jr_N,r_N},t)\right]=
\prod_{j=1}^{[(N-M_N)/3r_N]}\left(\bbE\left[\zeta(Y_{M_N+3jr_N},t)\right]+O(r_N^{-\te_2\ka})\right)
\end{equation}
$$
=\left(\zeta(t)\right)^{[(N-M_N)/3r_N]}+O(Nr_N^{-\te_2\ka})\leq \left(\zeta(t)\right)^{N^{c_0}}+o(N^{-1/2})
$$
where $c_0\in(0,1)$ is some constant.  
As in the proof of Theorems \ref{LCLT1} and \ref{LCLT2} in the absence of linear indexes, in both lattice and non-arithmetic cases we have $\zeta(t)<1$ for all $t\in J$. Since $\zeta(\cdot)$ is continuous we have
\begin{equation}\label{SuP}
\sup_{t\in J}\zeta(t)<1.
\end{equation}
The proof of the lemma is completed by successively applying \eqref{L2AP}-\eqref{SuP}.
\end{proof}

\section{Application to Bernoulli shfits}\label{BS}
 Let $\epsilon=\{\epsilon_j:\,j\in\bbZ\}$ be a sequence of iid random variables taking values in some measurable space $\cE$, which are defined on some probability space $(\Om,\cF,\bbP)$. 
  Let $(\cX,d)$ be a metric space and $g:\cE^\bbZ\to\cX$ be a measurable function. We consider here a stationary sequence of random variables $X_n:\Om\to\cX$ of the form  
\begin{equation}\label{BerFor}
X_n=g(...,\epsilon_{n-1},\epsilon_n,\epsilon_{n+1},...).
\end{equation}
Sequences of this form have been studied extensively in weak dependence theory, see \cite{Bill,IL}.
Next, let us fix some $r\in\bbN$ and take an independent copy $\ve'$ of $\ve$. Let us define 
$$
X_{n,r}=g(...,\ve'_{n-r-2},\ve'_{n-r-1},\ve_{n-r},...,\ve_{n-1},\ve_n,\ve_{n+1},...,\ve_{n+r},\ve'_{n+r+1},\ve'_{n+r+2},...).
$$
\begin{proposition}\label{MixProp}
After enlarging the probability space $(\Om,\cF,\bbP)$, there exists a family of $\sig$-algebras $\cF_{n,m}\subset\cF$ satisfying the conditions of Section \ref{sec2} with $\phi(n)=0$ and $$\beta_{p}(r)\leq \sup_{n}\|d(X_n,X_{n,r})\|_{L^p}$$
for all $r\in\bbN$ and $p\geq1$.
\end{proposition}

\begin{proof}
After enlarging the probability space we can assume that it also  supports independent copies 
$\epsilon^{(n)}=\{\epsilon_{k}^{(n)}:\,k\in\bbZ\},\,n\in\bbZ$ of $\epsilon$. 
For instance, this can be done by considering the probability space $\Om\times\Om^{\bbZ}$ and viewing $\epsilon$ and $\epsilon^{(n)}$ as functions of the appropriate $\Om$-directions.
 Let $\cF_{n,m}$ be the $\sig$-algebra generated by $\epsilon_{n},\epsilon_{n+1},...,\epsilon_{m}$ and $\epsilon^{(s)}$ for $n\leq s\leq m$. Then $\cF_{n,m}$ and $\cF_{n',m'}$ are independent if $n'>m$. Thus $\phi(n)=0$ for every $n\in\bbN$.
Moreover, set
\begin{equation}\label{bold X}
\textbf{X}_{n,r}=g(...,\epsilon^{(n)}_{n-r-2},\epsilon^{(n)}_{n-r-1},\epsilon_{n-r},...,\epsilon_{n-1},\epsilon_n,\epsilon_{n+1},...,\epsilon_{n+r},\epsilon^{(n)}_{n+r+1},\epsilon^{(n)}_{n+r+2},...).
\end{equation}
Then $\textbf{X}_{n,r}$ is $\cF_{n-r,n+r}$ measurable and it has the same distribution as $X_{n,r}$. Therefore, 
$$
\beta_p(r)\leq\sup_n\|d(X_n,\textbf{X}_{n,r})\|_{L^p}=\sup_n\|d(X_n,X_{n,r})\|_{L^p}.
$$
\end{proof}

\begin{remark}\label{Rem8.2}
If $Y_{n,k}$, $k\in\bbZ$ is obtained by replacing  the coordinate at place $n+k$ by $\ve'_{n+k}$ then 
\begin{equation}\label{Triangle}
\|d(X_n,X_{n,r})\|_{L^p}\leq\sum_{|k|\geq r}\|d(X_n,Y_{n,k})\|_{L^P}
\end{equation}
and so, with $\beta_p(r)$ defined by \eqref{beta} we have
\begin{equation}\label{beta2}
\beta_p(r)\leq \sup_{n}\max\left(\|d(X_n,Y_{n,r})\|_{L^p}, \|d(X_n,Y_{n,-r})\|_{L^p}\right):=\tilde\be_p(r).
\end{equation}
Thus, we can control the decay rate of $\be_p(r)$ as $r\to\infty$ in terms of the more familiar approximation rates $\tilde\be_p(k)=\sup_{n}\|d(X_n,Y_{n,k})\|_{L^P}$.
\end{remark}


\section{A nonconventional LLT with linear indexes}\label{Arith}
As in the previous section, Theorem \ref{LCLT2} will follow once we verify Assumptions \ref{DecRate1} and 
\ref{DecRate2} (in the non-arithmetic case) or Assumption \ref{DecRate3} with $h_0=1$ (in the lattice case)  with $Z_N=S_N^{\{q_j\}}G$. However, in comparison with the case when $q_\ell$ grows faster than linearly, the proof will require working with a certain ``associated" cocycles of random complex transfer operators, and to use certain ``spectral" properties of them, which are studied independently in Section \ref{sec tower}. 

\subsection{The CLT}
In order to apply Theorem \ref{General LLT } we first  need to establish the CLT for $N^{-1/2}\big(Z_N-\bar G N\big)$.
Let $X_0$ be a $\Del$-valued random variable which is distributed according to $\mu$, and let $X_n=F^nX_0$. 
\begin{proposition}\label{CLTPROP}
The sequence $\{X_n\}$  can be written in the form \eqref{BerFor}, and there are constants $c,C>0$ so that for every $r\in\bbN$ and $p\geq1$ we have $\be_p(r)\leq Ce^{-cr/p}$.
\end{proposition}
\begin{proof}
It follows from the arguments at the beginning of \cite[Section 5.3]{Kor1} that $X_n$ has the form \eqref{BerFor}, and that
$\tilde\be_p(r)$ given by the right hand side of \eqref{beta2} satisfies $\tilde\be_p(r)\leq C_0e^{-c_0r/p}$ for every $r,p\geq1$, where $c_0,C_0>0$ are some constants not depending on $r$ and $p$. Now the proposition follows from \eqref{Triangle} and \eqref{beta2}.
\end{proof}
Next, let us fix some $N$, and using \eqref{Inversion} let us write $\xi_n=X_{\ell N-n}$. We also  set $\xi_{n,r,N}:=\textbf{X}_{\ell N-n,r}$, where $\textbf{X}_{n,r}$ are defined in \eqref{bold X}. Then for any two sets $A,B\subset\{1,2,...,\ell N\}$ the random variables $\{\xi_{n,r,N}:\, n\in A\}$ and $\{\xi_{n,r,N}:\,n\in B\}$ are independent if $\inf_{n\in A,m\in B}\rho(n,m)>2r$, where $\rho$ is defined in \eqref{rhoDef}. Using this local dependence structure, taking into account Proposition \ref{CLTPROP} the proof of all the results concerning the asymptotic variance
$D^2$ are also proved  similarly to \cite{HK2}.

Next, arguing as in Section \ref{S1},  when $D^2>0$, in order to prove the CLT for $N^{-1/2}S_N^{\{q_j\}}G$, assuming again without loss of generality that $\bar G=0$, it is enough to prove the CLT 
for $W_N=\sum_{n}Z_{n,N}$ where with
 $$Q_{n,N}=G(\xi_{q_1(n),N,r_N},\xi_{q_2(n),N,r_N},...,\xi_{q_\ell(n),N,r_N})$$
 and 
$r_N=[N^\zeta]$ (for some $\zeta\in(0,1/4)$) we have
$$
Z_n=Z_{n,N}=\frac{Q_{n,N}-\bbE[Q_{n,N}]}{\sqrt N D}.
$$
The CLT for $W_N$ is proved similarly to Section \ref{Step2}. In fact, the proof is easier in our situation since
 $Z(A)=\{Z_{n}:\, n\in A\}$ and $Z(B)=\{Z_n:\,n\in B\}$ are independent if $A$ and $B$ are not connected by an edge in the graph $\cG=\cG_N$ defined in Section \ref{Step2}. Thus we get a true dependency graph, and so the terms $\del_i(N)$ from Section \ref{Step2} actually vanish.
\subsection{The LCLT}

\subsection{The conditioning step}
The first part in the proof of the LCLT is  a certain conditioning argument, whose purpose is to obtain upper bounds of the form \eqref{Red}.
 In the case of nonlinear indexes such a step was carried out in Lemmas \ref{ZeLem1} and \ref{ZeLem2} (leading to \eqref{Re1}), but when all $q_i$'s are linear the conditioning step is executed differently, and requires $\{\xi_n\}$ to be a Markov chain. The point is that since $q_\ell$ is linear, it is impossible to pass to independent copies of the variables $\xi_{q_\ell(n)}$, and instead we will use the Markov property, which will yield upper bounds involving certain type of random operators that will be studied in the next sections.
\vskip0.1cm
We first need some notations.
Let $\{\xi_n^{(j)}:\,n\in\bbN\}$, $j=1,2,...,\ell$ be $\ell$ independent copies of the Markov chain $\{\xi_n:\,n\geq0\}$. Consider the stationary Markov chain $\{\Xi_n:\,n\geq0\}$ given by 
$$
\Xi_n=(\xi_n^{(1)},\xi_{2n}^{(2)},...,\xi_{(\ell-1)n}^{(\ell-1)}).
$$
Next,  let us consider the operators $R_{it,\bar x},\,t\in\bbR,\bar x\in\Del^{\ell-1}$ which map a function $g$ on $\Del$ to another function $R_{it,\bar x}g$ given by 
\begin{equation}\label{R def}
R_{it,\bar x}g(x)=\bbE[e^{it G_\ell(\bar x,\xi_\ell)}g(\xi_\ell)|\xi_0=x]=\cA^\ell\big(e^{itG_{\ell,\bar x}}g\big)(x)=\frac{P^\ell\big(e^{itG_{\ell,\bar x}}gh\big)(x)}{h(x)}
\end{equation}
where $G_{\ell,\bar x}(x)=G_\ell(\bar x,x)=G(\bar x,x)-\int G(\bar x,y)d\mu(y)$ (which was also defined in \eqref{G ell}).
Let us consider the random operators 
$$
R_{it}^{\Xi,n}=R_{it}^{\Xi_{n-1}}\circ\dots\circ R_{it}^{\Xi_1}\circ R_{it}^{\Xi_0}
$$
and set $M_N=[a_\ell N]$, where $a_\ell=1-\frac1{4\ell}$.

The main result in this section is the following.
\begin{proposition}[Conditioning step for linear indexes]\label{Basic}
There  is a constant $C>0$ and a sequence $(\ve_N)$ so that $\lim_{N\to\infty}\sqrt N\ve_N=0$ and
for all $t\in\bbR$ and $N\geq 1$ we have
\begin{equation}\label{basic bound}
\left|\mathbb E\left[\exp\left(it S_N^{\{q_j\}}G\right)\right]\right|\leq\bbE\left[\int \left|R_{it}^{\Xi,N-M_N}\textbf{1}(x)\right|d\mu(x)\right]
+(|t|+1)\ve_N
\end{equation}
where $\textbf{1}$ is the function taking the constant value $1$.
\end{proposition}

\subsection{Preparations for the proof of Proposition \ref{Basic}}

Let us consider the variable 
$\bar x_N=\{x_{jn}:\,1\leq j<\ell, M_N<n\leq N\}$, $x_{jn}\in\Del$ and the function 
\begin{equation}\label{THEREE'}
H_{N,t}(\bar x_N,y)=\bbE\left[\exp\left(it\sum_{n=M_N+1}^{N}G(x_{n},x_{2n},...,x_{(\ell-1)n},\xi_{\ell n}^{(\ell)})\right)\Big|\xi_{\ell M_N}^{(\ell)}=y\right].
\end{equation}
Then, by the Markov property $H_{N,t}$ can also be written as \begin{equation}\label{THREE}
H_{N,t}(\bar x_N,y)=\left(\prod_{n=M_N+1}^{N}R_{it,(x_{n},x_{2n},...,x_{(\ell-1)n})}\right)\textbf{1}(y)
\end{equation}
where $\prod_{j=1}^{s}A_j=A_{s}\circ\dots \circ A_2\circ A_1$ for any operators $A_1,...,A_s$.

In the course of the proof of Proposition \ref{Basic} we will need the following result.
\begin{lemma}[Regularity of conditional expectations]
There exists a constant $C>0$ so that for every $N\in\bbN$, $t\in\bbR$
and $(\bar x_N,y)$ and $(\bar z_N,w)$ we have
\begin{equation}\label{FOUR}
|H_{N,t}(\bar x_N,y)-H_{N,t}(\bar z_N,w)|\leq C(|t|+1)\left(\sum_{k,n} d(x_{kn},z_{kn})+d(y,w)\right)
\end{equation}
where the sum ranges over the pairs $(k,n)$ of positive integers so that $1\leq k<\ell$ and $M_N<n\leq N$.
\end{lemma}
 \begin{proof}
First, since $|H_{N,t}|\leq 1$ if $y$ and $w$ do not belong to the same floor of $\Del$ then \eqref{FOUR} trivially holds true with $C=2$ since $d(y,w)=1$. Let us now assume that $y=(\tilde y_0,s),w=(\tilde w_0,s)\in\Del_{s}$ for some $s$. Let us also set $L_N=(N-M_N)\ell$.
In the case when $s\geq L_N$ the only preimages of $(y_0,s)$ and $(w_0,s)$ under $F^{L_N}$ are $y_N:=(\tilde y_0,s-L_N)$ and $w_N:=(\tilde w_0,s-L_N)$, respectively. In this case, with $\tilde x_n=(x_n,x_{2n},...,x_{(\ell-1)n})$
 we have 
$$
|H_{N,t}(\bar x_N,y)-H_{N,t}(\bar z_N,w)|
$$
$$=\left|\frac{h(y_N)e^{it\sum_{n=0}^{N-M_N-1}G_{\ell,\tilde x_{n+M_N+1}}(F^{\ell n}y_N)}}{h(y)}-\frac{h(w_N)e^{it\sum_{n=0}^{N-M_N-1}G_{\ell,\tilde z_{n+M_N+1}}(F^{\ell n}w_N)}}{h(w)}\right|.
$$
Using now that the invariant density $h$ is Lipschitz continuous,  bounded and bounded away from $0$, that 
$$
d_U(F^{\ell n}y_N,F^{\ell n}w_N)\leq \be^{L_N-\ell n}d_U(y,w),\,0\leq n<N-M_N
$$
and that the function $G$ satisfies \eqref{G Hold1} 
we obtain \eqref{FOUR} with some $C$ (the exact details are similar to the more complicated case $s<L_N$ considered below).

Let us assume next that $s<L_N$. Then we can write 
$$F^{-L_N}\{y\}=\{y_j\}\,\text{ and }\,F^{-L_N}\{w\}=\{w_j\}$$ 
where for each $j$ we have
\begin{equation}\label{Pair}
d_U(F^{\ell n}y_j,F^{\ell n}w_j)\leq \be^{L_N-\ell n}d_U(y,w)\leq d_U(y,w),\,\, \,0\leq n<N-M_{N}
\end{equation}
(the pairs $(y_j,w_j)$ belong to the same cylinder of length $L_N$).
Let us set $$U(\bar x_N,y_j)=\sum_{n=0}^{N-M_N-1}G_{\ell,\tilde x_{n+M_N+1}}(F^{\ell n}y_j),\,\,\,\tilde x_n=(x_n,x_{2n},...,x_{(\ell-1)n}),$$
and $V(\bar z_N,w_j)=\sum_{n=0}^{N-M_N-1}G_{\ell,\tilde z_{n+M_N+1}}(F^{\ell n}w_j)$.
Then
\begin{equation}\label{I's}
|H_{N,t}(\bar x_N,y)-H_{N,t}(\bar z_N,w)|
\end{equation}
$$=\left|h(y)^{-1}\sum_{j}\frac{h(y_j)e^{itU(\bar x_N,y_j)}}{JF^{L_N}(y_j)}-
h(w)^{-1}\sum_{j}\frac{h(w_j)e^{itV(\bar z_N,w_j)}}{JF^{L_N}(w_j)}\right|$$
$$
\leq I_1+I_2+I_3+I_4
$$
where
$$
I_1=|h(y)^{-1}-h(w )^{-1}|\left|\sum_{j}\frac{h(y_j)e^{itU(\bar x_N,y_j)}}{JF^{L_N}(y_j)}\right|,
$$
$$
I_2=h(w)^{-1}\left|\sum_{j}\frac{h(y_j)e^{itU(\bar x_N,y_j)}}{JF^{L_N}(y_j)}-
\sum_{j}\frac{h(w_j)e^{itU(\bar x_N,y_j)}}{JF^{L_N}(y_j)}\right|,
$$
$$
I_3=h(w)^{-1}\left|\sum_{j}\frac{h(w_j)e^{itU(\bar x_N,y_j)}}{JF^{L_N}(y_j)}-
\sum_{j}\frac{h(w_j)e^{itV(\bar z_N,w_j)}}{JF^{L_N}(y_j)}\right|
$$
and
$$
I_4=h(w)^{-1}\left|\sum_{j}\frac{h(w_j)e^{itV(\bar z_N,w_j)}}{JF^{L_N}(y_j)}-
\sum_{j}\frac{h(w_j)e^{itV(\bar z_N,w_j)}}{JF^{L_N}(w_j)}\right|.
$$
To estimate $I_1$, using that $h$ is bounded and bounded away from $0$ and that $\cA\textbf{1}=\textbf{1}$ (or $Ph=h$) we have
$$
I_1\leq C|h(y)-h(w)|h(y)^{-1}\sum_{j}\frac{h(y_j)}{JF^{L_N}(y_j)}= C|h(y)-h(w)|\leq C'd_U(w,y)
$$
where $C>0$ is some constant.
To bound $I_2$, since $h$ is bounded away from $0$ there is a constant $C_1>0$ so that
$$
I_2\leq C_1\sum_{j}\frac{|h(y_j)-h(w_j)|}{JF^{L_N}(y_j)}.
$$
Now, since $d_U(y_j,w_j)\leq d_U(y,w)$ and $h$ is bounded and bounded away from $0$ we have $|h(y_j)-h(w_j)|\leq C_1'h(y_j)h(y)^{-1}d_U(y,w)$, where $C_1'>0$ is another constant. Hence,
$$
I_2\leq C_1''d_U(y,w)\cA^{L_N}\textbf{1}(y)= C_1''d_U(y,w).
$$
Next, in order to estimate $I_3$, since $G$ satisfies \eqref{G Hold1}, by the mean value theorem we have
$$
|e^{itU(\bar x_N,y_j)}-e^{itV(\bar z_N,w_j)}|\leq |t|\cdot |U(\bar x_N,y_j)-V(\bar z_N,w_j)|
$$
$$\leq
C_2K|t|\left(\sum d_U(x_{kn},z_{kn})+d_U(w,y)\right)
$$
where we have also used \eqref{Pair} and $C=C_2$ is a constant which depends only on $\beta$ and $\ell$. Thus, using again that $h$ is bounded and bounded away from $0$,
with some constant $C_3>0$ we have
$$
I_3\leq C_3 |t|\left(\sum d_U(x_{kn},z_{kn})+d_U(w,y)\right)\cA^{L_N}\textbf{1}(y)$$$$=
C_3 |t|\left(\sum d_U(x_{kn},z_{kn})+d_U(w,y)\right).
$$
Finally, we have 
$$
I_4\leq h(w)^{-1}\sum_{j}h(w_j)\left|\frac{1}{JF^{L_N}(y_j)}-\frac{1}{JF^{L_N}(w_j)}\right|.
$$
Observe now that since $s<L_N$ we have $F^{L_N-s+k}y_j=(\tilde y_0,k)$ and 
$F^{L_N-s+k}w_j=(\tilde w_0,k)$, $k=0,1,...,s$ where we recall that $y=(\tilde y_0,s)$ and $w=(\tilde w_0,s)$. Thus, 
$$
JF^{L_N}(y_j)=\prod_{n=0}^{L_N-s-1}JF(F^{\ell n}y_j)
$$
and 
$$
JF^{\ell(N-M_N)}(w_j)=\prod_{k=0}^{L_N-s-1}JF(F^{\ell n}w_j)
$$
since the other values of $JF$ we have omitted equal $1$ because at the corresponding point the tower map just lifts the points to the next floor.  Using \eqref{Jack Reg1}, \eqref{Pair} and that $F^{\ell n}y_j$ and $F^{\ell n}w_j$, $0\leq n\leq L_N-s-1$, return to the base before reaching $y$ and $w$ (in the original $L_N$-th iterate), respectively, we have 
$$
\left|\frac{JF(F^{\ell n}w_j)}{JF(F^{\ell n}y_j)}-1\right|\leq C_4\be^{L_N-\ell n}d_U(y,w),\,\,0\leq n<L_N-s
$$
and so (by summing up the logarithms),
$$
\left|\frac{JF^{L_N}(w_j)}{JF^{L_N}(y_j)}-1\right|\leq C_4'd_U(y,w)
$$
where $C_4$ and $C_4'$ are some positive constants.
Thus,
$$
I_4\leq C_4'd_U(y,w)h(w)^{-1}\sum_j \frac{h(w_j)}{JF^{L_N}(w_j)}
=C_4'd_U(y,w).
$$
Now, in the case when $s<L_N$ we obtain \eqref{FOUR} from \eqref{I's} and the above estimates on $I_1,I_2,I_3$ and $I_4$.
 \end{proof}

\subsection{Proof of Proposition \ref{Basic}: the conditioning step}

First, recall that with $a_\ell=1-\frac4{\ell}$ we have $M_N=[a_\ell N]$. Thus, there exists $\del_0>0$ so that  for all $N$ large enough,
\begin{equation}\label{M N}
iN+\del_0N<(i+1)M_N-\del_0 N,\,\,i=1,2,...,\ell-1.
\end{equation}
Let $\cM=\cM_N$ be the $\sig$-algebra generated by the random variables $\xi_1,...,\xi_{\ell M_N}$.
Let us set $S_N=S_N^{\{q_j\}}G$ and $S_N G_\ell=S_N^{\{q_j\}}G_\ell$. Then $S_{M_N}$ is $\cM$-measurable and
so
\begin{equation}\label{ONE}
|\bbE[\exp(it S_N)]|=|\bbE[\exp(itS_{M_N})\bbE[\exp(it(S_N-S_{M_N}))|\cM]]|
\end{equation}
$$
\leq
\bbE[|\bbE[\exp(it(S_N-S_{M_N}))|\cM]|=\bbE[|\bbE[\exp(it(S_NG_\ell-S_{M_N}G_\ell))|\cM]|
$$
where in the last inequality we have used that the random variable $\int G(\xi_{n},...,\xi_{(\ell-1)n},x)d\mu(x)$ is $\cM$-measurable for all $1\leq n\leq N$.

Next, by the Markov property, with $\bar\xi_N=\{\xi_{jn}:\,1\leq j<\ell, M_N<n\leq N\}$ we have
\begin{equation}\label{TWO}
\bbE[\exp(it(S_NG_\ell-S_{M_N}G_\ell))|\cM]=H_{N,t}(\bar \xi_N,\xi_{\ell M_N})
\end{equation}
where $H_{N,t}$ was defined in \eqref{THEREE'}.
Now, using \eqref{Inversion} we can write $\xi_n=X_{\ell N-n}=F^{\ell N-n}X_0$, for $n=1,2,...,N\ell$. Let us set $r_N=[\del_0 N]$, $\xi_{n,r}=\xi_{n,N,r}=\textbf{X}_{\ell N-n,r_N}$ and $$\bar\xi_{N,r_N}=\{\xi_{jn,N,r_N}:\,1\leq j<\ell, M_N<n\leq N\},$$ where $\textbf{X}_{n,r}$ are defined in \eqref{bold X} (recall Proposition \ref{CLTPROP}).
Using \eqref{ONE}, \eqref{TWO} and \eqref{FOUR} we have
\begin{equation}\label{FIVE}
|\bbE[e^{it S_N}]|\leq\bbE[|H_{N,t}(\bar \xi_{N,r_N},\xi_{\ell M_N,r_N})|]+(|t|+1)O(N e^{-c\del_0 N})
\end{equation}
for some $c>0$.
Now, by \eqref{M N}, the random variables 
$$U_{j,N}=\{\xi_{jn,N,r_N}:\,M_N<n\leq N\}, \,j=1,2,...,\ell-1$$ 
and $\xi_{\ell M_N,N,r_N}$ are independent. By considering independent copies $$\{\xi_{jn}^{(j)}:\,M_N<n\leq N\}\,\, \text{ and }\,\,\{\xi_{jn,N,r_N}^{(j)}:\,M_N<n\leq N\}$$
of $\{\xi_{jn}:\,M_N<n\leq N\}$ and $\{\xi_{jn,N,r_N}:\,M_N<n\leq N\}$ for $j=1,2,...,\ell-1$ and applying again \eqref{FOUR} we see that  
\begin{equation}\label{SIX}
|\bbE[e^{it S_N}]|\leq\bbE[|H_{N,t}(\tilde \xi_{N},\xi_{\ell M_N}^{(\ell)})|]+(|t|+1)O(N e^{-c\del_0 N})
\end{equation}
where $\tilde\xi_{N}=\{\xi_{jn}^{(j)}:\,1\leq j<\ell, M_N<n\leq N\}\overset{d}{=}\{\Xi_n:\,1\leq j<\ell\}$ (where $\overset{d}{=}$ stands for equality in distribution). The lemma follows now by \eqref{SIX} and \eqref{THREE}, taking into account that $N e^{-c\del_0 N}=o(N^{-1/2})$, that $\xi_{\ell M_N}$ is distributed according to $\mu$ and that $\{\Xi_n:\,n\geq0\}$ is stationary.
\qed


\subsection{Verification of Assumption \ref{DecRate1}}\label{SecVer1}
We will prove here the following result.
\begin{proposition}\label{Prp}
There are constants  $\del_0,C_0,c_0>0$ and measurable sets $\Gamma_N\subset\Del^{(\ell-1) N}$ so that 
\vskip0.2cm
(i) $\lim\sqrt N\bbP((\Xi_{0},...,\Xi_{N-1})\not\in\Gamma_N)=0$;
\vskip0.2cm

(ii) when $(\Xi_{0},...,\Xi_{N-1})\in\Gamma_N$ and  $t\in[-\del_0,\del_0]$  we have 
\begin{equation}\label{Eq}
\left|\int R^{\Xi,N}_{it}\textbf{1}(x)d\mu(x)\right|\leq C_0e^{-c_0 N t^2}.
\end{equation}
\end{proposition}
\begin{proof}[Verification of Assumption \ref{DecRate1} relying on Proposition \ref{Prp}]
Since $N-M_N\geq a_1N$ for some $a_1>0$ and all $N$ large enough, 
by \eqref{basic bound} and Proposition \ref{Prp} for every $t\in[-\del_0,\del_0]$ and $N$ large enough we have
$$
\left|\mathbb E\left[\exp\left(it S_{N}^{\{q_j\}}G\right)\right]\right|\leq \bbP((\Xi_{1},...,\Xi_{N})\notin\Gamma_N)+ C_0e^{-c_0 N t^2}+\del_0\ve_N.
$$
It is evident now that Assumption \ref{DecRate1} holds with $$b_N=\bbP((\Xi_{0},...,\Xi_{N-1})\notin\Gamma_N)+\del_0\ve_N.$$
\end{proof}

\subsubsection{Proof of Proposition \ref{Prp}}
We first need the following.

\subsubsection{Associated random transfer operators and the RPF theorem.}
Let $\cH$ be the space of bounded Lipschitz continuous functions $u:\Del\to\bbC$ equipped with the norm 
$$
\|u\|=\|u\|_\infty+\text{Lip}(u)
$$
where $\|u\|_\infty=\sup|u|$ and $\text{Lip}(u)=\text{Lip}_{NU}(u)$ is the smallest number $A$ so that $$|u(x)-u(y)|\leq Ad_{NU}(x,y)$$ for all $x,y\in\Del$ which belong to the same floor.  Let us consider the function $v:\Del\to\bbR$ which is constant of the floors of $\Del$ and $v|\Del_k=e^{kp/2}:=v_k$, where $p$ comes from \eqref{ExTails}.
 Let $L$ be the operator which maps a function $g:\Del\to\bbC$ to a function $Lg$ on $\Del$ defined by $$Lg=P(gv)/v,$$ where $P$ is defined in (\ref{P def}).

For every $a\in\Delta^{\ell-1}$ consider the function $u_a:\Del\to\bbR$ given by $u_a= G_{\ell,a}=G_\ell(a,\cdot)$, where $G_\ell$ was defined as in (\ref{G ell}). Then $u_a:\Del\to\bbR,\,a\in\Del^{\ell-1}$ are Lipschitz continuous functions and $\sup_{a}\|u_a\|<\infty$. For every $z\in\bbC$ and $a\in\Delta^{\ell-1}$ consider the transfer operator $\cL_z^a$ given by
\[
\cL_z^{a}g=L^\ell(e^{zG_\ell(a,\cdot)}g)=L^\ell(e^{zu_a}g).
\] 
Set $\cL_z^{\Xi,n}=\cL_{z}^{\Xi_{n-1}}\circ\dots\circ\cL_z^{\Xi_1}\circ\cL_z^{\Xi_0}$. 
Then, since $P$ is the dual of the  Koopman operator $g\to g\circ F$ with respect to the measure $m$,
 the integral inside the absolute value on the right hand side of \eqref{Eq} can be rewritten as
\begin{equation}\label{Char0}
\int R^{\Xi,N}_{it}\textbf{1}(x)d\mu(x)=\int \Big(\cL_{it}^{\Xi,N}(h/v)\Big)(y)dm_L(y)
\end{equation}
where $d m_L(y)=v d m(y)$ (note that $L$ is the dual of $F$ with respect to $m_L$) and $h=d\mu/d m$, which is  Lipschitz continuous and it is bounded and bounded away from $0$.

The proof of Proposition \ref{Prp} is based on the following two results.

\begin{lemma}\label{Lem III}
There exist  constants $\del_1,c_0, C_0>0$ so that, for almost every realization of $\{\Xi_n:n\geq0\}$ and for all $t\in[-\del_1,\del_1]$ and $N\in\bbN$ we have
\begin{equation}\label{III.1}
\|\cL_{it}^{\Xi,N}\|\leq Ce^{-t^2\sig_{\Xi,N}^2/2+C_0t^2/2+c_0N|t|^3}
\end{equation}
where with $S_n^{\Xi} u=\sum_{j=0}^{n-1}u_{\Xi_j}\circ F^{j\ell}$,
$$
\sig_{\Xi,N}^2=\sig_{(\Xi_0,...,\Xi_{N-1}),N}^2=\text{Var}_{\mu}(S_N^\Xi u).
$$
\end{lemma}

\begin{lemma}\label{Lemm IIII}
There exists a constant $c_1>0$ with the following property. For every $N$ 
set $\Gam_N=\{\bar a\in(\Del^{\ell-1})^{N}:\,\sig_{\bar a,N}^2>c_1N\}$. Then there exists a constant $c_2>0$ so that for all $N\geq1$ we have
\begin{equation}\label{Gamm prob}
\bbP((\Xi_0,\Xi_1,...,\Xi_{N-1})\notin\Gam_N)\leq c_2N^{-1}.
\end{equation}
\end{lemma}

\paragraph{\textbf{Completing the proof of Proposition \ref{Prp} relying on Lemmas \ref{Lem III} and \ref{Lemm IIII}.}}
We first infer from (\ref{III.1}) that there exist constants $\del_0,C_3>0$ and $c_3>0$ so that 
when $(\Xi_0,...,\Xi_{N-1})\in\Gamma_N$, for every $t\in[-\del_0,\del_0]$ we have
\begin{equation}\label{Conc.1}
\|\cL_{it}^{\Xi,N}\|\leq C_3e^{-t^2c_3 N}.
\end{equation}
Combing this with \eqref{Char0}, when $(\Xi_0,\Xi_1,...,\Xi_{N-1})\in\Gamma_N$ we get that for every $t\in[-\del_0,\del_0]$,
$$
\left|\int R^{\Xi,N}_{it}\textbf{1}(x)d\mu(x)\right|\leq C_1e^{-t^2c_3 N}
$$
where $C_1$ is another constant, and
the proof of Proposition \ref{Prp} is complete, taking into account \eqref{Gamm prob}.

\begin{proof} [Proof of Lemma \ref{Lemm IIII}]
Let us first observe that Lemma \ref{Lemm IIII} concerns only the distribution of 
$(\Xi_0,...,\Xi_{N-1})$ for any fixed $N$. Set $T=F\times F^{2}\times\dots\times F^{\ell-1}$. Then by 
\eqref{Inversion},
\begin{equation}\label{Inv2}
(\Xi_0,\Xi_1,...,\Xi_{N-1})\overset{d}{=}(T^{N-1}\bar Y_0,T^{N-2}\bar Y_0,...,T\bar Y_0, \bar Y_0)
\end{equation}
where $\bar Y_0$ is distributed according to $\mu\times\mu\dots\times\mu=\mu^{\ell-1}$.
Therefore
$$
\sig_{\Xi,N}^2\overset{d}{=}\text{Var}_{\mu}(S_N^{\bar Y_0} u):=V_N(\bar Y_0)
$$
where $\overset{d}{=}$ stands for equality in distribution and
$$
S_N^{\bar y_0} u=\sum_{n=0}^{N-1}G_{\ell,T^{N-n}\bar y_0}\circ F^{\ell j}.
$$
Observe next that our assumption that $D_\ell^2>0$ (in Theorems \ref{Non lattice LLT3} and \ref{Lattice LLT3}) is equivalent to $\lim_{k\to\infty}\frac 1k\mathbb E[V_k(\bar Y_0)]:=b>0$, since 
$$
\bbE[V_N(\bar Y_0)]=\bbE[\sig_{\Xi,N}^2]=\bbE\left[\left(\sum_{n=1}^{N}G_\ell(\xi_{n}^{(1)},\xi_{2n}^{(2)},...,\xi_{\ell n}^{(\ell)})\right)^2\right]=
\bbE_{\mu^\ell}\left[\left(\sum_{n=0}^{N-1}G_\ell\circ F_\ell^n\right)^2\right]
$$
where $F_\ell=T\times F^\ell=F\times F^2\times\cdots\times F^\ell$.
 Therefore if $k$ is sufficiently large then 
\begin{equation}\label{EXP b}
\bbE[V_k(\bar Y_0)]\geq c:=\frac b2.
\end{equation}
It is clear that each $F^j$ has a tower extension, and therefore by \cite{TowerProduct} the map $T=F\times F^2\times\dots\times F^{\ell-1}$ also has a Tower extension (which is mixing since $F$ is mixing). Consider the functions $\hat V_k:(\Del^{\ell-1})^k\to\bbR$ given by
\[
\hat V_k(a_0,...,a_{k-1})=\text{Var}_{\mu}\left(\sum_{j=0}^{k-1}u_{a_j}\circ F^{j\ell}\right)=
\bbE\left[\Big(\sum_{j=0}^{k-1}G_\ell(a_j,F^{j\ell}X_0)\Big)^2\right]
\]
where $X_0$ is distributed according to $\mu$.
Then, since $G$ is bounded and satisfies \eqref{G Hold1}, for every $j$ the H\"older constant of $V_k$ at the direction of the variable $a_j$ does not exceed $ck$ for some constant $c$ not depending on $k$. Observe that
\[
V_k(\bar Y_0)=\hat V_k(T^{k-1}\bar Y_0,...,T\bar Y_0,\bar Y_0).
\]
Applying the results in \cite[Section 3]{ChazGO} with the function $\hat V$ and the map $T$, taking into account (\ref{EXP b}), we conclude (in particular) that for every $\eta>0$ there exists a constant $d_1>0$ so that for all sufficiently large $k$ we have
\begin{equation}\label{VarProb}
\mu^{\ell-1}\left\{\bar y_0:\sum_{j=0}^{N-1} V_k(T^{j}\bar y_0)\leq d_1 kN\right\}\leq d_2(k)N^{-\eta}
\end{equation}
where $d_2(k)$ is a constant which depends only on $k$ (using \eqref{EXP b} we can take $d_1=c/2$).
Let us fix some $k$ large enough.
Since \eqref{VarProb} holds true, we can apply \cite[Proposition 5.2.1]{Hafouta ETDS} and get that there exist constants $c_1,c_2>0$ so that for every $N\geq1$ we have
\[
\mu^{\ell-1}\left\{\bar y_0: \sig_{\bar y_0,N}^2\leq c_1N\right\}=\bbP((\Xi_0,...,\Xi_{N-1})\not\in\Gamma_N)\leq c_2 N^{-1}.
\] 
\end{proof}

\subsubsection{Proof of Lemma \ref{Lem III}: the random complex RPF theorem}
The proof of Lemma \ref{Lem III} relies on a more precise study of the asymptotics of the iterates $\cL_z^{\Xi,N}$.  To do that we first need to express the right hand side of \eqref{basic bound} by means of a random dynamical system. This will be useful since we eventually want to apply a theorem from \cite[Ch.4]{book} concerning random operators.

Set $\tilde\Om=(\Del^{\ell-1})^{\bbZ}$ and let $\te:\tilde\Om\to\tilde\Om$ be the left shift given by $(\te\om)_i=\om_{i+1}$ for $\om=(\om_i)\in\tilde\Om$. Let $\cB$ be  the Borel $\sig$-algebra of $\tilde \Om$.
By the Kolmogorov extension theorem there exists a unique $\te$-invariant probability measure $\tilde\bbP$ on $\tilde\Om$ so that for all $A_{0},A_1,...,A_s\in\Del^{\ell-1}$ we have $$\tilde\bbP(A_0\times A_1\times\cdots\times A_s)=\bbP(\Xi_i\in A_i;\, i=0,1,...,s).$$ Since $\{\Xi_n\}$ is mixing, the probability preserving system $(\tilde\Om,\cB,\tilde\bbP,\te)$ is mixing. 
We will abuse these notations and for $\om=(\om_i)\in\tilde\Om$ we write  $G_{\ell,\om}=G_{\ell,\om_0}$ and
$\cL_z^{\om}=\cL_{z}^{\om_0}$. Set $$\cL_z^{\om,n}=\cL_{z}^{\te^{n-1}\om}\circ\dots\circ\cL_z^{\te\om}\circ\dots\circ\cL_z^{\om}=\cL_{z}^{\om_{n-1}}\circ\dots\circ\cL_z^{\om_1}\circ\dots\circ\cL_z^{\om_0}.$$
Then, roughly speaking, the strategy of the proof of  Lemma \ref{Lem III} is to show that when $|z|$ is small enough  the iterates 
$\cL_z^{\om,n}$ behave like a one dimensional operator times the exponent of the real part of some pressure function $\Pi_{\om,n}(z)$ so that $\Pi_{\om,n}(0)=\Pi_{\om,n}'(0)=0$ and $\Pi_{\om,n}''(0)=\sig_{\om,n}^2+O(1)$, where $\sig_{\om,n}^2=\sig_{(\om_0,\om_1,...,\om_{n-1}),n}^2$.
Using that the desired estimates will follow from Taylor expansion of order $2$ of $\Pi_{\om,n}(z)$ around $0$. The one dimensional asymptotic behavior is the content of the following result.

\begin{theorem}[Random complex RPF theorem]\label{RPFthm}
There exist $r_0,c>0$ and $\del\in(0,1)$
so that for $\tilde \bbP$-almost every $\om$, for every complex number $z$ whose modulus does not exceed $r_0$ there is a function $h_{\om}^{(z)}\in\cH$, a non-zero complex number $\la_{\om}(z)$ and a complex linear functional $\nu_{\om}^{(z)}\in\cH^*$ so that $h_{\om}^{(0)}=h/v$, $\la_{\om}(0)=1$, $\nu_{\om}^{(0)}=m_L$ and
\begin{equation}\label{RPF}
\left\|\cL_z^{\om,n}/{\la_{\om,n}(z)}-h_{\te^n\om}^{(z)}\otimes \nu_{\om}^{(z)}\right\|\leq c\delta^n
\end{equation}
where $\la_{\om,n}(z)=\prod_{j=0}^{n-1}\la_{\te^j\om}(z)$ and $(f\otimes \nu)(g)=\nu(g)\cdot f$.
Moreover, 
\[
\nu_{\om}^{(z)}(h_{\om}^{(z)})=\nu_{\om}^{(z)}(h_\om^{(0)})=1,\, \cL_z^{\om} h_{\om}^{(z)}=h_{\te\om}^{(z)}\, \text{ and }\,(\cL_z^{\om})^* \nu_{\te\om}^{(z)}=\la_{\om}(z)\nu_{\om}^{(z)}.
\]
Furthermore, $\la_{\om}(z),h_{\om}^{(z)}$ and $\nu_{\om}^{(z)}$ are measurable in $\om$, analytic in $z$ and are uniformly bounded in $(\om,z)$.
\end{theorem}
\begin{proof}
Applying  Theorem \ref{YT cones thm} we see that the conditions of \cite[Theorems 4.1]{book} and \cite[Theorem 4.2]{book} in \cite{book} hold true, which yields Theorem \ref{RPFthm}. 
\end{proof}

\subsubsection{The random pressure function}
Since $\la_{\om}(0)=1$ and $\la_{\om}(z)$  are analytic and uniformly bounded, by decreasing $r_0$ we can also assume that $|\la_{\om}(z)|$ 
is uniformly bounded from below by some positive constant. Therefore, we can construct analytic functions $\Pi_{\om}(z)$ around $0$ so that $\Pi_{\om}(0)=0$, $|\Pi_{\om}(z)|\leq c$ and $e^{\Pi_{\om}(z)}=\la_{\om}(z)$, where $c>0$ is some constant. For each $n$ set $\Pi_{\om,n}(z)=\sum_{j=0}^{n-1}\Pi_{\te^j\om}(z)$.

Next, set $u_\om(x)=u_{\om_0}(x)=G_\ell(\om_0,x)$ and for every $n\in\bbN$,
$$S_n^{\om} u=\sum_{j=0}^{n-1}u_{\te^{j}\om}\circ  F^{j\ell}=\sum_{j=0}^{n-1}u_{\om_j}\circ  F^{j\ell}.$$
 Then since $\int G_\ell(x_1,...,x_{\ell-1},x)d\mu(x)=0$ and $\mu$ is $F$-invariant we have $\int S_n^{\om} u d\mu=0$.
\begin{lemma}
For $\tilde\bbP$-almost all $\om$ and  every $n\in\bbN$ we have
\begin{equation}\label{P'=0}
\Pi_{\om,n}'(0)=\frac{d}{dz}\Pi_{\om,n}'(z)\big|_{z=0}=\int S_n^{\om} u d\mu=0.
\end{equation}
\end{lemma}
\begin{proof}
Differentiating both sides of the identities $\nu_{\om}^{(z)}(h_{\om}^{(z)})=1$ and $\nu_{\om}^{(z)}(h_{\om}^{(0)})=1$ with respect to $z$ and plugging in $z=0$ we get that 
\begin{equation}\label{USING}
\nu_{\om}^{(0)}\left(\frac{d}{dz}h_{\om}^{(z)}\Big|_{z=0}\right)=0.
\end{equation}
Differentiating both sides of the identity 
$\cL_z^{\om,n}h_{\om}^{(z)}=\la_{\om,n}(z)h_{\te^n\om}^{(z)}$, plugging in $z=0$ and then integrating both sides with respect to $m_L=\nu_\om^{(0)}$ and using \eqref{USING} we get that 
\[
\la_{\om,n}'(0)=m_L(h_w^{(0)}S_n^{\om} u)=\int S_n^{\om} u d\mu=0.
\]
Since $\la_{\om,n}'(0)=\Pi_{\om,n}'(0)$ the proof of the claim is complete.
\end{proof}
The following lemma is an important ingredient in the verification of Assumption \ref{DecRate1},
and it connects between the variance of $S_n^{\om}u$ and  the second derivative of the pressure function at $z=0$. 
\begin{lemma}
There exists a constant $C_0>0$ so that $\tilde\bbP$-a.s. for every $n\geq1$ we have
\begin{equation}\label{I.1}
\left|\Pi_{\om,n}''(0)-\text{Var}_{\mu}(S_n^{\om} u)\right|\leq C_0.
\end{equation}
\end{lemma}
\begin{proof}
For every complex number $z$ we have
\begin{equation}\label{CharFunc}
\mu(e^{zS_n^{\om} u})=m_L\big(\cL_{z}^{\om,n}(h/v)\big).
\end{equation}
Using (\ref{RPF}) we can write
\[
m_L\big(\cL_{z}^{\om,n}(h/v)\big)=\la_{\om,n}(z)\left(m_L(h_{\te^n \om}^{(z)})\nu_{\om}^{(z)}(h/v)+\del_{\om,n}(z)\right)
\]
where $\del_{\om,n}(z)$ is an analytic function so that $|\del_{\om,n}(z)|\leq c\del^n$. Since the first summand inside the  brackets on the above right hand side is analytic in $z$, uniformly bounded in $\om, n$ and $z$ and takes the value $1$ when $z=0$, taking the logarithm of both sides of  (\ref{CharFunc}) and then considering the second derivative at $z=0$ we get that 
\begin{equation}\label{I.1}
\left|\text{Var}_{\mu}(S_n^{\om} u)-\Pi_{\om,n}''(0)\right|\leq C_0
\end{equation}
where $C_0>0$ is some constant which does not depend on $n$. 
\end{proof}

\subsubsection{Employing the pressure: completing proof of Lemma \ref{Lem III}}
First, by (\ref{RPF}) and using that both $\|h_{\om}^{(z)}\|$ and $\|\nu_{\om}^{(z)}\|$ are bounded in $\om$ and $z$, we see that there exist $r_1,C>0$  so that $\tilde\bbP$-a.s. for all $t\in[-r_1,r_1]$ and $N\geq 1$ we have
\begin{equation}\label{II.1}
\|\cL_{it}^{\om,N}\|\leq C|\la_{\om,N}(it)|= Ce^{\Re(\Pi_{\om,N}(it))}.
\end{equation}
Next, using (\ref{P'=0}) and that $\Pi_\om(z)$ is bounded in $z$ and $\om$, expanding $\Pi_{\om,N}(\cdot)$ around $0$ yields that there is $r_1'>0$ so that $\tilde \bbP$-a.s. for every $t\in[-r_1',r_1']$ and $N\in\bbN$ we have
\[
\left|\Pi_{\om,N}(it)+\frac12 t^2\Pi_{\om,N}''(0)\right|\leq cN|t|^3
\]
where $c>0$ is some constant.
This together with (\ref{I.1}) and (\ref{II.1}) yield that there exists $r_2>0$ so that $\tilde \bbP$ a.s. for every $t\in[-r_2,r_2]$ and all $N\in\bbN$ we have
\begin{equation}\label{III}
\|\cL_{it}^{\om,N}\|\leq Ce^{-t^2\sig_{\om,N}^2/2+C_0t^2/2+cN|t|^3}.
\end{equation}
Lemma \ref{Lem III} follows since both sides of \eqref{III} depend only on $(\om_0,\om_1,...,\om_{N-1})$.
\qed

\subsection{Verification of Assumption \ref{DecRate2} or Assumption \ref{DecRate3}}\label{TO2}
In this section we cannot use the RPF theorem, since we need to analyze $\bbE\left[\exp(it S_N^{\{q_j\}} G)\right]$ for $t$'s which might be far away from $0$.
  Therefore, there will be no need in passing to the invertible probability preserving system $(\tilde\Om,\cB,\tilde \bbP,\te)$.
Another difference, in comparison with Section \ref{SecVer1},  is that we consider here the following weighted Lipschitz norm. For every  $g:\Del\to\bbC$ we set
$$
\|g\|_W=\|g\|_s+\|g\|_h
$$
where with $v_k=e^{kp/2}$ and $p$ coming from \eqref{ExTails},
\[
\|g\|_{s}=\sup_k v_k^{-1}\|g\bbI_{\Del_{k}}\|_\infty, \,\|g\|_h=\sup_k v_k^{-1}\big|g\big|_{\beta,\Del_{k}}
\]
where for every $A\subset \Del$, 
\[
|g|_{\be,A}=\sup_{x,y\in A\,\,x\not=y} \frac{|g(x)-g(y)|}{d_U(x,y)}.
\]
Let us denote by $X$ the space of all functions $g:\Del\to\bbC$ so that $\|g\|_W<\infty$.
A third difference is that we will be using here 
the transfer operators $P_z^a$, $a\in\Del^{\ell-1}$, $z\in\bbC$ given by
\[
P_z^a g(x)=P^\ell (ge^{z u_a})(x)=
\sum_{y\in F^{-\ell}\{x\}}\frac{g(y)e^{z u_a(y)}}{J F^\ell(y)}
\]
where we recall that $u_a(x)=G_\ell(a,x)$.
Let us also set 
$$P_z^{\Xi,n}=P_z^{\Xi_{n-1}}\circ\cdots\circ P_z^{\Xi_1}\circ P_z^{\Xi_0}.$$
Then the first term on the right hand side of \eqref{basic bound} can be written as
\begin{equation}\label{Char01}
\bbE\left[\int \left|R_{it}^{\Xi,N-M_N}\textbf{1}(x)\right|d\mu(x)\right]=
\bbE\left[\int \left|P_{it}^{\Xi,N-M_N}h(x)\right|dm(x)\right].
\end{equation}
In the next section we will prove the following result.

\begin{proposition}\label{LeP}
In the non-arithmetic case, set $I=\bbR\setminus\{0\}$, while in the lattice case set $I=[-\pi,\pi]\setminus\{0\}$. Then in both cases, for every compact set $J\subset I$, there exist measurable sets
$B_N\subset(\Del^{\ell-1})^{N}$ and constants $c_1,C_1,c_2,C_2>0$, which might depend on $J$, so that for every $N\in\bbN$ we have
\vskip0.2cm
(1) $\bbP((\Xi_0,...,\Xi_{N-1})\not\in B_N)\leq C_1e^{-c_1N}$;
\vskip0.2cm
(2) when  $(\Xi_0,...,\Xi_{N-1})\in B_N$ then
$$
\sup_{t\in J}\|P_{it}^{\Xi,N}\|_W\leq C_2 2^{-c_2 N}
$$
where
 $\|A\|_W=\sup_{\|g\|_W=1}\|A_g\|_W$ for any linear operator $A:X\to X$.
\end{proposition}
Before proving the proposition, let us show that it indeed implies that the conditions of either Assumption \ref{DecRate2} or Assumption \ref{DecRate3} are met. 
Let $J$ be a compact set as specified in the lemma.
 It follows that for every $t\in J$ we have
$$
\bbE\left[\int \left|P_{it}^{\Xi,N-M_N}h(x)\right|dm(x)\right]\leq
$$
$$
\bbE\left[\bbI((\Xi_0,...,\Xi_{N-M_N-1})\in B_{N-M_N})\int \left|P_{it}^{\Xi,N-M_N}h(x)\right|dm(x)\right]+C_1e^{-c_1(N-M_N)}
$$
where we have used that $$\int \left|P_{it}^{\Xi,N-M_N}h(x)\right|dm(x)\leq \int P^{\ell(N-M_N)}h(x)dm(x)=\int h(x)dm(x)=1.$$
Next, by the definition of the norm $\|\cdot\|_W$,  for every $k\geq0$  and any realization $\Xi$ of $\{\Xi_n:\,n\geq 0\}$ we have 
\[
\sup|\bbI_{\Del_k}P_{it}^{\Xi,N-M_N}h|\leq \|h\|_W\|P_{it}^{\Xi,N-M_N}\|_W e^{kp/2}
\]
and we also note that $\|h\|_W<\infty$ and that on $\Del_k$ we have $|h|\leq\|h\|_W e^{kp/2}$. Here $p$ comes from \eqref{ExTails}.
Let $k=k_N$ be of the form $k_N=C\ln N$ where $C$ is so large that $m\{R\geq k_N+1\}\leq N^{-1}$.
Then,
$$
\bbE\left[\bbI((\Xi_0,...,\Xi_{N-M_N-1})\in B_{N-M_N})\int \left|P_{it}^{\Xi,N-M_N}h(x)\right|dm(x)\right]
$$
$$\leq \sum_{k\leq k_N}\bbE\left[\bbI((\Xi_0,...,\Xi_{N-M_N-1})\in B_{N-M_N})\int \left|P_{it}^{\Xi,N-M_N}h(x)\right|\bbI_{\Del k}(x)dm(x)\right]
$$
$$+\sum_{k> k_N}\int \left(\bbI_{\Del k}P^{\ell (N-M_N)}h(x)\right)dm(x)\leq $$
$$
 C_2\|h\|_W 2^{-c_2(N-M_N)}\sum_{k\leq k_N}e^{kp/2}\\+C_3\bar m\{R\geq k_N+1\}$$
$$\leq C_2\|h\|_W  2^{-c_2(N-M_N)}N^{p(C+1)/2}+O(N^{-1})=O(N^{-1})$$
where in the second inequality we have used that $Ph=h$ and that the density function $h=d\mu/dm$ is bounded by some constant $C_3$. Using the above estimates together with \eqref{Char01} and \eqref{basic bound} we see that the conditions of Assumption \ref{DecRate2} are met in the non-arithmetic case for $Z_N=S_N^{\{q_j\}}G$, and that the conditions of Assumption \ref{DecRate3} are met in the lattice case for $Z_N=S_N^{\{q_j\}}G$ with $h_0=1$.

\subsubsection{Proof of Proposition \ref{LeP}}
Let us fix some $N\in\bbN$. Then, by \eqref{Inv2} we can replace $\Xi_n$ for $n\leq N$ with $T^{N-n-1}\bar Y_0$, where  we recall that $T=F\times F^2\times\dots\times F^{\ell-1}$. We will abuse the notations and write
\begin{equation}\label{P TWO}
P_{it}^{\bar y_0,n}=P_{it}^{\bar y_0}\circ P_{it}^{Ty_0}\circ\cdots\circ P_{it}^{T^{N-1}y_0}.
\end{equation}
Then $P_{it}^{\bar Y_0,N}$ and $P_{it}^{\Xi,N}$ have the same distribution, when considered as random operators taking values in the Banach space of bounded operators on $X$.

\subsubsection{Auxiliary lemmas}
\begin{lemma}
 For every compact set $J\subset\bbR$ there exists a constant $B_J\geq 1$ so that
 \begin{equation}\label{BJ bound}
\sup\{\|P_{it}^{\bar y_0,n}\|:\,\bar y_0\in\Del^{\ell-1},\,n\geq 1,\, t\in J\}\leq B_J.
\end{equation}
\end{lemma}
\begin{proof}
The lemma follows from Proposition \ref{Prop LY}.
\end{proof}
Next, let $x_0$ be the periodic point of $F$ from Assumption \ref{Second continuity of F} and set $\bar v_0=(x_0,...,x_0)\in \Del^{\ell-1}$. Let $n_0$ be the period of $x_0$.
\begin{lemma}
The transfer operator $P^{\bar v,n_0}_{it}$ is quasi-compact when its spectral radius equals $1$.
\end{lemma}
\begin{proof}
This is explained in Section \ref{AppB}.
\end{proof}

\begin{lemma}\label{Lemma 5.9}
In the non-arithmetic case set $I=\bbR\setminus\{0\}$ while in the lattice case set $I=[-\pi,\pi]\setminus\{0\}$. Then, in both cases for every $t\in I$ the spectral radius of $P^{\bar v_0,n_0}_{it}$ is smaller than $1$.
\end{lemma}
\begin{proof}
Let $t\in I$.
First, by \eqref{BJ bound} the spectral radius of $P^{\bar v_0,n_0}_{it}$ does not exceed $1$.
If the spectral radius in question equals $1$ then  $P^{\bar v_0,n_0}_{it}$ is quasi compact, and $P^{\bar v_0,n_0}_{it}$ has  an eigenvalue of modulus one, but this is 
 equivalent to the function $e^{itS_{n_0}^{\bar v_0}u}$
 being  cohomologous to a constant w.r.t. the map $F^{\ell n_0}$. The latter is excluded in  Theorems \ref{Non lattice LLT3}  and \ref{Lattice LLT3}.
\end{proof}

\begin{corollary}\label{Cor 5.10}
In both lattice and non-arithmetic  cases, for every compact set $J\subset I$ there exist constants $\del_J\in(0,1)$ and $C_J>0$ so that for all sufficiently large $n$ we have 
\[
\sup_{t\in J}\left\|\left(P^{\bar v_0,n_0}_{it}\right)^n\right\|_W\leq C_J(1-\del_J)^n.
\]
\end{corollary}
\begin{proof}
The corollary follows from Lemma \ref{Lemma 5.9}, the compactness of $J$ and the arguments in the proof of  \cite[Lemma III.9]{HH}, which states that the spectral radius is upper semi-continuous.
\end{proof}

\begin{lemma}[Parametric continuity of transfer operators at the periodic orbit]\label{ContLem}
Let us fix some compact set $J\subset\bbR$ and let $m_0\in\bbN$. 
Then for every $0\leq j<n_0$,
$$
\lim_{\bar y\to T^j\bar v_0}\sup_{t\in J}\|P_{it}^{\bar y,m_0}-P_{it}^{T^j\bar v_0,m_0}\|_W=0
$$
where $T=F\times F^2\times\dots\times F^{\ell-1}$.
 \end{lemma}
 \begin{proof}
Taking into account \eqref{BJ bound}, since $T$ is continuous 
it is enough to prove the claim when $m_0=1$. In this case, for every $g\in X$  we have 
$$
\|P_{it}^{\bar y}g-P_{it}^{T^j\bar v_0}g\|_W=\|P^{\ell}\big(g(e^{it G_\ell(\bar y,\cdot)}-e^{it G_\ell(T^j\bar v_0,\cdot)})\big)\|_W
$$
$$
\leq \|P^\ell\|_W\|g\|_W\|e^{it G_\ell(\bar y,\cdot)}-e^{it G_\ell(T^j\bar v_0,\cdot)}\|_W.
$$
Using Assumption \ref{Second continuity of F} the last factor on the above right hand side converges to $0$ uniformly in $t\in J$ as $\bar y\to T^j\bar v_0$.
 \end{proof}

\begin{proof}[Proof of Proposition \ref{LeP}]
Since $P_{it}^{\bar Y_0,N}$ and $P_{it}^{\Xi,N}$ have the same distribution, it is enough to prove that for every compact set $J\subset I$, there are measurable sets $A_N\subset\Del^{\ell-1}$ and constants $c_1,c_2,C_1,C_2>0$ so that $1-\mu^{\ell-1}(A_N)\leq C_1e^{-c_1N}$
and 
$$
\sup_{t\in J}\sup_{\bar y_0\in A_N}\|P_{it}^{\bar y_0,N}\|_W\leq C_2 2^{-c_2 N}.
$$

Fix some compact set $J \subset I$ and let $n_J$ be so that $$C_J(1-\del_J)^{n_J}<\frac1{4B_J}$$
where $B_J$ comes from \eqref{BJ bound} and $\del_J$ and $C_J$ come from Corollary \ref{Cor 5.10}.
Let $\mathscr C _{\ell-1}=\mathscr C\times\mathscr C\times\cdots\times\mathscr C=\mathscr C^{\ell-1}$ be a Cartesian power of a  cylinder $\mathscr C= \bigcap_{j=0}^{M-1}F^{-j}\Del_{s_j,k_j}$ of length $M$ around $x_0$. Then $\bar v_0\in\mathscr C_{\ell-1}$. By applying Lemma \ref{ContLem} with $m_0=n_0 n_J$ we see that there exists $M_J\in\bbN$ so that
if  $M\geq M_J$ 
then for every $\bar y_0\in \mathscr C _{\ell-1}$ we have
\begin{equation}\label{UsIng}
\sup_{t\in J}\left\|P_{it}^{\bar y_0,n_0 n_J}-\left(P^{\bar v_0,n_0}_{it}\right)^{n_J}\right\|_W\leq \frac1{4B_J}.
\end{equation}
Let us set $M=M_J$.
For every $c>0$ we define 
\[
A_{N,c}=\left\{\bar y_0\in \Del^{\ell-1}: \sum_{j=0}^{N-1}\bbI(T^j\bar y_0\in \mathscr C_{\ell-1})\geq cN\right\}.
\]
Let $\bar y_0\in A_{N,c}$, and let $0\leq j_1<j_2<...<j_{R}\leq N$ be the indexes between $0$ to $N-1$  so that $T^{j}\bar y_0\in\mathscr C_\ell$. Then $R=R_{\bar y_0,N,J}\geq cN$. Let  $1\leq u\leq R$. Then, since $T^{j_u}\bar y_0\in\mathscr C_\ell$, by \eqref{UsIng} and Corollary \ref{Cor 5.10} we have
\begin{equation}\label{ThusBy}
\sup_{t\in J}\|P_{it}^{T^{j_u}\bar y_0,n_0n_J}\|_W\leq \frac1{4B_J}+\sup_{t\in J}\|P_{it}^{\bar v_0,n_0n_J}\|_W\leq 
\frac1{4B_J}+C_J(1-\del_J)^{n_0 n_J}<\frac{1}{2B_J}.
\end{equation}
Let us now set $a(m)=j_{1+mn_0n_J}$, where $0\leq m<[cN/{n_0n_J}]:=q_N$.  
Then, using \eqref{BJ bound}, for every $\bar y_0\in A_{N}$ the operator $P_{it}^{\bar y_0,N}$ can be decomposed as 
$$
P_{it}^{\bar y_0,N}=\cA_{1,t}\circ P_{it}^{T^{a(1)}\bar y_0,n_0n_J}\circ\cA_{2,t}\circ P_{it}^{T^{a(2)}\bar y_0,n_0n_J}\circ\dots\circ \cA_{q_N-1,t}P_{it}^{T^{a(q_N-1)}\bar y_0,n_0n_J}\circ \cA_{q_N,t}
$$
where the operators $\cA_{j,t}$ satisfy
$$
\sup_{t\in J}\|\cA_{j,t}\|_W\leq B_J.
$$
Thus, using also \eqref{ThusBy}, we see that for every $\bar y_0\in A_{N,c}$ we have 
$$
\sup_{t\in J}\|P_{it}^{\bar y_0,N}\|_W\leq 2B_J 2^{-a_JN}
$$
where  $a_J=\frac{c}{n_0n_J}>0$.

Using the above estimates together with (\ref{Char01}), the proposition would follow for $A_N=A_{N,c}$ if we show that exists $c>0$ so that 
\begin{equation}\label{Need}
1-\mu^{\ell-1}(A_{N,c})\leq C_1e^{-c_1N}.
\end{equation}
To establish that, we first notice that indicators of cylinder sets are Lipschitz continuous functions. Applying the results from \cite[Section 3]{ChazGO} with the map $T=F\times F^2\times\dots\times F^{\ell-1}$ we obtain that 
\[
\mu^{\ell-1}\left\{\bar y_0: \sum_{j=0}^{N-1}\bbI(T^j\bar y_0\in \mathscr C_{\ell-1})\leq \frac 12\mu^{\ell-1}(\mathscr C_{\ell-1})N\right\}\leq C_1e^{-c_1 N}
\]
for some $C_1,c_1>0$ which depend only on $\ell$ and $\mathscr C$. Thus (\ref{Need}) holds true with $c=c_J=\frac 12\mu^{\ell-1}(\mathscr C_{\ell-1})$.
\end{proof}


\section{Applications to partially hyperbolic maps}\label{HYPER}
In this section we  consider the hyperbolic maps $f$ from \cite{Y1}. To increase readability we only list the abstract properties of such maps. Moreover, in order not to overload the paper we will not explicitly formulate results, and instead we will explain how to derive  the LCLT for sums of the form 
$$
\sum_{n=1}^{N}G(f^{-n}X_0,f^{-2n}X_0,...,f^{-\ell n}X_0)
$$
in a way similar to \cite[Section 2.11.5]{book}.

Our abstract description of the maps $f$ is as follows.
Let $M$ be a Riemannian manifold
with finite volume and let $f:M\to M$ be a $C^{1+\ve}$ diffeomorphism. Let us denote by $\nu_0$ the volume measure on $M$.
 We assume that there is a set $\Gamma\subset M$ with hyperbolic structure (see \cite{Y1}), and at most countable partition $\{\Gamma_i\}$ of $\Gamma$ (up to measure $0$) so that for each $i$ there is a return time $R_i\in\bbN$ such that $f^{R_i}(\Gamma_i)\subset \Gamma$. Moreover, the sets $\Gamma_i$ are $s$-subsets (see again \cite{Y1}) and if we denote by $\gamma^s(x)$ and $\gamma^u(x)$ the stable and unstable folliations on $\Gamma$ passing through $x\in\Gamma$ then 
 $f^{R_i}(\gamma^s(x))\subset \gamma^s(f^{R_i}x)$ and  $\gamma^u(f^{R_i}x)\subset f^{R_i}(\gamma^u(x))$. In particular, $f^{R_i}(\Gamma_i)$ is a $u$-subset. 
 \begin{assumption}
 We have $\gcd\{R_i\}=1$. Moreover,
 there are $p>0$ and $q>0$ so that for every $n\geq1$,
 $$
 \nu_0(\{x\in\Lambda: R(x)>n\})\leq qe^{-pn}
 $$
 where $R:\Lambda\to\bbN$ is given by $R|\Lambda_i=R_i$.
 \end{assumption}

 Let us define a tower $\tilde \Del$ by setting its $k$-th floor $\tilde \Del_{k}$ to be the set of pairs $(x,k)$ with $x\in\Gamma_i$ and $R_i>k$. In particular $\tilde\Del_0$ is a copy of $\Gamma$. The corresponding tower map $\tilde F:\tilde\Del\to\tilde\Del$ is  defined similarly to Section \ref{YTsec} with $f_0=f^R$ (namely $f_0|\Gamma_i\times\{0\}=(f^{R_i}(\cdot),0)$). Let us denote by $\tilde d_U$ the uniform metric on $\tilde\Del$. 
 \begin{assumption}
 There is a $\beta\in(0,1)$ for which the map $\tilde\pi:\tilde \Del\to M$ given by $(x,k)\to f^{k}x$  is H\"older continuous with respect to the uniform metric determined by $\be$ and the Riemannian metric on $M$.
 \end{assumption}

 Let $\bar\Lambda$ be the quotient space generated by $\Lambda$ and the equivalence relation 
 $$x\equiv y\,\Leftrightarrow y\in \gamma^s(x)$$
 and let $(\Del,F)$ be  the tower map defined by this relation with the base $\Del_0=\bar\Lambda\times\{0\}$ and $f_0=\bar f^{R_i}$ (the map induced on the quotient space). 
 Let $\pi:\tilde\Del\to\Del$ be the projection map given by $\pi(x,k)=(\bar x,k)$, where $\bar x$ is the equivalence class of $x$. Let $d_U$ be the uniform metric in $\Del$ determined by the above $\be$. 
\begin{assumption}\label{JackAss}
The tower $(\Del,F, m)$ satisfies \eqref{Jack Reg1}, where $m$ is  the volume measure on the quotient space $\bar \Gamma$. 
\end{assumption} 
 \begin{assumption}\label{AssDiam}
The projection map $\pi:\tilde\Del\to\Del$ mapping $x$ to its equivalence class $\bar x$ is H\"older continuous with respect to $\tilde d_U$ and $d_U$. In fact, we have the following exponential approximation: there are constants $\del\in(0,1)$ and $C>0$ so that for any cylinder $\cM_{2k}$ of length $2k$ in $\tilde\Del$ we have
$$
\text{diam}_{\Del}\left(\pi(\tilde F^k(\cM_{2k}))\right)\leq C\del^k.
$$
 \end{assumption} 
 
In \cite[Theorem]{Y1} it was shown that there exists an  $f$-invariant SRB measure $\mu_{M}$ with exponential decay of correlations for bounded H\"older continuous observables. This measure has the form $\mu_M=\tilde\pi_*\tilde\mu$ for some invariant measure $\tilde\mu$ on $\tilde\Del$. Moreover, the measure $\mu:=\pi_*\tilde\mu$ is the absolutely continuous invariant measure from Section \ref{YTsec} (see the beginning of \cite[Section 4]{Y1}).
 Let $G:M^\ell\to\bbR$ be a bounded H\"older continuous function and let $X_0$ be an $M$-valued random variable whose distribution is $\mu_M$. In what follows we will explain how to prove the LCLT for $Z_N=S_NG$ given by  
 $$
 S_N G=\sum_{n=1}^{N}G(f^{-n}X_0,f^{-2n}X_0,...,f^{-\ell n}X_0).
 $$
The first step is to observe that
$$
S_NG\overset{d}{=}G(\tilde \pi \tilde F^{N\ell-n}\tilde X_0,\tilde\pi \tilde F^{N\ell-2n}\tilde X_0,...,\tilde\pi \tilde F^{N\ell-\ell n}\tilde X_0)
$$ 
where $\tilde X_0$ is distributed according to $\tilde\mu$, and $\overset{d}{=}$ stands for equality in distribution.

The second step is the following result, which is proved essentially in the same way as \cite[Lemma 1.6]{Bow}, and it corresponds to \cite[Lemma 2.11.2]{book}.

\begin{lemma}\label{SinLemma}
Denote  $F_\ell=F\times F^2\times F^3\times\cdots\times F^{\ell}$.
There exist bounded H\"older continuous functions $\psi:\tilde \Del^\ell\to\bbR$ and $\bar G:\Del^\ell\to\bbR$ so that 
\begin{equation}\label{CobHold}
G\circ\tilde \pi=\bar G\circ\pi+\psi-\psi\circ F_\ell.
\end{equation}
Moreover, if the map $x\to G(x,\cdot)$ is continuous with respect to the H\"older norm then the maps $\bar x\to \bar G(\bar x,\cdot)$ are continuous with respect to the appropriate H\"older norm.
 Moreover, the limits $D^2$ and $D_\ell^2$ remain unchanged if we replace $\bar G$ with $G$.
\end{lemma}

Using the above lemma and Assumption \ref{AssDiam}, similar arguments to the ones in \cite[Section 2.11.5]{book} show that it is essentially enough to verify Assumptions \ref{DecRate1}, \ref{DecRate2} and \ref{DecRate3} with
$$
Z_N=\sum_{n=0}^{N-1}\bar G(F^{N\ell-n}\bar X_0,...,F^{N\ell-\ell n}\bar X_0)\overset{d}{=}
\sum_{n=0}^{N-1}\bar G(\xi_n,\xi_{2n},...,\xi_{\ell n})
$$
where $\bar X_0$ is distributed according to $\mu$ and $\{\xi_n\}$ is the Markov chain described in Section \ref{MC}. 


\section{Appendix A: complex projective metrics on non-uniform towers}\label{sec tower}

 Let $(\Del,F,d,m_0)$ be a non-uniform Young tower, as described in Section \ref{YTsec} so that there exist constants $p>0$ and $q>0$ such that for every $n\geq1$,
\begin{equation}\label{ExTails1}
m_0\{x:\,R(x)>n\}\leq qe^{-pn}.
\end{equation}
Let the transfer operator $L_0$ be defined\footnote{For notational convenience the operator $P$ from  Section \ref{YTsec} is denoted here by $L_0$, the function $h$ is denoted by $h_0$ and  the measure $m$ by $m_0$.} by 
\[
L_0f(x)=\sum_{y\in F^{-1}\{x\}}JF(y)^{-1}f(y)
\]
where $J_F$ is the Jacobian of $F$ ($L_0$ is the dual of the Koopman operator corresponding to $F$ w.r.t. $m_0$).
Note that on $\Del_k,\,k>0$ we have $L_0f(x,k)=f(x,k-1)$, while on $\Del_0$ the members of the set $F^{-1}\{x\}$ are of the form $y=(y^0,k)$ with $R(y^0)=k+1$, and then $JF(y)=JF^R(y^0,0)$.

Next, for each function $f:\Del\to\bbC$, let  $\|f\|_\infty$ denote its supremum and let $\text{Lip}(f)$ denote the infimum of all possible values $L$ so that for all $k$ and $x,y\in\Del_k$ we have
\[
|f(x)-f(y)|\leq Ld(x,y).
\]
We will say that $f$ is locally Lipschitz continuous if $\|f\|:=\max\{\|f\|_\infty,\text{Lip}(f)\}<\infty$, and let us denote by $\cH$ the Banach spaces of all complex valued functions $f$ so that $\|f\|<\infty$.
We will also assume here that the greatest common divisor of the $R_i$'s equals $1$. In this case, by \cite[Theorem 1]{Y2}, there exists a locally Lipschitz continuous function $h_0$ which is bounded, positive and bounded away from $0$ so that $L_0h_0=h_0$, $m_0(h_0)=1$, the measure $\mu=h_0dm_0$ is $F$-invariant and the measure preserving system $(\Del,\cF_0,\mu,F)$ is mixing. Now, for each $k\geq0$ set $v_k=e^{\frac12k p}$ (where $p$ comes from \eqref{ExTails1}). We view $\{v_k\}$ as a function $v:\Del\to\bbR$ so that $v|\Del_k\equiv v_k$, and  let $m$ be the measure on $\Del$ given by  $dm=vdm_0$ (which is finite in view of \eqref{ExTails1}). We also set $h=\frac{h_0}v$. Following \cite{Viv2}, consider the transfer operator $L$ given by 
\[
Lg=\frac{L_0(gv)}{v}.
\]
Then $Lh=h$ and $L^*m=m$ (since $L_0^*m_0=m_0$), and the space $\cH$ is $L$-invariant. In fact (see \cite[Lemma 1.4]{Viv2} and \cite[Lemma 3.4]{Viv2}), the operator norms $\|L^n\|$ are uniformly bounded in $n$.

Next, let $(\Om,\cF,\bbP,\te)$ be an invertible ergodic measure preserving system, $\ell_0$ be a positive integer and $u_\om:\Del\to\bbR$ be a family of functions (where $\om\in\Om$),  so that $(\om,x)\to u_\om(x)$ is measurable  and
$B_u:=\text{ess-sup}\|u_\om\|<\infty$. For each $\om\in\Om$ and $z\in\bbC$ let the transfer operator $\cL_z^{\om}$ be defined by 
\[
\cL_z^{\om}g=L^{\ell_0}(ge^{zu_\om}).
\]
Then, for each $\om$ and $z\in\bbC$, the space $\cH$ is $\cL_z^{\om}$-invariant (since $e^{zu_\om}$  are members of $\cH$ ).
Since the map $z\to e^{zu_\om}\in\cH$ is analytic, the operators $\cL_z^{\om}$ are analytic  in $z$, when viewed as  maps to the space of continuous linear operators $A:\cH\to\cH$, equipped with the operator norm. For each $\om$, a complex number $z$ and $n\in\bbN$ set 
\[
S_n^\om u=\sum_{j=0}^{n-1}u_{\te^j\om}\circ F^{j\ell_0}
\]
and 
\[
\cL_z^{\om,n}=\cL_z^{\te^{n-1}\om}\circ\cdots\circ\cL_z^{\te\om}\circ \cL_z^{\om}
\]
which satisfy $\cL_z^{\om,n}g=\cL_0^{\om,n}(ge^{zS_n^\om u})=L^{\ell_0 n}(ge^{zS_n^\om u})$. 
Henceforth, we will refer to $q$ and $p$ from \eqref{ExTails}, $C$ from (\ref{Jack Reg}) and  $B_u$  as the ``initial parameters".


Next, for any $\ve_0>0$ and $s\geq1$ we can partition $\Del$ into a finite number of disjoint sets $P_2$ and $P',\,P'\in\cP_1$ so that $m(P_2)<\ve_0$ and the diameter each one of the $P'$'s is less than $\gam_s$, where $\gam_s\to 0$ when $s\to\infty$. One way to construct such partitions is as in \cite{Viv2}, and another way is to take a finite collection $\Gam_s$ of the $\Del_\ell^j$'s so that the set 
\[
P_2=\bigcup_{i=0}^s\big(F^R\big)^{-i}\bigcup_{\Del_\ell^j\not\in\Gam_s}\Del_\ell^j
\]
satisfies $m(P_2)<\ve_0$. Denote the above partition by $\cP$. Note that since $\cP$ is finite, then by applying \cite[Theorem 1.2]{Viv2} we deduce that for every $0<\al<1<\al'$ there exists $q_0$ so that for all $k\geq q_0$ and $P,P'\in\cP$ we have 
\[
\al<\frac{m(P\cap F^{-k}P')}{m(P)\mu(P')}<\al'.
\]

Following \cite{Viv2}, for every $a,b,c>0$ let the real cone $\cC_{a,b,c,\ve_0,s}$ consist of all the real-valued  locally Lipschitz continuous functions $f$ so that:
\begin{itemize}
\item
$0\leq\frac{1}{\mu(P)}\int_P fdm=\frac{1}{\mu(P)}\int_P (f/h) d\mu\leq a\int fdm;\,\,\forall\,P\in\cP$.
\\
\item
$\text{Lip}(f)\leq b\int fdm$.
\\
\item
$|f(x)|\leq c\int fdm,\,\,\text{for any }\,x\in P_2$.
\end{itemize}
If $f\in\cC_\bbR$ then for every $x\in\Del\setminus P_2$,
\[
|f(x)|\leq\frac{1}{m(P_1(x))}\int_{P_1(x)} fdm+\gam_s\text{Lip}(f)\leq (a\|h\|_\infty+b\gam_s)\int fdm
\]
where $P_1(x)\in\cP_1$ is the partition element containing $x$, and we have used that $\mu=hdm$. Therefore, with 
\[
c_1=c_1(s,a,b)=a\|h\|_\infty+b\gam_s
\]
and $c_2=\max\{c,c_1\}$ we have 
\begin{equation}\label{f bound}
\|f\|_\infty\leq c_2\int fdm.
\end{equation}
This essentially means that we could have just required that the third condition holds true for all $x\in\Del$, and not only on $P_2$ (by taking $c>c_1$).
Note that if $\int L^{k\ell_0}fdm=0$ for some $k$ and
$f\in\cC_{a,b,c,\ve_0,s}$ then, since 
\[
\int L^{k\ell_0}fdm=\int fdm=0
\]
it follows from (\ref{f bound}) that $f=0$. This means that if, for some $k$, the cone $\cC_{a,b,c,\ve_0,s}$ is $\cL_0^{\om,k}$-invariant then $\cL_0^{\om,k}$ is strictly positive with respect to this cone (recall that $\cL_0^{\om,k}=L^{k\ell_0}$)

The following result was (essentially) proven in \cite{Viv2}:
\begin{theorem}\label{Viv2Thm}
For every  $\sig\in(0,1)$ small enough and positive numbers $a_0,b_0,c_0$ there is a positive integer $k_0$, positive numbers $a\geq a_0,b\geq b_0$ and $c\geq c_0$ and $\ve_0>0$ and $s\geq1$ so that  with $\cC_\bbR=\cC_{a,b,c,\ve_0,s}$, for every $k\geq k_0$ we have 
\[
L^k\cC\subset\cC_{\sig a,\sig b,\sig c,\ve_0,s}
\]
and for all $f,g\in\cC_\bbR$,
\[
d_{\cC_\bbR}(L^kf,L^kg)\leq d_0<\infty
\]
where $d_{\cC_\bbR}$ is the real Hilbert (projective) metric corresponding to the cone $\cC_\bbR$ and $d_0$ is some constant.
\end{theorem}
Let us denote by $\cC=\cC_{a,b,c,\ve_0,s}$ the canonical complexification of the real cone $\cC_\bbR$ from Theorem \ref{Viv2Thm} (we refer to \cite[Appendix A]{book} for all the relevant definitions regarding real and complex cones).
The main result in this section is the following:

\begin{theorem}\label{YT cones thm}
For all sufficiently large $a,b$ and $c$ we have:

(i) The cone $\cC$ is linearly convex, it contains the functions $h$ and $\textbf{1}$ (the function which takes the constant value $1$). Moreover,  the measure $m$, when viewed as a linear functional, is a member of the dual cone $\cC_\bbR^*$ and
the cone  $\cC$ and its dual $\cC^*$ have bounded aperture. In fact,
there exist constants $K,M>0$ so that for every $f\in\cC$ and $\mu\in\cC^*$, 
\begin{equation}\label{aperture0}
\|f\|\leq K|m(f)|
\end{equation}
and
\begin{equation}\label{aperture}
\|\mu\|\leq M|\mu(h)|.
\end{equation}

(ii) The cone $\cC$ is reproducing. In fact, there exists a constant $K_1$ so that for every $f\in\cH$ there is $R(f)\in\bbC$ so that $|R(f)|\leq K_1\|f\|$
and 
\[
f+R(f)h\in\cC.
\]

(iii) There exist  constants $r>0$ and $d_1>0$ so that for P-almost every $\om$, a complex number $z\in B(0,r)$ and $k_0\leq k\leq 2k_0$, where $k_0$ comes from Theorem \ref{Viv2Thm}, we have 
\[
\cL_z^{\om,k}\cC'\subset\cC'
\]
and 
\[
\sup_{f,g\in\cC'}\del_{\cC}(\cL_z^{\om,k}f,\cL_z^{\om,k}g)\leq d_1
\]
where $\cC'=\cC\setminus\{0\}$ and $\del_{\cC}$ is the complex (projective) Hilbert metric corresponding to $\cC$ (see \cite[Appendix A]{book} for the definition of this metric as well as for the definitions of real and complex dual cones).
\end{theorem}
Once this theorem is obtained the random complex Ruelle-Perron-Frobenius theorem for the operators $\cL_z{\om}$ follows from \cite[Theorems 4.2.1, 4.2.2]{book}. This theorem essentially means that Theorem \ref{RPF} also holds for the more general operators  $\cL_z^\om$.

\subsection{Proof of Theorem \ref{YT cones thm}}\label{Sec 4.1}
(i) We begin with the proof of the first item. First, since 
\[
\int_A hdm=\int Ad\mu=\mu(A)
\]
for any measurable  set $A$, 
it is clear that $h\in\cC_\bbR$ if $a>1$, $b>\text{Lip}(h)$ and $c>\|h\|_\infty$. Moreover, if $c>1$ and $a>D$, where
\begin{equation}\label{D def}
D=\max\Big\{\frac{m(P)}{\mu(P)}:\,P\in\cP\Big\}
\end{equation}
then $\textbf{1}\in\bbC_\bbR$.

Next, if $f\in\bbC_\bbR'$ and $m(f)=0$ then by (\ref{f bound}) we have $f=0$ and so $m\in\cC_\bbR^*=\{\mu\in\cH^*: \mu|\cC_\bbR'>0\}$ (since $m\geq0$ on $\cC_\bbR$). In fact, it follows
from the definition of the norm $\|f\|$ and from (\ref{f bound}) that 
\[
\|f\|\leq\|f\|_\infty+\text{Lip}(f)\leq (c_2+b)m(f)=(c_2+b)\int fdm
\]
and therefore by \cite[Lemma 5.2]{Rug} the inequality (\ref{aperture0}) holds true with $K=2\sqrt 2(c_2+b)$. According to \cite[Lemma A.2.7]{book} (appearing in the appendix there), for every $M>0$,
inequality (\ref{aperture}) holds true for all $\mu\in\cC^*=\{\mu\in\cH^*: \mu(\cC_\bbR')\subset\bbC'\}$ if
\begin{equation}\label{Const M}
\left\{x\in X: \|x-h\|<\frac1M\right\}\subset\cC.
\end{equation}
Now we will show how to find a constant $M$ for which (\ref{Const M}) holds true.
For any $f\in\cH$, $P\in\cP$ and $x_1\in P_2$, and distinct $x,y$ which belong to the same level  $\Del_\ell$ (for some $\ell$) set
\begin{eqnarray*}
\Upsilon_P(f)=\frac1{\mu(P)}\int_P fdm,\,\,
\Gam_P(f)=a\int fdm-\frac1{\mu(P)}\int_P fdm,\\
\Gam_{x,y}(f)=b\int fdm-\frac{f(x)-f(y)}{d(x,y)}
\, \text{ and }\,
\Gam_{x_1,\pm}(f)=c\int fdm\pm f(x_1).
\end{eqnarray*}
Let $\cS$ be the collection of all the above linear functionals. Then 
\[
\cC_\bbR=\{f\in\cH:\,s(f)\geq0,\,\forall s\in\cS\}
\]
and so, by the definition of the canonical complexification of a real cone (see \cite{Rug, book}), we have
\begin{equation}\label{Complexification1}
\cC_\bbC=\{f\in \cH:\,\Re\big(\overline{\nu_1(f)}\nu_2(f)\big)
\geq0,\,\,\,\,\forall\nu_1,\nu_2\in\cS\}.
\end{equation}
Let $g\in\cH$ be of the form $g=h+q$ for some $q\in\cH$. We need to find a constant $M>0$ so that $h+q\in \cC$ if $\|q\|<\frac1M$. In view of (\ref{Complexification1}), there are several cases to consider. First, suppose that $\nu_1=\Upsilon_{P}$ and $\nu_2=\Upsilon_Q$ for some $P,Q\in\cP$. Since
\[
\frac{1}{\mu(A)}\int hdm=\frac{1}{\mu(A)}\int 1d\mu=1
\]
for any measurable set $A$ with positive measure, we have 
\[
\Re\big(\overline{\nu_1(h+q)}\nu_2(h+q)\big)\geq 1-(D^2\|q\|^2+2D\|q\|)
\]
where $D$ was defined in (\ref{D def}).
Hence
\[
\Re\big(\overline{\nu_1(h+q)}\nu_2(h+q)\big)>0
\]
if $\|q\|$ is sufficiently small. Now consider the case when $\nu_1=\Upsilon_P$ for some $P\in\cP$ and
$\nu_2$ is one of the $\Gamma$'s, say $\nu=\Gam_{x,y}$. Then
\begin{eqnarray*}
\Re\big(\overline{\nu_1(h+q)}\nu_2(h+q)\big)\geq b-\|h\|-bm(\textbf{1})\|q\|-\|q\|\\-D\|q\|(b+\|h\|+bm(\textbf{1})\|q\|+\|q\|)
\geq b-\|h\|-C(D,b)(\|h\|+\|q\|+\|q\|)^2 
\end{eqnarray*}
where $C(D,b)>0$ depends only on $D$ and $b$. If $\|q\|$ is sufficiently small and $b>\|h\|$ then 
the above left hand side is clearly positive. Similarly, if $\|h\|<\min\{a,b,c\}$ and $\|q\|$ is sufficiently small then 
\[
\Re\big(\overline{\mu(h+q)}\nu(h+q)\big)>0
\]
when either $\nu_2=\Gam_{x_1,\pm}$ or $\nu_2=\Gam_{x,y}$.

Next, consider the case when $\nu_1=\Gam_{x_1,\pm}$ for some $x_1\in P_2$ and  $\nu_2=\Gam_{x,y}$ for some distinct $x$ and $y$ in the same floor. Then
\begin{eqnarray*}
\Re\big(\overline{\nu_1(h+q)}\nu_2(h+q)\big)\geq 
\left(c-\|h\|-cm(\textbf{1})\|q\|-\|q\|\right)\cdot\left(b-\|h\|-bm(\textbf{1})\|q\|-\|q\|\right)
\end{eqnarray*}
where we have used  again that $\int hdm=1$. Therefore, if $\|q\|$ is sufficiently small and $c$ and $b$ are sufficiently large then
\[
\Re\big(\overline{\nu_1(h+q)}\nu_2(h+q)\big)>0.
\] 
Similarly, since 
\[
\left|\frac1{\mu (P)}\int_P qdm\right|\leq D\|q\|
\]
when $a,b,c$ are large enough there are constants $A_1,A_2>0$ so that 
for any other choice of $\mu,\nu\in\cS\setminus\{\Upsilon_P\}$
we have
\begin{eqnarray*}
\Re\big(\overline{\nu_1(h+q)}\nu_2(h+q)\big)\geq A_1(1-A_2(\|q\|+\|q\|^2))
\end{eqnarray*}
and so, when $\|q\|$ is sufficiently small then the above left hand side is positive. The proof of Theorem \ref{YT cones thm} (i) is now complete. 
\vskip0.1cm
(ii) The proof of Theorem \ref{YT cones thm} (ii) proceeds exactly as the proof of  \cite[Lemma 3.11]{Viv2}: for a real-valued function $f\in\cH$, it is clearly enough to take any $R(f)>0$ so that
\begin{eqnarray*}
R(f)>(a-1)^{-1}\cdot\max\Big\{\frac1{\mu(P)}\int_P fdm-a\int fdm:\,\,P\in\cP\Big\},\\
R(f)>\frac{\text{Lip}(f)-b\int fdm}{b-\text{Lip}(h)},\,\,R(f)>\max\Big\{-\frac1{\mu(P)}\int_P fdm:\,\,P\in\cP\Big\}\,\,\text{ and }\\
R(f)>\frac{c\int fdm-\|f\|_\infty}{c-\|h\|_\infty}
\end{eqnarray*}
where we take $a,b$ and $c$ so that all the denominators appearing in the above inequalities  are positive,
and we have used that $\frac{1}{\mu(A)}\int_A hdm=1$ for any measurable set $A$ (apply this with $A=P\in\cP$).
For complex valued $f$'s we can write $f=f_1+if_2$, then take $R(f)=R(f_1)+iR(f_2)$ and use that with $\bbC'=\bbC\setminus\{0\}$, 
\[
\cC=\bbC'(\cC_\bbR+i\cC_\bbR).
\]
We refer to \cite[Appendix A]{book} for references regarding the above polar decomposition of $\cC$.

\vskip0.1cm
(iii) Now we will prove Theorem \ref{YT cones thm} (iii). Let $k_0\leq k\leq 2k_0$, where $k_0$ comes from Theorem \ref{Viv2Thm}. Let  $\ve>0$ be so that 
\[
\del:=2\ve\Big(1+\cosh\big(\frac12 d_0\big)\Big)<1
\]
where $d_0$ comes from Theorem \ref{Viv2Thm}. Then,
according to Theorem A.2.4  in Appendix A of \cite{book} (which is \cite[Theorem 4.5]{Dub2}), if 
\begin{equation}\label{Comp1}
|s(\cL_z^{\om,k} f)-s(\cL_0^{k}f)|\leq \ve s(\cL_0^{j,k}f)
\end{equation}
for all nonzero $f\in\cC_\bbR$ and $s\in\cS$ ($\cS$ was defined before (\ref{Complexification1})), 
then, with $\cC'=\cC\setminus\{0\}$,
\begin{equation}\label{I}
\cL_z^{\om,k}\cC'\subset\cC'
\end{equation}
and
\begin{equation}\label{II} 
\sup_{f,g\in\cC}(\cL_z^{\om,k}f,\cL_z^{\om,k}g)\leq d_0+6|\ln(1-\del)|.
\end{equation}
We will show now that there exists a constant $r>0$ so that (\ref{Comp1}) holds true for every $z\in B(0,r)$ and $f\in\cC_\bbR$. We first need the following very elementary result, which for the sake of convenience is formulated here as a lemma.
\begin{lemma}\label{A A' lemma}
Let $A$ and $A'$ be complex numbers, $B$ and $B'$ be real numbers, and let $\ve_1>0$ and $\sig\in(0,1)$ so that
\begin{itemize}
\item
$B>B'$
\item
$|A-B|\leq\ve_1B$
\item
$|A'-B'|\leq\ve_1 B $
\item
$|B'/B|\leq\sigma$.
\end{itemize}
Then 
\[
\left|\frac{A-A'}{B-B'}-1\right|\leq 2\ve_1(1-\sig)^{-1}.
\]
\end{lemma}
The proof of Lemma \ref{A A' lemma} is very simple, just write
\[
\left|\frac{A-A'}{B-B'}-1\right|\leq\left|\frac{A-B}{B-B'}\right|+
\left|\frac{A'-B'}{B-B'}\right|\leq \frac{2B\ve_1}{B-B'}=\frac{2\ve_1}{1-B'/B}.
\]
Next, let $f\in\cC_\bbR'$. First, suppose that $s$ has the form 
$s=\Gam_{P}$ for some $P\in\cP$. Set
\begin{eqnarray*}
A=a\int \cL_z^{\om,k}fdm,\,\, A'=\frac1{\mu(P)}\int_P \cL_z^{\om,k}fdm,\\
B=a\int\cL_0^{\om,k}fdm\,\,\text{ and }\,\,B'=\frac1{\mu(P)}\int_P \cL_0^{\om,k}fdm.
\end{eqnarray*}
Then $B=a\int fdm$ (since $m$ is conformal) and
\[
|s(\cL_z^{\om,k})-s(\cL_0^{\om,k})|=|A-A'-(B-B')|.
\]
We want to show that the conditions of Lemma \ref{A A' lemma} hold true. 
By Theorem \ref{Viv2Thm} we have 
\begin{equation}\label{Smaller cone}
\cL_0^{\om,k}f\in\cC_{\sig a,\sig b,\sig c,s,\ve_0}
\end{equation}
which in particular implies that 
\[
0\leq B'\leq \sig a\int\cL_0^{\om,k}fdm=\sig B.
\]
Since $f$ is nonzero and $\int\cL_0^{\om,k}fdm=\int fdm\geq 0$ the number $B$ is positive 
(since (\ref{aperture0}) holds true). It follows that $B>B'$ and that 
\[
|B'/B|\leq\sig<1. 
\]
Now we will estimate $|A-B|$. Let us fix some complex number $z$ so that $|z|\leq1$. Then
\begin{eqnarray*}
|A-B|=a\left|\int L^{k\ell_0}\big(f(e^{zS_k^\om u}-1)\big)dm\right|\leq a\|f\|_\infty\|e^{zS_k^\om u}-1\|_\infty
\int L^{k\ell_0}\textbf{1}dm\\= a\|f\|_\infty\|e^{zS_k^\om u}-1\|_\infty\int \textbf{1}dm=
m(\textbf{1})a\|f\|_\infty\|e^{zS_k^\om u}-1\|_\infty\\
\leq ac_2m(\textbf{1})\int fdm\,\cdot(2k_0 e^{2k_0\|u\|_\infty}\cdot|z|\|u\|_\infty)\\=2am(\textbf{1})c_2k_0\|u\|_\infty|z|\int L^{k\ell_0}f dm=R_1|z|B
\end{eqnarray*}
where $\textbf{1}$ is the function which takes the constant value $1$, 
\[
\|u\|_\infty=\text{ess-sup}\|u_\om\|_\infty
\] 
and 
\[
R_1=2c_2k_0e^{k_0\|u\|_\infty}m(\textbf{1})\|u\|_\infty.
\]
In the latter estimates we have also used (\ref{f bound}).
It follows  that in the second condition of Lemma \ref{A A' lemma} we can take $\ve\leq R_1|z|$. Now we will estimate $|A'-B'|$. 
First, we have
\begin{eqnarray*}
|A'-B'|\leq\frac1{\mu(P)}\int_P\big|\cL_z^{\om,k}f-\cL_0^{\om,k}f\big|dm=
\frac1{\mu(P)}\int_P\big|L^{k\ell_0}\big(f(e^{zS_k^\om u}-1)\big)|dm\\
\leq\|f\|_\infty \|e^{zS_k^\om u}-1\|_\infty\frac1{\mu(P)}\int_P L^{k\ell_0}\textbf{1}dm
\leq M_1\|f\|_\infty \|e^{zS_k^\om u}-1\|_\infty\frac{m(P)}{\mu(P)}\\\leq
 Dc_2\int fdm\,\cdot 2k_0\|u\|_\infty e^{2k_0\|u\|_\infty}|z|
=R_2|z|B
\end{eqnarray*}
where $D$ was defined in (\ref{D def}), $M_1$ is an upper bound on $\|L^{k\ell_0}\textbf{1}\|_\infty$ for $k_0\leq k\leq 2k_0$ (in fact, we can use \cite[Lemma 1.4]{Viv2} and obtain an upper bound which does not depend on $k_0$)
 and
\[
R_2= Da^{-1}2c_2k_0\|u\|_\infty e^{2k_0\|u\|_\infty}.
\]
We conclude now from Lemma \ref{A A' lemma} that 
\[
|s(\cL_z^{j,k})-s(\cL_0^{j,k})|\leq 2R_3(1-\sig)^{-1}|z|s(\cL_0^{\om,k})
\]
where $R_3=\max(R_1,R_2)$.

Next, consider the case when $s$ has the form $s=\Gam_{x,\pm}$ for some $x\in P_2$. Set
\begin{eqnarray*}
A=c\int \cL_z^{\om,k}fdm,\,\, A'=\pm \cL_z^{\om,k}f(x),\\
B=c\int\cL_0^{\om,k}fdm\,\,\text{ and }\,\,B'=\pm\cL_0^{\om,k}f(x).
\end{eqnarray*}
Then $B>0$ and by (\ref{Smaller cone}) we have 
\[
|B'|\leq \sig B.
\]
Similarly to the previous case, we have 
\[
|A-B|\leq R_4B|z|
\]
where $R_4=2c_2k_0\|u\|_\infty$. Now we will estimate $|A'-B'|$. Using (\ref{f bound}) 
we have
\begin{eqnarray*}
|A'-B'|=|\cL_z^{\om,k}f(x)-\cL_0^{\om,k}f(x)|\leq \|f\|_\infty\|e^{zS_k^\om u}-1\|_\infty\cL_0^{\om,k}\textbf{1}(x)\\
\\\leq c_2\int fdm\,\cdot(2k_0|z|\|u\|_\infty e^{2k_0\|u\|_\infty} M_1)=BR_5|z| 
\end{eqnarray*}
where $R_5=2c_2k_0\|u\|_\infty M_1$ and
$M_1$ is an upper bound on $\|L^{k\ell_0}\textbf{1}\|_\infty$ for $k_0\leq k\leq 2k_0$.
Since 
\[
|s(\cL_z^{\om,k})-s(\cL_0^{\om,k})|=|A-A'-(B-B')|,
\]
we conclude from Lemma \ref{A A' lemma} that 
\[
|s(\cL_z^{\om,k})-s(\cL_0^{\om,k})|\leq 2R_6(1-\sig)^{-1}|z|s(\cL_0^{\om,k})
\]
where $R_6=\max\{R_4,R_5\}$. 

Finally,  consider the case when $s=\Gam_{x,x'}$ for some distinct $x'$ and $x'$ which belong to the same floor of $\Del$. 
Set 
\begin{eqnarray*}
A=b\int \cL_z^{\om,k}fdm,\,\, A'=\frac{\cL_z^{\om,k}f(x)-\cL_z^{\om,k}f(x')}{d(x,x')},\\
B=b\int\cL_0^{\om,k}fdm\,\,\text{ and }\,\,B'=\frac{\cL_0^{\om,k}f(x)-\cL_0^{\om,k}f(x')}{d(x,x')}.
\end{eqnarray*}
Then, exactly as in the previous cases, $B>0$,  $|B'|\leq \sig B$,
\[
|s(\cL_z^{\om,k})-s(\cL_0^{\om,k})|=|A-A'-(B-B')|
\]
and 
\[
|A-B|\leq R_7B|z|
\]
where $R_7=2c_2k_0b^{-1}\|u\|_\infty$.  Now we will estimate $|A'-B'|$.
Let $\ell$ be so that $x,x'\in\Del_\ell$ and write
$x=(x_0,\ell)$ and $x'=(x_0',\ell)$. Then $d(x,x')=d((x_0,m),(x_0',m))$ for every  $0\leq m\leq\ell$.
If $k\ell_0\leq \ell$ then for every $z\in\bbC$,
\[
\cL_z^{\om,k}f(x)=v_\ell^{-1}v_{\ell-k\ell_0}e^{zS_k^\om u(x_0,\ell-k\ell_0)}f(x_0,\ell-k\ell_0)
\]
and a similar equality holds true with $x'$ in place of $x$.
Set 
\begin{eqnarray*}
U(z)=f(x_0,\ell-k\ell_0)e^{zS_k^\om u(x_0,\ell-k\ell_0)} \,\text{ and }\,  V(z)=f(x_0',\ell-k\ell_0)e^{zS_k^\om u(x_0',\ell-k\ell_0)}
\end{eqnarray*}
and $W(z)=U(z)-V(z)$.
Then for every $z\in\bbC$ so that $|z|\leq1$ we have
\[
d(x,x')|A'-B'|=v_\ell^{-1}v_{\ell-k\ell_0}|W(z)-W(0)|\leq |z|\sup_{|\zeta|\leq1}|W'(\zeta)|.
\]
Since the functions $u_\om$ and $f$ are locally Lipschitz continuous (uniformly in $\om$) we obtain that for all complex $\zeta$ so that $|\zeta|\leq1$,
\[
|W'(\zeta)|\leq C_1d(x,x')\|f\| \leq d(x,x')C_1(b+c_2)\int fdm=d(x,x')C_1b^{-1}(b+c_2)B
\]
where $C_1$ depends only on $k_0$ and $B_u=\text{ess-sup}\|u_\om\|$.

Next, suppose that $k\ell_0>\ell$, where $\ell$ is such that $x,x'\in\Del_\ell$.
The approximation of  $|A'-B'|$ in this case relies on classical arguments from the theory of distance expanding map.
Since $k\ell_0>\ell$  we can write 
\[
F^{-k\ell_0}\{x\}=\{y\},\,\,F^{-k\ell_0}\{x'\}=\{y'\}
\] 
where both sets are at most countable, the map $y\to y'$ is bijective and satisfies that for every $0\leq q\leq k\ell_0$,
\[
d(F^qy,F^qy')\leq \beta^{m_q(y)} d(x,x')\leq d(x,x').
\]
Here $m_q(y)$ is the number of the points among $F^{q+m}y$, $0\leq m\leq k-q$ which belong
to the base $\Del_0$ (so $m_0(y)\geq1$, since $\ell<k$). Note also that the pairs $(y,y')$ also belong to the same partition element $\Del_\ell^j$.
Using these notation, for every $z\in\bbC$ we can write
\[
\cL_z^{\om,k}f(x)=v_\ell^{-1}\sum_{y}v(y)JF^{k\ell_0}(y)^{-1}e^{zS_k^\om u(y)}f(y)
\] 
and 
\[
\cL_z^{\om,k}f(x')=v_\ell^{-1}\sum_{y}v(y)JF^{k\ell_0}(y')^{-1}e^{zS_k^\om u(y')}f(y')
\] 
where we note that $v(y)=v(y')$ since $y$ and $y'$ belong to the same floor. For every $y$ set 
\[
U_{y}(z)=JF^{k\ell_0}(y)^{-1}e^{zS_k^\om u(y)}f(y)
\]
and 
\[
W_{y,y'}(z)=U_{y}(z)-U_{y'}(z).
\]
Then for every complex $z$ so that $|z|\leq1$ we have
\[
|W_{y,y'}(z)-W_{y,y'}(0)|\leq|z|\sup_{|\zeta|\leq 1}|W'_{y,y'}(\zeta)|.
\]
Since $JF^{R}$ satisfies (\ref{Jack Reg}) and $u_\om$ and $f$ are locally Lipschitz continuous (uniformly in $\om$) we 
derive that 
\begin{equation}\label{Der Bound}
\sup_{|\zeta|\leq 1}|W'_{y,y'}(\zeta)|\leq C_2\|f\|d(x,x')(JF^{k\ell_0}(y)^{-1}+JF^{k\ell_0}(y')^{-1})
\end{equation}
for some constant $C_2$ which depends only on $B_u,k_0$ and $C$ from (\ref{Jack Reg}).
Using that 
\[
\|f\|\leq(c_2+b)\int fdm
\]
we derive now from (\ref{Der Bound})  that 
\begin{eqnarray*}
d(x,x')|A'-B'|=v_\ell^{-1}\left|\sum_{y}v(y)\big(W_{y,y'}(z)-W_{y,y'}(0)\big)\right|
\\\leq \big(|z|d(x,x')C_2\|f\|\big)v_\ell^{-1}\sum_{y}v(y)(JF^{k\ell_0}(y)^{-1}+JF^{k\ell_0}(y')^{-1})\\=
\big(|z|d(x,x')C_2\|f\|\big)\cdot\big(L^{k\ell_0}\textbf{1}(x)+L^{k\ell_0}\textbf{1}(x')\big)\leq E_1|z| B
\end{eqnarray*}
where $E_1=2M_1C_2b^{-1}(c_2+b)$ and $M_1=\sup_n\|L^n\|_\infty$, which is finite in view of \cite[Lemma 1.4]{Viv2}.
We conclude that there exists a constant $C_0$ so that for every $s\in\cS$, $f\in\cC'$, $z\in\bbC$ and $k_0\leq k\leq 2k_0$,
\[
|s(\cL_z^{\om,k})-s(\cL_0^{\om,k})|\leq C_0|z|s(\cL_0^{\om,k}).
\]
Let $r>0$ be a positive number so that 
\[
\del_r:=2C_0r\Big(1+\cosh\big(\frac12 d_0\big)\Big)<1.
\]
Then, by (\ref{Comp1}) and what proceeds it, (\ref{I}) and (\ref{II}) hold true
for every  $z\in\bbC$ with $|z|<r$, $\om\in\Om$ and $k_0\leq k\leq 2k_0$, and the proof of Theorem \ref{YT cones thm} is complete.
\qed

\section{Appendix B: A Lasota-Yorke inequality for random transfer operators  and Quasi-Compactness of deterministic ones}\label{AppB}

The following result is proved for the transfer operators $P_{it}^{\bar y_0,n},\,\bar y_0\in\Del^{\ell-1}$ defined in \eqref{P TWO}, exactly as \cite[Proposition 2.2.1]{RandTower} (taking into account that  $v_k=e^{kp/2}$).
\begin{proposition}\label{Prop LY}
(i) For every $N$ and $k$ so that $N\leq k$, a function $g:\Del\to\bbC$, $\bar y_0\in \Del^{\ell-1}$ and $x,y\in\Del$ we have 
\begin{equation}\label{LY1.1}
|P_{it}^{\bar y_0,N}g(x)|\leq e^{(k-N)p/2}\|g\|_s
\end{equation}
and 
\begin{equation}\label{LY1.2}
|P_{it}^{\bar y_0,N}g(x)-P_{it}^{\bar y_0,N}g(y)|\leq (\|g\|_h\beta^N+(A|t|+2\be^{-1})\|g\|_s)e^{(k-N)p/2}d_{U}(x,y)
\end{equation}
where $A=(1-\be)^{-1}\sup_{a}\sup_{s}|u_{a}|_{\be,\Del_s}$ (recall $u_a=G_\ell(a,\cdot)$).

(ii) For every $N$ and $k$ so that $N>k$, a function $g:\Del\to\bbC$, $\bar y_0\in \Del^{\ell-1}$ and $x,y\in\Del$ we have 
\begin{equation}\label{LY2.1}
|P_{it}^{\bar y_0,N}g(x)| \leq Q\left(\int |g|d m+\be^N\|g\|_h\cdot C_2\right):=R_{N}(g)
\end{equation}
and 
\begin{equation}\label{LY2.2}
|P_{it}^{\bar y_0,N}g(x)-P_{it}^{\bar y_0,N}g(y)|\leq \left(C_1+2\be^{-1}+|t|A\right)R_N(g)d_{U}(x,y)
\end{equation}
where $C_2$ and $Q$ are some constants.

In particular 
\begin{eqnarray*}
\|P_{it}^{\bar y_0,N}g\|_{W}\\\leq \max\left(e^{-Np/2}\left((1+|A|t)\|g\|_s+\beta^N\|g\|_h\right), R_N(g)(2+C_1+|t|A)\right).
\end{eqnarray*}
Therefore, for every compact set $J\subset\bbR$ the operator norms $\|P_{it}^{\bar y_0,N}\|_{W}$ with respect to the norm $\|\cdot\|_{W}$ are uniformly bounded in $\bar y_0\in \Del^{\ell-1}, N\geq1$ and $t\in J$.
\end{proposition}

The above proposition holds true for the periodic point $\bar y_0=\bar v_0$, and it yields a deterministic  Lasota-Yorke inequality for the operators $P_{it}^{\bar v_0,n_0}$. Using that,
the quasi-compactness of these operators follow from arguments similar to \cite[Section 3.4]{Y1} (the main key is that only the $L^1(m)$ norm  appears without a factor of the form $\rho^N$ for some $\rho\in(0,1)$).

\end{document}